\documentclass[10pt]{amsart} 

\usepackage[utf8]{inputenc}
\usepackage{color}
\usepackage{amsmath} 
\usepackage{amssymb} 
\usepackage{amsthm}
\usepackage{graphicx}
\usepackage{esint}
\usepackage[colorlinks=true,linkcolor=blue]{hyperref}

\theoremstyle{plain}
\newtheorem{thm}{Theorem}
\newtheorem{prop}{Proposition}[section]
\newtheorem{lem}[prop]{Lemma}
\newtheorem{cor}[prop]{Corollary}

\newtheorem{defi}[prop]{Definition}
\newtheorem{rmk}[prop]{Remark}

\newtheorem{example}[prop]{Example}

\newcommand {\R} {\mathbb{R}} \newcommand {\Z} {\mathbb{Z}}
 \newcommand {\N} {\mathbb{N}}
 
\newcommand {\p} {\partial}
\newcommand {\dt} {\partial_t}

\newcommand {\D} {\Delta}

\newcommand {\supp} {\text{supp}}

\DeclareMathOperator{\vol}{vol}

\DeclareMathOperator {\dist} {dist}
\DeclareMathOperator {\Ree} {Re}
\DeclareMathOperator {\Imm} {Im}

\DeclareMathOperator{\spa} {span}
\DeclareMathOperator{\inte} {int}

\usepackage{xspace}

\newcommand{\propsemicont}{4.3 in \cite{KRS14}\xspace}

\newcommand{\propblowup}{4.5 in \cite{KRS14}\xspace}

\newcommand{\propindep}{4.2 in \cite{KRS14}\xspace}

\newcommand{\corind}{4.2 in \cite{KRS14}\xspace}

\newcommand{\thmtwo}{2 in \cite{KRSI}\xspace}
\newcommand{\propasym}{4.6 in \cite{KRSI}\xspace}
\newcommand{\corasym}{4.8 in \cite{KRSI}\xspace}

\title{The Variable Coefficient Thin Obstacle Problem: Higher Regularity} 

\author{Herbert Koch}
\author{Angkana R\"uland }
\author{Wenhui Shi}

\address{
Mathematisches Institut, Universit\"at Bonn, Endenicher Allee 60, 53115 Bonn, Germany }
\email{koch@math.uni-bonn.de}

\address{
Mathematical Institute of the University of Oxford, Andrew Wiles Building, Radcliffe Observatory Quarter, Woodstock Road, OX2 6GG Oxford, United Kingdom }
\email{ruland@maths.ox.ac.uk}

\address{
Mathematisches Institut, Universit\"at Bonn, Endenicher Allee 64, 53115 Bonn, Germany  }
\email{wenhui.shi@hcm.uni-bonn.de}

\begin{document}

\begin{abstract}
In this article we continue our investigation of the thin obstacle problem with variable coefficients which was initiated in \cite{KRS14}, \cite{KRSI}. Using a partial Hodograph-Legendre transform and the implicit function theorem, we prove higher order Hölder regularity for the regular free boundary, if the associated coefficients are of the corresponding regularity. For the zero obstacle this yields an improvement of a \emph{full derivative} for the free boundary regularity compared to the regularity of the metric. In the presence of non-zero obstacles or inhomogeneities, we gain \emph{three halves of a derivative} for the free boundary regularity with respect to the regularity of the inhomogeneity. Further we show analyticity of the regular free boundary for analytic metrics. We also discuss the low regularity set-up of $W^{1,p}$ metrics with $p>n+1$ with and without ($L^p$) inhomogeneities.\\ 
Key new ingredients in our analysis are the introduction of generalized Hölder spaces, which allow to interpret the transformed fully nonlinear, degenerate (sub)elliptic equation as a perturbation of the Baouendi-Grushin operator, various uses of intrinsic geometries associated with appropriate operators, the application of the implicit function theorem to deduce (higher) regularity and the splitting technique from \cite{KRSI}.   
\end{abstract}



\thanks{
H.K. acknowledges support by the DFG through SFB 1060.
A.R. acknowledges a Junior Research Fellowship at Christ Church.
W.S. is supported by the Hausdorff Center of Mathematics.}

\maketitle

\tableofcontents

\section{Introduction}

This article is devoted the study of the higher regularity properties of the free boundary of solutions to the \emph{thin obstacle} or \emph{Signorini problem}. To this end, we consider local minimizers to the functional
\begin{align*}
J(v):=\int\limits_{B_1^+}a^{ij}\p_iv\p_jv dx,\quad v\in \mathcal{K},
\end{align*}
with $\mathcal{K}:=\{u\in H^1(B_1^+)| \ u\geq 0 \mbox{ on } B_1'\times \{0\}\}$. Here $B_1^+ := \{x\in B_1 \subset \R^{n+1}| \ x_{n+1}\geq 0\}$ and $B_1':= \{x\in B_1 \subset \R^{n+1}| \ x_{n+1}=0\}$ denote the $(n+1)$-dimensional upper half ball and the co-dimension one ball, respectively. The tensor field $a^{ij}: B_1^+ \rightarrow \R^{(n+1)\times (n+1)}_{sym}$ is assumed to be uniformly elliptic, symmetric and $W^{1,p}(B_1^+)$ regular for some $p> n+1$. Here and in the sequel we use the summation convention.\\

Due to classical results on variational inequalities \cite{U87}, \cite{F10}, minimizers of this problem exist and are unique (under appropriate boundary conditions). Moreover, minimizers are $C^{1,\min\{1-\frac{n+1}{p},\frac{1}{2}\}}(B_{1/2}^+)$ regular (c.f. \cite{AC06}, \cite{KRSI}) and solve the following uniformly elliptic equation with \emph{complementary} or \emph{Signorini boundary conditions} 
\begin{equation}
\label{eq:thin_obst}
\begin{split}
\p_i a^{ij} \p_j w &= 0 \mbox{ in } B_1^+,\\
w \geq 0, \ a^{n+1, j} \p_{j}w\leq 0,\ w (a^{n+1,j}\p_j w) &= 0 \mbox{ in } B_1' \times \{0\}.
\end{split}
\end{equation}
Here the bulk equation is to be interpreted weakly, while the boundary conditions hold pointwise. In particular, the constraint originating from the convex set $\mathcal{K}$ only acts on the boundary; in this sense the obstacle is \emph{thin}. The constraint on functions in $\mathcal{K}$ divides the boundary $B_1' \times \{0\}$ into three different regions: The \emph{contact set} $\Lambda_w := \{x\in B_{1}'\times \{0\}| \ w=0\}$, where the minimizer attains the obstacle, the \emph{non-coincidence set}, $\Omega_w:= \{x\in B_1' \times \{0\}| \ w>0\}$, where the minimizer lies strictly above the obstacle, and the \emph{free boundary}, $\Gamma_w := \partial \Omega_w$, which separates the contact set from the non-coincidence set. \\

As we seek to obtain a more detailed analysis of the (regular) free boundary under higher regularity assumptions on the metric tensor $a^{ij}$, we briefly recall the, for our purposes, most relevant known properties of the free boundary (c.f. \cite{CSS}, \cite{PSU},  \cite{GSVG14}, \cite{GPSVG15}, \cite{KRS14} \cite{KRSI}):
Considering metrics which need not be more regular than $W^{1,p}$ with $p\in(n+1,\infty]$ and carrying out a blow-up analysis of solutions, $w$, of (\ref{eq:thin_obst}) around free boundary points, it is possible to assign to each free boundary point $x_0\in \Gamma_w \cap B'_1$ the uniquely determined \emph{order of vanishing} $\kappa(x_0)$ of $w$ at this point (c.f. Proposition \propindep):
\begin{align*}
\kappa(x_0):= \lim\limits_{r \rightarrow 0_+} \frac{\ln\left( r^{-\frac{n+1}{2}} \left\| w \right\|_{L^2(B_r^+(x_0))} \right)}{\ln(r)} .
\end{align*} 
Since the order of vanishing satisfies the gap property that either $\kappa(x_0)= \frac{3}{2}$ or $\kappa(x_0)\geq 2$ (c.f. Corollary \corind), the free boundary can be decomposed as follows:
\begin{align*}
\Gamma_w \cap B_{1}' := \Gamma_{3/2}(w)\cup \bigcup\limits_{\kappa \geq 2} \Gamma_{\kappa}(w), 
\end{align*}
where $\Gamma_{\kappa}(w):= \{x_0\in \Gamma_{w} \cap B_{1}'| \ \kappa(x_0)= \kappa\}$. Moreover, noting that the mapping $\Gamma_w \ni x_0\mapsto \kappa(x_0)$ is upper-semi-continuous (c.f. Proposition \propsemicont), we obtain the set $\Gamma_{3/2}(w)$, which is called the \emph{regular free boundary}, is a relatively open subset of $\Gamma_w$. At each regular free boundary point $x_0\in \Gamma_{3/2}(w)$, there exists an $L^2$-normalized blow-up sequence $w_{x_0,r_j}$, which converges to a nontrivial global solution $w_{3/2}(Q(x_0)x)$ with flat free boundary. Here $w_{3/2}(x):=\Ree (x_n+ix_{n+1})^{3/2}$ is a model solution and $Q(x_0)\in SO(n+1)$ (c.f. Proposition \propblowup). 
By a more detailed analysis the regular free boundary can be seen to be $C^{1,\alpha}$ regular (c.f. Theorem \thmtwo) and a leading order expansion of solutions $w$ at the regular free boundary can be determined (c.f. Proposition \propasym and Corollary \corasym, c.f. also Proposition \ref{prop:asym2} in Section \ref{sec:asymp}).\\

In the sequel we will exclusively focus on the \emph{regular} free boundary. Due to its relative openness and by scaling, it is always possible to assume that the whole boundary in a given domain consists only of the regular free boundary. This convention will be used throughout the article; whenever referring to the ``free boundary'' without further details, we will mean the regular free boundary.

\subsection{Main results and ideas}

In this article our main objective is to prove \emph{higher} regularity of the (regular) free boundary if the metric $a ^{ij}$ is of higher (Hölder) regularity. In particular, we prove the analyticity of the free boundary for analytic coefficients:

\begin{thm}
\label{thm:higher_reg}
Let $a^{ij}:B_1^+ \rightarrow \R^{(n+1)\times (n+1)}_{sym}$ be a uniformly elliptic, symmetric, $W^{1,p}$ tensor field with  $p\in(n+1,\infty]$. Suppose that $w:B_{1}^+ \rightarrow \R$ is a solution of the variable coefficient thin obstacle problem (\ref{eq:thin_obst}) with metric $a^{ij}$.
\begin{itemize}
\item[(i)] Then the regular free boundary $\Gamma_{3/2}(w)$ is locally a $C^{1,1-\frac{n+1}{p}}$ graph if $p<\infty$ and a $C^{1,1-}$ graph if $p=\infty$.
\item[(ii)] Assume further that $a^{ij}$ is $C^{k,\gamma}$ regular with $k\geq 1$ and $\gamma\in (0,1)$. Then the regular free boundary $\Gamma_{3/2}(w)$ is locally a $C^{k+1,\gamma}$ graph.
\item[(iii)] Assume in addition that $a^{ij}$ is real analytic. Then the regular free boundary $\Gamma_{3/2}(w)$ is locally real analytic.
\end{itemize}
\end{thm}

We note that these results are sharp on the Hölder scale. In deriving the sharp gain of a full derivative, the choice of our function spaces play a key role (c.f. the discussion below for details on the motivation of our function spaces and Remark \ref{rmk:optreg} in Section \ref{sec:IFT1} for the optimality on the Hölder scale and the role of our function spaces).  \\
In addition to the previously stated results, we also deal with the regularity problem in the presence of inhomogeneities. 

\begin{thm}
\label{thm:higher_reg_inhom}
Let $a^{ij}:B_1^+ \rightarrow \R^{(n+1)\times (n+1)}_{sym}$ be a $W^{1,p}$ tensor field with $p\in(2(n+1),\infty]$ and let $f:B_1^+ \rightarrow \R$ be an $L^p(B_1^+)$ function. Suppose that $w:B_{1}^+ \rightarrow \R$ is a solution of the variable coefficient thin obstacle problem with metric $a^{ij}$ and inhomogeneity $f$:
\begin{equation}
\label{eq:thin_obst_inhom}
\begin{split}
\p_i a^{ij} \p_j w &= f \mbox{ in } B_1^+,\\
w \geq 0, \ a^{n+1, j} \p_{j}w\leq 0,\ w (a^{n+1,j}\p_j w) &= 0 \mbox{ in } B_1' \times \{0\}.
\end{split}
\end{equation}
\begin{itemize}
\item[(i)] Then the regular free boundary $\Gamma_{3/2}(w)$ is locally a $C^{1,\frac{1}{2}-\frac{n+1}{p}}$ graph.
\item[(ii) ] Assume in addition that $a^{ij}$ is a $C^{k,\gamma}$ tensor field with $k\geq 1$ and $\gamma\in (0,1]$ and let $f:B_1^+ \rightarrow \R$ be a $C^{k -1,\gamma}$ function. Then the regular free boundary $\Gamma_{3/2}(w)$ is locally a $C^{k+[\gamma+ \frac{1}{2}], \gamma+\frac{1}{2}- [\gamma+ \frac{1}{2}]}$ graph.
\item[(iii)] Moreover, assume that $a^{ij}$ is a real analytic tensor field and let $f:B_1^+ \rightarrow \R$ be a real analytic function. Then the regular free boundary $\Gamma_{3/2}(w)$ is locally real analytic.
\end{itemize}
Here $[\cdot]$ denotes the floor function.
\end{thm}

We note that this in particular includes the set-up with non-zero obstacles with as low as $W^{2,p}$, $p\in(2(n+1),\infty]$, regularity (c.f. Section \ref{sec:nonzero}). To to best of our knowledge Theorems \ref{thm:higher_reg} and \ref{thm:higher_reg_inhom} are the first results on higher regularity for the thin obstacle problem with variable coefficients and inhomogeneities.\\

In order to obtain a better understanding for the gain of the free boundary regularity with respect to the regularity of the inhomogeneity, it is instructive to compare Theorem \ref{thm:higher_reg_inhom}, i.e. the situation of the variable coefficient \emph{thin} obstacle problem, with that of the variable coefficient \emph{classical} obstacle problem (c.f. \cite{KN77}, \cite{F10}): In the classical obstacle problem (for the Laplace operator) there is a gain of \emph{one} order of differentiability with respect to the inhomogeneity, i.e. if $f\in C^{k,\alpha}$, then the (regular) free boundary $\Gamma_w$ is $C^{k+1,\alpha}$ regular.  This can be seen to be optimal by for instance considering an inhomogeneity which only depends on a single variable (the variable $x_n$ in whose direction the free boundary is a graph, i.e. $\Gamma_w=\{x\in B_1| \ x_n = g(x')\}$, with a choice of parametrization such that locally $|\nabla' g| \neq 0$), by using up to the boundary elliptic regularity estimates for all derivatives $\p_i w$ with respect to directions orthogonal to $e_n$ and by expressing the partial derivative $\p_{nn}w$ along the free boundary in terms of the parametrization $g$.\\
In contrast, in our situation of \emph{thin} obstacles, we gain \emph{three halves} of a derivative with respect to a general inhomogeneity. We conjecture that this is the optimal gain. As we are however dealing with a co-dimension two free boundary value problem, it seems harder to prove the optimality of this gain by similar means as for the classical obstacle problem. Yet, we remark that this gain of three-halves of a derivative also fits the scaling behavior (though not the regularity assumptions) of the inhomogeneities treated in \cite{DSS14}.\\

Let us explain the main ideas of deriving the regularity results of Theorems \ref{thm:higher_reg} and \ref{thm:higher_reg_inhom}:
In order to prove higher regularity properties of the free boundary, we rely on the partial Legendre-Hodograph transform (c.f. \cite{KN77}, \cite{KPS}) 
\begin{equation}\label{eq:L}
\begin{split}
T: B_1^+ &\rightarrow  Q_+:=\{ y \in \R^{n+1}| \ y_n \geq 0, y_{n+1} \leq 0\},\\
y:=T(x)&:= (x'', \p_{n}w, \p_{n+1} w), \ v(y):= w(x) - x_n y_{n} - x_{n+1}y_{n+1},
\end{split}
\end{equation}
which allows us to fix the (regular) free boundary:
\begin{align*}
T(\Gamma_w) \subset \{y\in \R^{n+1}| \ y_n=y_{n+1}=0\}.
\end{align*}
The asymptotic expansion of $w$ around $\Gamma_w$ implies that the transformation $T$ is asymptotically a square root mapping. Similar arguments as in \cite{KPS} yield that $T$ is invertible with inverse given by
\begin{align*}
T^{-1}(y)= (y'', - \p_n v(y), -\p_{n+1}v(y)).
\end{align*}
Thus, the free boundary $\Gamma_w$ can be parametrized in terms of the Legendre function $v$ as 
\begin{align*}
\Gamma_w \cap B_{1/2}' := \{x'\in B_{1/2}'| \ x_n = -\p_{n} v(x'',0,0)\}. 
\end{align*}
Therefore, it suffices to study the regularity properties of the Legendre function $v$, in order to derive higher regularity properties of the free boundary $\Gamma_w$.\\

Pursuing this strategy and investigating the properties of the Legendre function $v$, we encounter several difficulties:\\

\emph{Nonlinearity and subellipticity of the transformed equation, function spaces.} In analogy to the observations in \cite{KPS} the Hodograh-Legendre transformation $T$ transforms the \emph{uniformly} elliptic equation for $w$ into a fully nonlinear, \emph{degenerate} (sub)elliptic equation for $v$ (c.f. Proposition \ref{prop:bulk_eq}). Moreover, studying the asymptotics of $v$ at the degenerate set of the nonlinear operator (which is the image of the free boundary under the transformation $T$), the linearized operator (at $v$) is identified a perturbation of the (subelliptic) \emph{Baouendi-Grushin} operator (c.f. Section~\ref{sec:grushin}).\\
In this context a central new ingredient and major contribution of the article enters: Seeking to deduce regularity by an application of the implicit function theorem instead of direct and tedious elliptic estimates (c.f. Section \ref{sec:IFT1}), we have to capture the relation between the linearized and nonlinear operators in terms of our \emph{function spaces} (c.f. Section \ref{sec:holder}). This leads to the challenge of finding function spaces which on the one hand mimic the asymptotics of the Legendre function. This is a key requirement, since the perturbative interpretation of the fully nonlinear operator (as a nonlinear Baouendi-Grushin type operator) crucially relies on the asymptotics close the the straightened free boundary. On the other hand, the spaces have to allow for good regularity estimates for the linearized equation which is of Baouendi-Grushin type. In this context, we note that Calderon-Zygmund estimates and Schauder estimates have natural analogues for subelliptic operators like the Baouendi-Grushin operator. The mismatch between the vector space structure (relevant for derivatives) and the subelliptic
geometry allows for nontrivial choices in the definition of Sobolev spaces and higher order H\"older spaces.\\
In order to deal with both of the described conditions, we introduce \emph{generalized Hölder} spaces which are on the one hand adapted to the Baouendi-Grushin operator (for instance by relying on the intrinsic geometry induced by this operator) and on the other hand measure the distance to an approximating polynomial with the ``correct'' asymptotics close to the straightened free boundary (c.f. Section \ref{sec:holder} for the definition and properties of our generalized Hölder spaces and the Appendix, Section \ref{sec:append} for the proofs of these results). These function spaces are reminiscent of Campanato type spaces (c.f. \cite{Ca64}, \cite{Mo09}) and also of the polynomial approximations used by De Silva and Savin \cite{DSS14}. While similar constructions are possible for elliptic equations they seem to be not relevant there. For our problem however they are crucial.\\

\emph{Partial regularity and the implicit function theorem.} Seeking to avoid lengthy and tedious higher order derivative estimates for the Legendre function $v$, we deviate from the previous strategies of proving higher regularity that are present in the literature on the thin obstacle problem. Instead we reason by the \emph{implicit function theorem} along the lines of an argument introduced by Angenent (c.f. \cite{AN90}, \cite{AN90a}, \cite{KL12}). In this context we pre-compose our Legendre function, $v$, with a one-parameter family, $\Phi_a$, of diffeomorphisms leading to a one-parameter family of ``Legendre functions'', $v_a$ (Section~\ref{subsec:IFT0}). Here the diffeomeorphisms are chosen such that the parameter dependence on $a$ is analytic and so that the diffeomorphisms are the identity outside of a fixed compact set, whereas at the free boundary infinitesimally they generate a family of translations in the tangential directions. The functions $v_a$ satisfy a similar fully nonlinear, degenerate elliptic equation as $v$. Invoking the analytic implicit function theorem, we then establish that solutions of this equation are necessarily analytic in the parameter $a$, which, due to the uniqueness of solutions, implies that the family $v_a$ depends on $a$ analytically. As the family of diffeomorphisms, $\Phi_a$, infinitesimally generates translations in the tangential directions, this immediately entails the partial analyticity of the original Legendre function $v$ in the tangential variables.\\

\emph{Corner domain, function spaces.} 
Compared with the constant coefficient case in \cite{KPS}, the presence of variable coefficients leads to a completely new difficulty: By the definition of the Hodograph-Legendre transform \eqref{eq:L}, the transformation $T$ maps the upper half ball $B_1^+$ into the quarter space $Q_+$ (c.f. Section \ref{sec:Hodo}). In particular, the free boundary is mapped into the \emph{edge} of $Q_+$, which does not allow us to invoke standard interior regularity estimates there. \\
In contrast to the argument in \cite{KPS}, we cannot overcome this problem by reflecting the resulting solution so as to obtain a problem in which the free boundary is in the interior of the domain: Indeed, this would immediately lead to a loss of regularity of the coefficients $a^{ij}$ and hence would not allow us to prove higher regularity estimates up to the boundary.
Thus, instead, we have to work in the setting of an equation that is posed in the quarter space, where the singularity of the domain is centered at the straightened free boundary. This in particular necessitates regularity estimates in this (singular) domain which hold uniformly up to the boundary (c.f. Appendix, Section \ref{sec:quarter_Hoelder}). \\
In deducing these regularity estimates, we strongly rely on the form of our generalized Hölder spaces and on the interpretation of our fully nonlinear equation as a perturbation of the Baouendi-Grushin Laplacian in the quarter space which satisfies homogeneous Dirichlet data on $\{y_{n}=0\}$ and homogeneous Neumann data on $\{y_{n+1}=0\}$. As it is possible to classify and explicitly compute all the homogeneous solutions to this operator, an approximation argument in the spirit of \cite{Wa03} yields the desired regularity estimates in our generalized Hölder spaces (c.f. Appendix, Section \ref{sec:quarter_Hoelder}).\\

\emph{Low regularity metrics.} 
In the case of only $W^{1,p}$ regular metrics with $p\in(n+1,\infty]$,  and/or $L^p$ inhomogeneities, even \emph{away} from the free boundary a general solution $w$ is only $W^{2,p}$ regular. Thus, the previous arguments leading to the invertibility of the Hodograph-Legendre transform do not apply directly, as they rely on pointwise bounds for $D^2w$. 
To resolve this issue, we use the \emph{splitting technique} from \cite{KRSI} and introduce a mechanism that exchanges \emph{decay} and \emph{regularity}: More precisely, we split a general solution $w$ into two components $w=\tilde{u}+u$. Here the first component $\tilde{u}$ deals with the low regularity of the coefficients and the inhomogeneity:
\begin{align*}
a^{ij}\p_{ij} \tilde{u} - \dist(x,\Gamma_w)^{-2}\tilde{u} = f - (\p_i a^{ij})\p_j w \mbox{ in } B_1 \setminus \Lambda_w, \ \tilde{u}=0 \mbox{ on } \Lambda_w.
\end{align*}
Due to the inclusion of the strongly coercive term $-\dist(x,\Gamma_w)^{-2}\tilde{u}$ in the equation, the solution $\tilde{u}$ has a \emph{strong decay} properties (compared to $w$) towards $\Gamma_w$. We hence interpret it as a controlled error.\\
The second contribution $u$ is now of better \emph{regularity} away from the free boundary $\Gamma_w$, as it solves the non-divergence form elliptic equation
\begin{align*}
a^{ij}\p_{ij} u  = - \dist(x,\Gamma_w)^{-2}\tilde{u} \mbox{ in } B_1 \setminus \Lambda_w, \ \tilde{u}=0 \mbox{ on } \Lambda_w.
\end{align*}
Moreover, it captures the essential behavior of the original function $w$ (c.f. Lemma \ref{lem:lower1'} and Proposition \ref{prop:improved_reg1}). In particular, the free boundary $\Gamma_w$ is the same as the free boundary $\Gamma_u:=\partial_{B_1'}\{x\in B_1': u(x)>0\}$ of $u$. We then apply our previous arguments to $u$ and correspondingly obtain the regularity of the free boundary.

\subsection{Literature and related results}
The thin obstacle problem has been studied extensively beginning with the fundamental works of Caffarelli \cite{Ca79}, Uraltseva \cite{U85}, \cite{U87}, Kinderlehrer \cite{Ki81}, and the break through results of Athanasopoulos, Caffarelli \cite{AC06}, as well as Athanasopoulos, Caffarelli, Salsa \cite{ACS08} and Caffarelli, Silvestre, Salsa \cite{CSS}. While there is a quite good understanding of many aspects of the \emph{constant} coefficient problem, the \emph{variable} problem has only recently received a large amount of attention: Here, besides the early work of Uraltseva \cite{U87}, in particular the articles by Garofalo, Smit Vega Garcia \cite{GSVG14} and Garofalo, Petrosyan, Smit Vega Garcia \cite{GPSVG15} and the present authors \cite{KRS14}, \cite{KRSI} should be mentioned. While the methods differ -- the first two articles rely on a frequency function approach and an epiperimetric inequality, the second two articles build on a Carleman estimate as well as careful comparison arguments -- in both works the regularity of the regular free boundary is obtained for the variable coefficient problem under low regularity assumptions on the metric.\\
Hence, it is natural to ask whether the free boundary regularity can be improved if higher regularity assumptions are made on the coefficients and what the precise dependence on the regularity of the coefficients amounts to. In the constant coefficient setting, the higher regularity question has independently been addressed by De Silva, Savin \cite{DSS14}, who prove $C^{\infty}$ regularity of the free boundary by approximation arguments, and by Koch, Petrosyan, Shi \cite{KPS}, who prove analyticity of the free boundary. While the precise dependence on the coefficient regularity is well understood for the \emph{classical} obstacle with variable coefficients \cite{F10}, to the best of our knowledge this question has not yet been addressed in the framework of the \emph{variable} coefficient \emph{thin} obstacle problem.

\subsection{Outline of the article}
The remainder of the article is organized as follows: After briefly introducing the precise setting of our problem and fixing our notation in the following Section \ref{sec:prelim}, we recollect the asymptotic behavior of solutions of (\ref{eq:thin_obst}) in Section \ref{sec:asymp}. With this at hand, in Section \ref{sec:Hodo} we introduce the partial Hodograph-Legendre transformation in the case of $C^{k,\gamma}$ metrics with $k\geq 1$, obtain its invertibility (c.f. Proposition \ref{prop:invertibility}) and in Section \ref{sec:Legendre} derive the fully nonlinear, degenerate elliptic equation which is satisfied by the Legendre function $v$  (Proposition \ref{prop:bulk_eq}). Motivated by the linearization of this equation, we introduce our generalized Hölder spaces (c.f. Definitions \ref{defi:Hoelder}, \ref{defi:Hoelder1} and \ref{defi:spaces}) which are adapted to the geometry of the Baouendi-Grushin Laplacian (Section \ref{sec:holder}). Exploring the (self-improving) structure of the nonlinear equation for the Legendre function $v$ (c.f. Proposition \ref{prop:error_gain}), we deduce regularity properties of the Legendre function $v$ (c.f. Proposition \ref{prop:regasymp}) by an iterative bootstrap argument (c.f. Proposition \ref{prop:error_gain}) in Section \ref{sec:improve_reg}. In Section~\ref{sec:fb_reg} we build on this regularity result and proceed with the application of the implicit function theorem to prove the optimal regularity of the regular free boundary when the metrics $a^{ij}$ are $ C^{k,\gamma}$ Hölder regular for some $k\geq 1$ (c.f. Theorem \ref{prop:hoelder_reg_a}). Moreover, we also derive analyticity of the free boundary for analytic metrics (c.f. Theorem \ref{prop:analytic}). This provides the argument for the first two parts of Theorem \ref{thm:higher_reg}. Next, in Section~\ref{sec:W1p} we study the Hodograph-Legendre transformation for $W^{1,p}$ metrics with $p\in (n+1,\infty]$ and thus derive the optimal regularity result of Theorem \ref{thm:higher_reg} (i). Using similar ideas, we also discuss the necessary adaptations in proving regularity results in the presence of inhomogeneities and nonzero obstacles (c.f. Proposition \ref{prop:inhomo_2}).  Finally, in the Appendix, Section \ref{sec:append}, we prove a characterization of our function spaces introduced in Section~\ref{sec:holder} and show an a priori estimate for the Baouendi-Grushin Laplacian in these function spaces (c.f. Section \ref{sec:quarter_Hoelder}). We also discuss auxiliary regularity and mapping properties which we use in the derivation of the asymptotics and in the application of the implicit function theorem (c.f. Sections \ref{sec:XY}, \ref{sec:kernel}).

\section{Preliminaries}
\label{sec:prelim}
\subsection{Conventions and normalizations}
\label{sec:conventions}
In the sequel we introduce a number of conventions which will be used throughout this paper.
Any tensor field $a^{ij}:B_1^+\rightarrow \R^{(n+1)\times (n+1)}_{sym}$ in this paper is uniformly elliptic, symmetric and at least $W^{1,p}$ regular for some $p\in (n+1,\infty]$. Furthermore, we assume that
\begin{itemize}
\item[(A1)]$a^{ij}(0)=\delta^{ij}$,
\item[(A2)] (Uniform ellipticity) $\frac{1}{2}|\xi|^2\leq a^{ij}(x)\xi_i\xi_j \leq 2|\xi|^2$ for each $x\in B_1^+$ and $\xi\in \R^{n+1}$,
\item[(A3)] (Off-diagonal) $a^{i,n+1}(x',0)=0$ for $i\in \{1,\dots,n\}$.
\end{itemize}
Here (A1)-(A2) follow from an affine transformation. The off-diagonal assumption (A3) is a consequence of a change of coordinates (c.f. for instance Section 2.1 in \cite{KRS14} and Uraltseva \cite{U85}), which allows to reduce \eqref{eq:thin_obst} to 
\begin{equation}
\label{eq:varcoeff}
\begin{split}
\p_ia^{ij}\p_jw=0 &\text{ in } B_1^+,\\
w\geq 0,\quad \p_{n+1}w\leq 0, \ w(\p_{n+1}w)=0 &\text{ on } B'_1.
\end{split}
\end{equation}

Under the above assumptions (A1)-(A3), a solution $w$ to the thin obstacle problem is  $C^{1,\min\{1-\frac{n+1}{p},\frac{1}{2}\}}_{loc}$ regular and of $\dist(x,\Gamma_w)^{3/2}$ growth at the free boundary $\Gamma_w$, i.e. for any $x_0\in \Gamma_w\cap B^+_{1/2}$,
\begin{equation}\label{eq:interior_est}
\sup _{ B_r(x_0)}|\nabla w|\leq C(n,p, \|a^{ij}\|_{W^{1,p}})\|w\|_{L^2(B_1^+)}r^{\frac{1}{2}}, \quad r\in (0,1/2).
\end{equation} 
This regularity and growth behavior is optimal by the interior regularity and the growth behavior of the model solution $w_{3/2}(x)=\Ree(x_n+ix_{n+1})^{3/2}$. We refer to \cite{U85} for the $C^{1,\alpha}_{loc}$ regularity and to \cite{KRSI} for the optimal $C^{1,\min\{1-\frac{n+1}{p},\frac{1}{2}\}}_{loc}$ regularity as well as the growth result. In this paper we will always work with a solution $w\in C^{1,\min\{1-\frac{n+1}{p},\frac{1}{2}\}}_{loc}(B_1^+)$, for which \eqref{eq:interior_est} holds true.\\

In order to further simplify our set-up, we observe the following symmetry properties of our problem:\\
(Symmetry) Equation \eqref{eq:thin_obst} is invariant under scaling and multiplication. More precisely, if $w$ is a solution to \eqref{eq:thin_obst}, then for $x_0\in K\Subset B'_1$, for $c\geq 0$ and $\lambda>0$, the function 
$$x\mapsto  cw(x_0+\lambda x)$$ 
is a solution to \eqref{eq:thin_obst} (with coefficients $a^{ij}(x_0+\lambda\cdot)$) in $B_r^+$, $r\in (0, \lambda^{-1}(1-|x_0|)]$.\\

These symmetry properties are for instance crucial in carrying out rescalings around the (regular) free boundary:
Assuming that $x_0\in \Gamma_{3/2}(w)$ is a regular free boundary point and defining $w_{x_0,\lambda}(x):=w(x_0+\lambda x)/\lambda^{\frac{3}{2}}$, $\lambda\in (0,1/4)$, the asymptotic expansion around the regular free boundary (c.f. Proposition \propasym) yields that
\begin{align*}
w_{x_0,\lambda}(x)\rightarrow a(x_0)w_{x_0}(x) \mbox{ in } C^{1,\beta}_{loc}(\R^{n+1}_+)
\end{align*}
for each $\beta\in (0,1/2)$ as $\lambda\rightarrow 0_+$. Here $w_{x_0}$ is a global solution with flat free boundary and $a(x_0)>0$ is a constant. \\

In this paper we are interested in the higher regularity of $\Gamma_{3/2}(w)$ under an appropriate higher regularity assumption on the metric $a^{ij}$. All the results given below are \emph{local} estimates around \emph{regular} free boundary points. Using the scaling and multiplication symmetries of the equation, we may hence without loss of generality suppose the following normalization assumptions (A4)-(A7): 
\begin{itemize}
\item[(A4)] $0\in \Gamma_{3/2}(w)$,
\end{itemize}
that $w$ is sufficiently close to $w_{3/2}$ and that the metric is sufficiently flat in the following sense: For $\epsilon_0,c_\ast>0$ small
\begin{itemize}
\item[(A5)] $\|w-w_{3/2}\|_{C^1(B_1^+)}\leq \epsilon_0$,
\item[(A6)] $\|\nabla a^{ij}\|_{L^p(B_1^+)}\leq c_\ast$.
\end{itemize}
By \cite{KRSI}, if $\epsilon_0$, $c_\ast $ are sufficiently small depending on $n,p,\|w\|_{L^2(B_1^+)}$, then assumptions (A5)-(A6) imply that $\Gamma_w\cap B'_{1/2}\subset \Gamma_{3/2}(w)$ and that $\Gamma_w\cap B'_{1/2}$ is a $C^{1,\alpha}$ graph, i.e. after a rotation of coordinates $\Gamma_{w}\cap B_{1/2}' = \{x'=(x'',x_n,0)\cap B_{1/2}'| \ x_n = g(x'')\}$, for some $\alpha\in (0,1)$. Moreover, we have the following estimate for the (in-plane) outer unit normal $\nu_{x_0}$ of $\Lambda_w$ at $x_0$
\begin{equation}\label{eq:normal}
|\nu_{x_0}-\nu_{\tilde{x}_0}|\lesssim \max\{\epsilon_0,c_\ast\}|x_0-\tilde{x}_0|^\alpha, \text{ for any }x_0, \tilde{x}_0\in \Gamma_w\cap B'_{1/2},
\end{equation}
For notational simplicity, we also assume that
\begin{itemize}
\item[(A7)] $\nu_0=e_n$. 
\end{itemize}
From now on, we will always work under the assumptions (A1)-(A7).

\subsection{Notation}
\label{sec:notation}
Similarly as in \cite{KRSI} we use the following notation:\\
\emph{Geometry.}
\begin{itemize}
\item $\R^{n+1}_+:=\{(x'',x_n,x_{n+1})\in \R^{n+1} | x_{n+1}\geq 0\}$.
\item $B_r(x_0):=\{x\in \R^{n+1}| |x-x_0|<r\}$, where $|\cdot|$ is the norm induced by the Euclidean metric, $B_r^+(x_0):=B_r(x_0)\cap \R^{n+1}_+$, $B'_r(x_0):=B_r(x_0)\cap \{x_{n+1}=0\}$. If $x_0$ is the origin, we simply write $B_r$, $B_r^+$ and $B'_r$.
\item Let $w$ be a solution of \eqref{eq:thin_obst}, then $\Lambda_w:=\{(x',0)\in B'_1| w(x',0)=0\}$ is the \emph{contact set}, $\Omega_w:=B'_1\setminus \Lambda_w$ is the \emph{positivity set}, $\Gamma_w:=\p_{B'_1}\Lambda_w\cap B'_1$ is the \emph{free boundary}, $\Gamma_{3/2}(w):=\{x\in \Gamma_w| \kappa_x=\frac{3}{2}\}$ is the \emph{regular set} of the free boundary, where $\kappa_x$ is the vanishing order at $x$.  
\item For $x_0\in \Gamma_w$, we denote by $\mathcal{N}_{x_0}=\{x\in B^+_{1/4}(x_0)\big|\dist(x,\Gamma_w)\geq \frac{1}{2}|x-x_0|\}$ the non-tangential cone at $x_0$. 
\item $\mathcal{C}'_\eta(e_n):=\mathcal{C}_\eta(e_n)\cap \{e_{n+1}=0\}$ is a tangential cone (with axis $e_n$ and opening angle $\eta$).
\item $Q_+:=\{(y'',y_n,y_{n+1})\in \R^{n+1} | y_n\geq 0, y_{n+1}\leq 0\}$.
\item $\tilde{\mathcal{B}}_r(y_0):=\{y\in \R^{n+1}| d_G(y,y_0)<r\}$, where $d_G(\cdot, \cdot)$ is the Baouendi-Grushin metric (c.f. Definition~\ref{defi:Grushinvf}). $\tilde{\mathcal{B}}_r^+(y_0):=\tilde{\mathcal{B}}_r(y_0)\cap Q_+$.
\item In the $Q_+$ with the Baouendi-Grushin metric $d_G(\cdot,\cdot)$, given $y_0\in P:=\{y_n=y_{n+1}=0\}$ we denote by $\mathcal{N}_G(y_0):=\{x\in \tilde{\mathcal{B}}_{1/4}^+(y_0)\big|\dist_G(y,P)\geq \frac{1}{2}d_G(y,y_0)\}$ the Baouendi-Grushin non-tangential cone at $y_0$.
\item We use the Baouendi-Grushin vector fields $Y_i$, $i\in\{1,\dots,n+1\}$, (c.f. Definition \ref{defi:Grushinvf}) and the modified Baouendi-Grushin vector fields $\tilde{Y}_i$, $i\in\{1,\dots,2n\}$ (c.f. Definition \ref{defi:Hoelder1}).
\item For $k\in \N$, we denote by $\mathcal{P}_k^{hom}$ the space of homogeneous polynomials (w.r.t. the Grushin scaling) of order $k$ (c.f. Definition \ref{defi:poly}), and by $\mathcal{P}_k$ the vector space of homogeneous polynomials of order less than or equal to $k$.
\end{itemize}
\emph{Functions and function spaces.}
\begin{itemize}
\item $w_{3/2}(x):= c_n \Ree(x_n + i x_{n+1})^{3/2}$, where $c_n>0$ is a normalization constant ensuring that $\| w_{3/2} \|_{L^2(B_{1}^+)}=1$.
\item $w_{1/2}(x):= c_n \Ree(x_n + i x_{n+1})^{1/2}$ and $\bar{w}_{1/2}(x):= - c_n \Imm(x_n + i x_{n+1})^{1/2}$, where $c_n>0$ denotes the same normalization constant as above.
\item We denote the \emph{asymptotic profile} at a point $x_0 \in \Gamma_{3/2}(w)\cap B_{1}'$ by $\mathcal{W}_{x_0}$. It is given by
\begin{align*}
\mathcal{W}_{x_0}(x)=a(x_0)w_{3/2}\left(\frac{(x-x_0)\cdot \nu_{x_0}}{(\nu_{x_0}\cdot A(x_0)\nu_{x_0})^{1/2}}, \frac{x_{n+1}}{(a^{n+1,n+1}(x_0))^{1/2}}\right).
\end{align*}
\item For a solution $w$ to (\ref{eq:varcoeff}) and a point $x_0\in \Gamma_{w}$ we define a \emph{blow-up sequence} $w_{x_0,\lambda}(x):=\frac{w(x_0 + \lambda x)}{\lambda^{3/2}}$ by rescaling with $\lambda \in (0,1)$. The asymptotic expansion from Proposition \propasym implies that as $\lambda \rightarrow 0$ it converges to the \emph{blow-up profile} $\mathcal{W}_{x_0}(\cdot+ x_0)$. Here $\mathcal{W}_{x_0}$ denotes the asymptotic profile from above.
\item In the sequel, we use spaces adapted to the Baouendi-Grushin operator $\Delta_G$ and denote the corresponding Hölder spaces by $C^{k,\alpha}_{\ast}$ (c.f. Definitions \ref{defi:Hoelder}, \ref{defi:Hoelder1}). Moreover, relying on these, we construct our generalized Hölder spaces $X_{\alpha,\epsilon}, Y_{\alpha,\epsilon}$ which are appropriate for our corner domains (c.f. Definition \ref{defi:spaces}).
\item We use the notation $C_0(Q_+)$ to denote the space of all continuous functions vanishing at infinity.
\item Let $\R^{(n+1)\times (n+1)}_{sym}$ denote the space of symmetric matrices and let 
$$G: \R^{(n+1)\times (n+1)}_{sym} \times \R^{n+1} \times \R^{n+1} \rightarrow \R, \ (M,P,y)\mapsto G(M,P,y),$$ 
with $M=(m_{k\ell})_{k \ell} \in \R^{(n+1) \times (n+1)}_{sym}$, $P = (p_1,\dots,p_{n+1})\in \R^{n+1}$ and $y=(y_1,\dots,y_{n+1})\in \R^{n+1}$. We denote the partial derivative with respect to the different components by
\begin{align*}
\p_{m_{k\ell}}G(M,P,y)&:= \frac{\p G(M,P,y)}{\partial m_{k \ell}},\\
\p_{p_{k}}G(M,p,y) &:=\frac{\p G(M,P,y)}{\partial p_{k}},\\
\p_{y_k}G(M,p,y)&:=\frac{\p G(M,P,y)}{\partial y_{k}}.
\end{align*}
\item $\D_G$ stands for the Baouendi-Grushin operator
\begin{align*}
\D_{G} v:= (y_n^2 + y_{n+1}^2)\D'' v + \p_{nn}v + \p_{n+1,n+1}v,
\end{align*}
where $\D''$ denotes the Laplacian in the tangential variables, i.e. in the $y''$-variables of $y=(y'',y_n,y_{n+1})$.
\end{itemize}
The notation $A\lesssim B$ means that $A\leq CB$ with $C$ depending only on dimension $n$. 

\section{Hodograph-Legendre Transformation}
\label{sec:HLTrafo}

In this section we perform a partial Hodograph-Legendre transform of our problem (\ref{eq:varcoeff}). While fixing the free boundary, this comes at the price of transforming our uniformly elliptic equation in the upper half ball into a fully nonlinear, degenerate (sub)elliptic equation in the lower quarter ball (c.f. Sections \ref{sec:Hodo}, \ref{sec:Legendre}, Propositions \ref{prop:invertibility}, \ref{prop:bulk_eq}). In particular, in addition to the difficulties in \cite{KPS}, the domain in which our problem is posed now contains a corner. In spite of this additional problem, as in \cite{KPS} we identify the fully nonlinear equation as a perturbation of the Baouendi-Grushin operator with symmetry (i.e. with Dirichlet-Neumann data) by a careful analysis of the asymptotic behavior of the Legendre transform (c.f. Section \ref{sec:Legendre}, Example \ref{ex:linear} and Section \ref{sec:grushin}).

\subsection{Asymptotic behavior of the solution $w$}
\label{sec:asymp}

We begin by deriving and collecting asymptotic expansions for higher order derivatives of solutions to our equation (c.f. \cite{KRSI}). This will prove to be advantageous in the later sections (e.g. Sections \ref{sec:Hodo}, \ref{sec:Leg}).

\begin{prop}[\cite{KRSI}, Proposition 4.6]
Let $a^{ij}\in W^{1,p}(B_1^+, \R^{(n+1)\times (n+1)}_{sym})$ with $p\in (n+1,\infty]$ be a uniformly elliptic tensor.
Assume that $w:B_1^+ \rightarrow \R$ is a solution to the variable coefficient thin obstacle problem and that it satisfies the following conditions: 
\label{prop:asym2}
There exist positive constants $\epsilon_0$ and $c_{\ast}$ such that
\begin{itemize}
\item[(i)] $\|w-w_{3/2}\|_{C^1(B_1^+)}\leq \epsilon_0$,
\item[(ii)] $\|\nabla a^{ij}\|_{L^p(B_1^+)}\leq c_\ast$.
\end{itemize}
Then if $\epsilon_0$ and $c_\ast$ are sufficiently small depending on $n,p$, there exists some $\alpha\in (0,1-\frac{n+1}{p}]$  such that $\Gamma_w\cap B_{1/2}^+$ is a $C^{1,\alpha}$ graph. Moreover, at each free boundary point $x_0\in \Gamma_w\cap B^+_{1/4}$, there exists an asymptotic profile, $\mathcal{W}_{x_0}(x)$, 
\begin{align*}
\mathcal{W}_{x_0}(x)=a(x_0)w_{3/2}\left(\frac{(x-x_0)\cdot \nu_{x_0}}{(\nu_{x_0}\cdot A(x_0)\nu_{x_0})^{1/2}}, \frac{x_{n+1}}{(a^{n+1,n+1}(x_0))^{1/2}}\right),
\end{align*}
such that for any $x\in B_{1/4}^+(x_0)$
\begin{align*}
(i)\quad &\left|\p_i w(x)-\p_i \mathcal{W}_{x_0}(x)\right|\leq C_{n,p}\max\{\epsilon_0,c_{\ast}\}|x-x_0|^{\frac{1}{2}+\alpha}, \quad i\in\{1,\dots,n\},\\
(ii)\quad &\left| \p_{n+1} w(x)-\p_{n+1}\mathcal{W}_{x_0}(x)\right|\leq C_{n,p}\max\{\epsilon_0,c_{\ast}\} |x-x_0|^{\frac{1}{2}+\alpha},\\
(iii)\quad &\left|w(x)-\mathcal{W}_{x_0}(x)\right|\leq C_{n,p}\max\{\epsilon_0,c_{\ast}\} |x-x_0|^{\frac{3}{2}+\alpha}.
\end{align*}
Here $x_0\mapsto a(x_0)\in C^{0,\alpha}(\Gamma_w\cap B_{1/2}^+)$, $\nu_{x_0}$ is the (in-plane) outer unit normal of $\Lambda_w$ at $x_0$ and $A(x_0)=(a^{ij}(x_0))$. Furthermore, $w_{3/2}(x)= c_n \Ree(x_n+ i x_{n+1})^{3/2}$, where $c_n>0$ is a dimensional constant which is chosen such that $\| w_{3/2}\|_{L^{2}(B_1^+)}=1$.
\end{prop}

Assuming higher regularity of the metric allows us to use a scaling argument to deduce the  asymptotics for higher order derivatives in non-tangential cones.

\begin{prop}
\label{prop:improved_reg}
Let $a^{ij}\in C^{k,\gamma}(B_1^+, \R^{(n+1)\times (n+1)}_{sym})$ with $k\geq 1$ and $\gamma\in (0,1)$ be uniformly elliptic. Let $\alpha>0$ be the Hölder exponent from Proposition \ref{prop:asym2}. There exist $\epsilon_0$ and $c_\ast$ sufficiently small depending on $n,p$ such that if 
\begin{itemize}
\item[(i)] $ \|w-w_{3/2}\|_{C^1(B_1^+)}\leq \epsilon_0$,
\item[(ii)]$ [a^{ij}]_{\dot{C}^{k,\gamma}(B_1^+)}\leq c_\ast,$
\end{itemize}
then for each $x_0\in \Gamma_w\cap B^+_{1/4}$, an associated non-tangential cone $x\in \mathcal{N}_{x_0}:=\{x\in B^+_{1/4}(x_0)| \ \dist(x,\Gamma_w)\geq \frac{1}{2}|x-x_0|\}$ and for all multi-indeces $\beta$ with $|\beta|\leq k+1$ we have
\begin{align*}
\left|\p^\beta w(x)-\p^\beta \mathcal{W}_{x_0}(x)\right|& \leq C_{\beta,n,p} \max\{\epsilon_0,c_{\ast}\} |x-x_0|^{\frac{3}{2}+\alpha-|\beta|},\\
\left[\p^\beta w-\p^\beta \mathcal{W}_{x_0}\right]_{\dot{C}^{0,\gamma}(\mathcal{N}_{x_0}\cap (B_{3\lambda /4 }^+(x_0)\setminus B_{\lambda/2 }^+(x_0)))}& \leq C_{\beta,n,p}\max\{\epsilon_0,c_{\ast}\}\lambda^{\frac{3}{2}+\alpha-\gamma-|\beta|}.
\end{align*}
Here $\alpha$ is the same exponent as in Proposition~\ref{prop:asym2} and $\lambda \in (0,1)$.
\end{prop}

\begin{rmk}
\label{rmk:improved_reg}
It is possible to extend the above asymptotics for $x$ in the full neighborhood $B_{1/4}^+(x_0)$ as
\begin{align*}
&\left|\p^\beta w(x)-\p^\beta \mathcal{W}_{x_0}(x)\right|\leq C_{\beta,n,p} \max\{\epsilon_0,c_{\ast}\} |x-x_0|^{\frac{1}{2}+\alpha}\dist(x,\Gamma_w)^{-|\beta|+1}.
\end{align*}
Here it is necessary to introduce the distance to the free boundary instead of measuring it with a negative power of $|x-x_0|$.
\end{rmk}

Before coming to the proof of Proposition \ref{prop:improved_reg}, we state an immediate corollary, which will be important in the derivation of the asymptotics of the Legendre function in Proposition \ref{prop:holder_v} in Section \ref{sec:Leg}.

\begin{cor}
\label{cor:improved_reg}
Assume that the conditions of Proposition \ref{prop:improved_reg} hold. Let $$w_{x_0,\lambda}(x):=\frac{w(x_0+\lambda x)}{\lambda^{3/2}},\quad \lambda>0.$$ Then, 
\begin{align*}
&[\p^\beta w_{x_0,\lambda}-\p^\beta \mathcal{W}_{x_0}(x_0+\cdot)]_{\dot{C}^{0,\gamma}(\mathcal{N}_{0}\cap (B_{3 /4}^+\setminus B_{1/2}^+))}\leq C_{n,p} \max\{\epsilon_0,c_{\ast}\}\lambda^{\alpha}.
\end{align*}
\end{cor}

\begin{proof}[Proof of Proposition \ref{prop:improved_reg}]
The proof of the proposition follows from elliptic estimates in Whitney cubes, which in turn are reduced to estimates on the scale one by scaling the problem.\\
We only prove the result for $k=1$ (i.e. in case of $|\beta|=2$) and restrict ourselves to the $L^{\infty}$ estimates. For $k> 1$ and for the second estimate the argument is similar. Moreover, we observe that the case $|\beta|=1$ is already covered in Proposition~\ref{prop:asym2}.  We begin by considering the tangential derivatives of $w$: Let $\tilde{v}:=\p_\ell w$ with $\ell\in \{1,\dots, n\}$. Then $\tilde{v}$ satisfies
\begin{align*}
\p_i(a^{ij}\p_j\tilde{v})=\p_iF^i, \quad F^i=-(\p_\ell a^{ij})\p_jw,
\end{align*}
with the boundary conditions 
\begin{align*}
\tilde{v}&=0 \text{ on } \Lambda_w, \quad \p_{n+1}\tilde{v}=0 \text{ on } B'_1\setminus \Lambda_w. 
\end{align*}
Also, the derivative of the profile functions, $\p_\ell \mathcal{W}_{x_0}$, satisfies
\begin{align*}
\p_i(a^{ij}\p_j(\p_\ell \mathcal{W}_{x_0}))=g_1+g_2,
\end{align*}
with 
\begin{align*}
 g_1=(\p_{\ell} a^{ij}) \p_j\p_i \mathcal{W}_{x_0}, \ g_2=(a^{ij}(x)-a^{ij}(x_0))\p_{ij}\p_\ell\mathcal{W}_{x_0}.
\end{align*}

Seeking to combine the information on the functions $\tilde{v}$ and $\p_{\ell} \mathcal{W}_{x_0}$, we define  
\begin{align*}
\tilde{u}(x):=\frac{w(x_0+\lambda x)-\mathcal{W}_{x_0}(x_0+\lambda x)}{\lambda^{\frac{3}{2}+\alpha}}, \ 0<\lambda<1/4.
\end{align*}
Due to the previous considerations,
$\p_\ell \tilde{u}$ satisfies the equation
\begin{align*}
\p_i(a^{ij}(x_0+\lambda \cdot)\p_j \p_\ell \tilde{u})=\p_i \tilde{F}^i - \tilde{g}_1 - \tilde{g}_2 \text{ in } B_1^{+}.
\end{align*}
Here 
\begin{equation}
\begin{split}
\label{eq:tilde}
\tilde{F}^i(x)=\lambda^{\frac{1}{2}-\alpha}F^i(x_0+\lambda x),\\
\tilde{g}_1(x)=\lambda^{\frac{3}{2}-\alpha} g_1(x_0+\lambda x),\\
\tilde{g}_2(x)=\lambda^{\frac{3}{2}-\alpha} g_2(x_0+\lambda x).
\end{split}
\end{equation}
Moreover, by the asymptotics of $w$ at $x_0$ which were given in (iii) of Proposition~\ref{prop:asym2}, we obtain the following $L^{\infty}$ bound in the non-tangential cone $\mathcal{N}_0=\{x\in B^+_{1/4}|\dist(x,\Gamma_{w_{x_0,\lambda}})\geq \frac{1}{2}|x|\}$ for all $\ell \in \{1,\dots,n+1\}$: 
\begin{align}\label{eq:max}
|\partial_{\ell} \tilde{u}|\lesssim C_{n,p}\max\{\epsilon_0,c_{\ast}\}.
\end{align} 
Noting that
\begin{align*}
|F^i(x)|&\lesssim c_{\ast} \dist(x,\Gamma_w)^{1/2},\\
|g_1(x)|&\lesssim  c_{\ast} \dist(x,\Gamma_w)^{-1/2},\\
|g_2(x)|&\lesssim c_{\ast} |x-x_0| \dist(x,\Gamma_w)^{-3/2},
\end{align*}
recalling that $\lambda \dist(x, \Gamma_{w_{x_0,\lambda}}) = \dist(\lambda (x-x_0), \Gamma_w)$ and using (\ref{eq:tilde}) yields
\begin{equation}
\label{eq:distresc}
\begin{split}
|\tilde{F}^i(x)|&\lesssim c_{\ast} \lambda^{1-\alpha}\dist(x,\Gamma_{w_{x_0,\lambda}})^{1/2},\\
|\tilde{g}_1(x)|&\lesssim c_{\ast} \lambda^{1-\alpha}\dist(x,\Gamma_{w_{x_0,\lambda}})^{-1/2},\\
|\tilde{g}_2(x)|&\lesssim c_{\ast} \lambda^{1-\alpha}\dist(x,\Gamma_{w_{x_0,\lambda}})^{-3/2}.
\end{split}
\end{equation}
By the definition of $\mathcal{N}_0$ the expressions involving the distance functions in (\ref{eq:distresc}) are uniformly (in $\lambda$) bounded in $B_1^+\setminus B_{1/4}^+$. Moreover, it is immediate to check that the semi-norms $[\tilde{F}^i]_{C^{0,\gamma}}$ are uniformly bounded. 
For $\ell\in \{1,\dots, n\}$, we apply the $C^{1,\gamma}$ estimate to $\p_\ell\tilde{u}$, which holds up to the boundary, in $\mathcal{N}_0\cap (B^+_1\setminus B^+_{1/4})$ (note that with $\epsilon_0, c_\ast$ sufficiently small, $\mathcal{N}_0\cap (B_1^+\setminus B^+_{1/4})$ does not intersect the free boundaries $\Gamma_w$ or $\Gamma_{\mathcal{W}_{x_0}}$, thus $\tilde{u}$ satisfies either Dirichlet or Neumann conditions):
\begin{align}
\label{eq:tangential}
\|\p_\ell\tilde{u}\|_{C^{1,\gamma}(\mathcal{N}_0\cap (B^+_{3/4}\setminus B^+_{1/2}))} \lesssim \|\partial_{\ell}\tilde{u}\|_{L^\infty(\mathcal{N}_0\cap (B^+_1\setminus B^+_{1/4}))}+  c_{\ast}\lambda^{1-\alpha}.
\end{align}
In order to obtain a full second derivatives estimate, we now combine (\ref{eq:tangential}) with the equation for $\partial_{n+1}\tilde{u}$ to also obtain
\begin{align*}
\|\p_{n+1,n+1}\tilde{u}\|_{C^{0,\gamma}(\mathcal{N}_0\cap (B^+_{3/4}\setminus B^+_{1/2}))} \lesssim \sum\limits_{\ell=1}^{n}\|\partial_{\ell}\tilde{u}\|_{L^\infty(\mathcal{N}_0\cap (B^+_1\setminus B^+_{1/4}))}+  c_{\ast} \lambda^{1-\alpha}.
\end{align*}
Rescaling back and using \eqref{eq:max} consequently leads to 
\begin{align*}
|\nabla \p_\ell w(x)-\nabla \p_\ell \mathcal{W}_{x_0}(x)|\lesssim \max\{\epsilon_0,c_{\ast}\}\lambda^{-\frac{1}{2}+\alpha} \text{ in } \mathcal{N}_{x_0}\cap (B^+_{3\lambda/4}(x_0)\setminus B^+_{\lambda/2}(x_0)),
\end{align*}
for $\ell\in \{1,\dots, n+1\}$.
Since this holds for any $\lambda\in (0,1/4)$, we conclude that
\begin{align*}
|\p^\beta w(x)-\p^\beta \mathcal{W}_{x_0}(x)|\lesssim \max\{\epsilon_0,c_{\ast}\} |x-x_0|^{-\frac{1}{2}+\alpha}, \quad x\in \mathcal{N}_{x_0},\ |\beta|=2.
\end{align*}
\end{proof}

\begin{rmk}
\label{rmk:normal}
In the next section we will strongly use the asymptotics of the first order derivatives $\p_\ell w$ with $\ell \in \{1,\dots,n+1\}$. Hence, for future reference we state them explicitly:
\begin{align*}
\p_e\mathcal{W}_{x_0}(x)&=b_e(x_0)w_{1/2}\left(\frac{(x-x_0)\cdot \nu_{x_0}}{(\nu_{x_0}\cdot A(x_0)\nu_{x_0})^{1/2}}, \frac{x_{n+1}}{(a^{n+1,n+1}(x_0))^{1/2}}\right),\\
\p_{n+1}\mathcal{W}_{x_0}(x)&=b_{n+1}(x_0)\bar w_{1/2}\left(\frac{(x-x_0)\cdot \nu_{x_0}}{(\nu_{x_0}\cdot A(x_0)\nu_{x_0})^{1/2}}, \frac{x_{n+1}}{(a^{n+1,n+1}(x_0))^{1/2}}\right),
\end{align*}
where
\begin{align*}
w_{1/2}(x)&= c_n \Ree(x_n + i x_{n+1})^{1/2},\quad \bar{w}_{1/2}(x)= - c_n \Imm(x_n + i x_{n+1})^{1/2},\\
b_e(x_0)&=\frac{3(e\cdot \nu_{x_0}) a(x_0)}{2(\nu_{x_0}\cdot A(x_0)\nu_{x_0})^{1/2}},\quad b_{n+1}(x_0)=\frac{3 a(x_0)}{2(a^{n+1,n+1}(x_0))^{1/2}},
\end{align*}
and $c_n>0$ is the same normalization constant as in Proposition \ref{prop:asym2}.
\end{rmk}

\begin{rmk}\label{rmk:convention}
For simplicity we can, and in the sequel will, further assume that $0\in \Gamma_w$ and that
\begin{align*}
\nu_0=e_n, \quad b_n(0)=b_{n+1}(0)=1\ (\text{which corresponds to } a(0)=2/3).
\end{align*}
Thus, 
$$\nabla \mathcal{W}_0(x)=(0,w_{1/2}(x),\bar w_{1/2}(x)).$$
Moreover, under the assumptions of Proposition \ref{prop:asym2} we can also bound the $\dot{C}^{0,\alpha}$ semi-norm of $b_n, b_{n+1}$ by $\max\{\epsilon_0,c_{\ast}\}$.
\end{rmk}

Last but not least, we recall a sign condition on $\p_n w$ and $\p_{n+1}w$, which plays an important role in the determination of the image of the Hodograph-Legendre transform in (\ref{eq:mapT}) in Section \ref{sec:Hodo}. An extension of this to the set-up of $W^{1,p}$, $p\in (n+1,\infty]$, metrics is recalled in Section \ref{sec:ext}. As explained in \cite{KRSI} this requires an additional splitting step. 

\begin{lem}[Positivity, \cite{KRSI}, Lemma 4.12.]
\label{lem:lower1}
Let $a^{ij}:B_1^+ \rightarrow \R^{(n+1)\times (n+1)}_{sym}$ be a tensor field that satisfies the conditions from Section \ref{sec:conventions} and in addition is $C^{1,\gamma}$ regular for some $\gamma \in (0,1)$. Let $w:B_1^+ \rightarrow \R$ be a solution of the thin obstacle problem with metric $a^{ij}$ and assume that it satisfies the normalizations from Section \ref{sec:conventions}. Then there exist positive constants $\eta= \eta(n)$ and $c=c(n)$ such that
\begin{align}
\label{eq:lower1}
\p_ew(x)\geq c\dist(x,\Lambda_w)\dist(x,\Gamma_w)^{-\frac{1}{2}}, \quad  x\in B_{\frac{1}{2}}^+
\end{align}
for $e\in \mathcal{C}'_\eta(e_n):=\mathcal{C}_\eta(e_n)\cap \{e_{n+1}=0\}$,  which is a tangential cone (with axis $e_n$ and opening angle $\eta$). Similarly, 
\begin{align*}
\p_{n+1}w(x)\leq -c \dist(x,\Omega_w)\dist(x,\Gamma_w)^{-\frac{1}{2}}, \quad x\in B_{\frac{1}{2}}^+.
\end{align*}
\end{lem}

\subsection{Hodograph-Legendre transformation}
\label{sec:Hodo}
In this section we perform a partial Hodograph-Legendre transformation to show the higher regularity of the free boundary with zero obstacle. In the sequel, we assume that the metric satisfies $a^{ij}\in C^{1,\gamma}(B_1^+, \R^{(n+1)\times (n+1)}_{sym})$ with $\gamma\in (0,1)$.\\

We define the partial Hodograph-Legendre transformation associated with $w$ as
\begin{align}
\label{eq:def_Legendre}
T=T^w:B_1^+\rightarrow \R^{n+1}, \quad y=T(x)=(x'', \partial_{n} w(x), \partial_{n+1}w(x)).
\end{align}
The regularity of $w$ immediately implies that $T\in C^{0,1/2}(B_1^+)$. Moreover, 
\begin{equation}
\label{eq:mapT}
\begin{split}
T(B_1^+\setminus B'_1)&\subset \{y_n>0, y_{n+1}<0\},\\
T(\Lambda_w)&\subset\{y_n=0, y_{n+1}\leq 0\}, \ T(B'_1\setminus \Lambda_w)\subset \{y_n>0, y_{n+1}=0\},\\
T(\Gamma_w)&\subset\{y_n=y_{n+1}=0\}.
\end{split}
\end{equation}
Here the first inclusion is a consequence of Lemma \ref{lem:lower1}.
Using the leading order asymptotic expansions from Section \ref{sec:asymp}, we prove the invertibility of the transformation: 

\begin{prop}[Invertibility of $T$]\label{prop:invertibility}
Suppose that the assumptions of Proposition~\ref{prop:asym2} hold. Then, if $[\nabla a^{ij}]_{\dot{C}^{0,\gamma}(B_1^+)}\leq c_{\ast}$ and if $\epsilon_0$ and $c_\ast $ are sufficiently small, the map $T$ is a homeomorphism from $B_{1/2}^+$ to $T(B_{1/2}^+) \subset \{y\in \R^{n+1}| \ y_n\geq 0, y_{n+1}\leq 0\}$. Moreover, away from $\Gamma_w$, $T$ is a $C^1$ diffeomorphism. 
\end{prop}

The proof of this result essentially relies on the facts that for each fixed $x''$, the transformation $T$ is asymptotically a square root mapping and the free boundary $\Gamma_w$ is sufficiently flat (i.e. it is a $C^{1,\alpha}$ graph with slow varying normals c.f. \eqref{eq:normal}). Hence, the main idea is to show the injectivity of $T$ on dyadic annuli around the free boundary. At these points the map $T$ is differentiable, which allows us to exploit the non-degeneracy of the derivative of $T$. To achieve this reduction to dyadic annuli we exploit the asymptotic structure of the functions $w$ (c.f. Propositions \ref{prop:asym2} and \ref{prop:improved_reg}).

\begin{proof}
\emph{Step 1: Homeomorphism.}\\
We begin with the injectivity of $T$ in $B_{1/2}^+$.
Since $T$ fixes the first $n-1$ variables, it is enough to show that for each $x_0\in \Gamma_w\cap B_{1/2}'$, $T$ is injective on the set $H_{x_0}:=\{(x''_0,x_n,x_{n+1})\}\cap B_{1/2}^+$. Moreover, as $\Gamma_w$ is given as a graph of a $C^{1,\alpha}$ function $g$, it suffices to prove that $T(x)\neq T(\tilde{x})$ for any two points $x,\tilde{x}\in H_{x_0}$ such that $x,\tilde{x}\notin \Gamma_w$.
In order to obtain this, we first prove that the mapping $T_1:=\psi\circ T$ is injective (and a homeomorphism) on $B_{1/2}^+$. Here $\psi:\R^{n+1}\rightarrow \R^{n+1}$ with $\psi(z)=(z'',z_n^2-z_{n+1}^2, -2z_nz_{n+1})$. Note that $T_1(x)=(x'',(\p_n w(x))^2 -(\p_{n+1} w(x))^2, - 2 \p_n w(x) \p_{n+1}w(x))$. We rely on the asymptotic expansion of $\nabla w$. In a second step, we then return to the mapping properties of $T$.\\

\emph{Step 1a: $T_1$ is a homeomorphism.}
We begin with the injectivity of $T_1$.
By Proposition~\ref{prop:improved_reg}, for $x\in B^+_{1/2}$
\begin{align*}
\p_n w(x)&= w_{1/2}(x)+\max\{\epsilon_0,c_{\ast}\}O(|x|^{\frac{1}{2}+\alpha}),\\
\p_{n+1}w(x)&=\bar{w}_{1/2}(x)+\max\{\epsilon_0,c_{\ast}\}O(|x|^{\frac{1}{2}+\alpha}).
\end{align*}
Hence, a direct computation gives that 
\begin{align*}
&T_1(x)=x+E_0(x), \\
&\text{where } E_0:B_{1/2}^+\rightarrow \R^{n+1},\ |E_0(x)|=\max\{\epsilon_0,c_{\ast}\}O(|x|^{1+\alpha}).
\end{align*}
In general, by the explicit asymptotic expansions of $\p_nw$ and $\p_{n+1}w$ around $x_0\in \Gamma_w\cap B^+_{1/2}$ (c.f. Proposition~\ref{prop:improved_reg}) and by using the fact that $b_n,b_{n+1}\in C^{0,\alpha}(\Gamma_w\cap B^+_{1/2})$, $\nu(x_0)=\nu_{x_0}\in C^{0,\alpha}(\Gamma_w\cap B^+_{1/2})$ and $A(x_0)\in C^{0,\alpha}(\Gamma_w\cap B^+_{1/2})$, we have 
\begin{equation}
\label{eq:identity}
\begin{split}
T_1(x)- T_1(x_0)&=(x-x_0) + E_{x_0}(x), \quad x\in B_{1/2}^+(x_0)\\
\text{where } |E_{x_0}(x)|&\lesssim \max\{\epsilon_0,c_{\ast}\} \left(|x-x_0||x_0|^\alpha+ |x-x_0|^{1+\alpha}\right).
\end{split}
\end{equation}
Here we recall that as indicated in Remark \ref{rmk:convention} we may assume that the Hölder constants of $b_n(x_0), b_{n+1}(x_0)$ are controlled by $\max\{\epsilon_0,c_{\ast}\}$.
From the identity (\ref{eq:identity}) we note that if $\epsilon_0, c_*$ are sufficiently small and if $x_0\in \Gamma_w\cap B_{1/2}^+$, then for $x\in B_{1/2}^+(x_0)$
\begin{align}
\label{eq:cont_boundary}
(1-\frac{1}{4})|T_1(x) - T_1(x_0)|  \leq |x-x_0|\leq (1+\frac{1}{4})|T_1(x) - T_1(x_0)|.
\end{align}
Thus, if there are $x, \tilde{x}\in H_{x_0}$ with $x, \tilde{x}\notin \Gamma_w$ such that $T_1(x)=T_1(\tilde{x})$, then necessarily 
\begin{align}\label{eq:quo}
\frac{1}{2}\leq \frac{|x-x_0|}{|\tilde{x}-x_0|}\leq 2.
\end{align}
Without loss of generality, we assume that $|x-x_0|\leq |\tilde{x}-x_0|$ and define $r:=|x-x_0|$. Then \eqref{eq:quo} implies that $x, \tilde{x}\in A_{r,2r}^+(x_0)\cap H_{x_0}$, where $A_{r,2r}^+(x_0)$  is the closed cylinder centered at $x_0$:
\begin{align*}
A^+_{r,2r}(x_0)&:=\{(x'',x_n,x_{n+1})\in B_1^+| \ |x''-x''_0|\leq r,\\
&\qquad r\leq \sqrt{(x_n-(x_0)_n)^2+(x_{n+1}-(x_0)_{n+1})^2}\leq 2r\}.
\end{align*}
Since $\Gamma_w$ is $C^{1,\alpha}$ with $|\nu_{x_0}-\nu_{\tilde{x}_0}|\lesssim \max\{\epsilon_0,c_\ast\}|x_0-\tilde{x}_0|^\alpha$,  for any $x_0, \tilde{x}_0\in \Gamma_w\cap B'_{1/2}$ and $\nu_0=e_n$, we have that $\Gamma_w \cap A^+_{r,2r}(x_0)=\emptyset$ for a sufficiently small (but independent of $r$) choice of the constants $\epsilon_0, c_\ast$. Thus, $T_1$ is a $C^1$ mapping in $A^+_{r,2r}(x_0)\cap B_{1/2}^+$ (because $w$ is $C^{2,\gamma}$ away from $\Gamma_w$). We compute $DT_1$ in $A_{r,2r}^+(x_0)\cap B_{1/2}^+$. By using the asymptotics of $D w$ and $D^2w$ around $x_0$ (c.f. Propositions \ref{prop:asym2}, \ref{prop:improved_reg}), we obtain
\begin{align*}
|DT_1(x)-I|\lesssim \max\{\epsilon_0,c_{\ast}\}  \left(|x_0|^{\alpha}+r^{2\alpha}\right), \quad x\in A_{r,2r}^+(x_0) \cap \mathcal{N}_{x_0}\cap B_{1/2}^+,
\end{align*}
where $I$ is the identity map.
Therefore, for sufficiently small, universal constants $\epsilon_0, c_\ast$, the map $T_1$ is injective in $A_{r,2r}^+(x_0)\cap \mathcal{N}_{x_0}\cap B_{1/2}^+$. This implies that $T_1(x)\neq T_1(\tilde{x})$.\\

\emph{Step 1b: $T_1:B_{1/2}^+ \rightarrow T_1(B_{1/2}^+)$ is a homeomorphism.}
By the continuity of $T_1$ and by the invariance of domain theorem, we infer that, as a mapping from $\inte(B_{1/2}^+)$ to $T_1(\inte(B_{1/2}^+))$, $T_1$ is a homeomorphism. We claim that this is also true for $T_1$ as a map from $B_{1/2}^+$ to $T_1(B_{1/2}^+)$. Indeed, due to our previous considerations in Step 1a, $T_1$ is injective (and hence invertible) on the whole of $B_{1/2}^+$ (as a map onto its image). Hence, it suffices to prove the continuity of the inverse. Here we distinguish three cases: Let $y\in T_1(B_{1/2}^+)$ and first assume that $y\in T_1(\Gamma_w\cap B_{1/2}')$. Then, (\ref{eq:cont_boundary}) immediately implies the continuity of $T_1^{-1}$ at $y$. Secondly, we assume that $y\in T_1(B_{1/2}'\setminus \Lambda_w)$. Let $x=T^{-1}_1(y)\in B_{1/2}'\setminus \Lambda_w$. Then we carry out an even reflection of $w$ about $x_{n+1}$ (and a corresponding partly even, partly odd reflection for $a^{ij}$) as described in Remark 3.8 in \cite{KRSI}. The resulting reflected function $\tilde{w}$ is still $C^{1,1/2}$ regular in a (sufficiently small) neighborhood $B_{\rho}(x) \subset B_{1/2}^+ \setminus \Lambda_w$ of $x$. Moreover, the $y_{n+1}$ -component of $T_1^{\tilde{w}}$ changes sign on passing from $x_{n+1}>0$ to $x_{n+1}<0$. Thus, the mapping $T_1^{\tilde{w}}$ is still injective as a mapping from $B_{\rho}(x)$ to $T^{\tilde{w}}_1(B_{\rho}(x))$. Since it is also continuous, the invariance of domain theorem implies that it is a homeomorphism from $B_{\rho}(x)$ to $T^{\tilde{w}}_1(B_{\rho}(x))$, which is an open subset in $\R^{n+1}$ containing $y$. In particular, this implies that our original mapping, $(T^{w}_1)^{-1}$, is continuous at $y\in T_1(B_{1/2}'\setminus \Lambda_w)$. Last but not least, for a point $y\in T_1(B_{1/2}' \cap \inte(\Lambda_w))$, we argue similarly. However, instead of using an even reflection, we carry out an odd reflection of $w$ about $x_{n+1}$. Again, we note that the associated map  $T_1^{\tilde{w}}$ changes sign on passing from $x_{n+1}>0$ to $x_{n+1}<0$. Thus, arguing as in the second case, we again obtain the continuity of $(T_1^{w})^{-1}$ at $y$. Combining the results of the three cases therefore yields that $T_1$ is a homeomorphism as a map from $B_{1/2}^+$ to $T_1(B_{1/2}^+)$, which is relatively open in $\{y_n\geq 0, y_{n+1}\leq 0\}$.
\\

\emph{Step 1c: $T$ is a homeomorphism.}
By definition of $T_1$, we have that $T_1 = \psi\circ T$, where $\psi(x):=(x'', x_n^2 - x_{n+1}^2, -2 x_{n}x_{n+1})$. We show that the injectivity of $T$ follows immediately from the injectivity of $T_1$. As $T(B_1^+)\subset \{y\in \R^{n+1}| y_{n}\geq 0, y_{n+1}\leq 0\}$ and as $\psi$ is injective on this quadrant, we obtain $T(U)=\psi^{-1}\circ T_1(U)$ for any $U\subset B_1^+$. Since $T_1$ is open and $\psi$ is continuous, this implies that $T$ is open. Combining this with the continuity of $T$, we obtain that $T$ is homeomorphism from $B_{1/2}^+$ to $T(B_{1/2}^+)\subset \{y\in \R^{n+1}| y_n\geq 0, y_{n+1}\leq 0\}$.\\

\emph{Step 2: Differentiability.}
Recalling the regularity of the metric, $a^{ij}\in C^{1,\gamma}$ for $\gamma>0$, we observe that $w\in C^{2,\gamma}_{loc}(B_1^+\setminus \Gamma_w)$. Thus, $T$ is $C^1$ away from $\Gamma_w$. In order to show that $T$ is a $C^1$ diffeomorphism away from $\Gamma_w$, it suffices to compute its Jacobian. For $x\in B_{1/2}^+\setminus \Gamma_w$, let $x_0=(x'',g(x''),0)$ be the projection onto $\Gamma_w$. Then, by the asymptotics for $D^2w$ (Proposition~\ref{prop:improved_reg} applied in the non-tangential cone $\mathcal{N}_{x_0}$), we have
\begin{align*}
\det(DT(x))&=\p_{nn}w\p_{n+1,n+1}w-(\p_{n,n+1}w)^2 \\
&=\p_{nn}\mathcal{W}_{x_0}\p_{n+1,n+1}\mathcal{W}_{x_0}-(\p_{n,n+1}\mathcal{W}_{x_0})^2 + \max\{\epsilon_0,c_{\ast}\}  O(|x-x_0|^{-1+\alpha}).
\end{align*}
A direct computation gives
\begin{align*}
&\p_{nn}\mathcal{W}_{x_0}\p_{n+1,n+1}\mathcal{W}_{x_0}-(\p_{n,n+1}\mathcal{W}_{x_0})^2\\
&=-\frac{9}{16}a(x_0)^2\frac{e_n\cdot \nu_{x_0}}{(\nu_{x_0}\cdot A(x_0)\nu_{x_0})(a^{n+1,n+1}(x_0))} \frac{1}{\tilde{r}},
\end{align*}
where 
\begin{align*}
\tilde{r}=\left(\frac{((x-x_0)\cdot \nu_{x_0})^2}{(\nu_{x_0}\cdot A(x_0)\nu_{x_0})}+\frac{x_{n+1}^2}{a^{n+1,n+1}(x_0)}\right)^{1/2}.
\end{align*}
The $C^{0,\alpha}$ regularity of $\nu_{x_0}$ and the ellipticity of $A(x)=(a^{ij}(x))$ entail that
\begin{align*}
c|x-x_0|\leq \tilde{r}\leq C |x-x_0|, \text{ for some absolute constants }0<c<C<\infty.
\end{align*}
Thus, 
\begin{equation}
\label{eq:jacobi}
\det(DT(x))=-c|x-x_0|^{-1}+\max\{\epsilon_0,c_{\ast}\}  O(|x-x_0|^{-1+\alpha})<0.
\end{equation}
Therefore, after potentially choosing the constant $1/2=1/2(n,p,\alpha)>0$ even smaller, the implicit function theorem implies that $T$ and $T^{-1}$ are locally $C^{1}$. Due to the global invertibility, which we have proved above, the statement follows. 
\end{proof}

\subsection{Legendre function and nonlinear PDE}
\label{sec:Legendre}

In this section we compute a partial Legendre transform of a solution $w$ of our problem (\ref{eq:varcoeff}). In this context it becomes convenient to view the equation (\ref{eq:varcoeff}) in non-divergence form and to regard the equation in the interior as a special case of the problem
\begin{align*}
a^{ij}\p_{ij} u = f(Du,u,y),
\end{align*}
for a suitable function $f$. In our case $f(Du,u,y)=-(\p_{i} a^{ij})\p_j u$.
Starting from this non-divergence form, we compute the equation which the Legendre function satisfies (c.f. Proposition \ref{prop:bulk_eq}). By considering the explicit example of the Legendre transform of $\mathcal{W}_{x_0}$ for $x_0=0$, we motivate that the fully nonlinear equation in the bulk is related to the Baouendi-Grushin operator (c.f. Example \ref{ex:linear}). \\

From now on we will work in the image domain $T(B_{1/2}^+)$, where $T$ is the partial Hodograph transformation defined in \eqref{eq:def_Legendre}. For simplicity, we set $U:=T(B_{1/2}^+)$ and denote the straightened free boundary by $P:=T(\Gamma_w\cap B_{1/2}')$. We recall that the Hodograph transform was seen to be invertible in $U$ (c.f. Proposition \ref{prop:invertibility}).
For $y\in U$, we define the partial Legendre transform of $w$ by the identity 
\begin{equation}\label{eq:legendre}
v(y)=w(x)-x_{n}y_n-x_{n+1}y_{n+1}, \quad x=T^{-1}(y).
\end{equation}
A direct computation shows that 
\begin{equation}\label{eq:dual}
\partial_{y_i}v=\partial_{x_i}w, \ i=1,\ldots, n-1,\quad \partial_{y_{n}}v=-x_n,\quad \partial_{y_{n+1}}v=-x_{n+1}.
\end{equation}
As a consequence of \eqref{eq:dual}, the free boundary $\Gamma_w\cap B_{1/2}'$ is parametrized by 
\begin{align}
\label{eq:boundaryLH}
x_n=-\p_{y_n}v(y'',0,0).
\end{align}

As in \cite{KPS} the advantage of passing to the Legendre-Hodograph transform consists of fixing (the image of the) free boundary, i.e. by mapping it to the co-dimension two hyperplane $y=(y'',0,0)$. However, this comes at the expense of a more complicated, fully nonlinear, degenerate (sub)elliptic equation for $v$. We summarize this in the following:

\begin{prop}[Bulk equation]
\label{prop:bulk_eq}
Suppose that $a^{ij}\in C^{1,\gamma}(B_1^+, \R^{(n+1)\times (n+1)}_{sym})$ is uniformly elliptic.
Let $w:B_1^+ \rightarrow \R$ be a solution of the variable coefficient thin obstacle problem and let $v:U \rightarrow \R$ be its partial Legendre-Hodograph transform. Then $v\in C^{1}(U)$ and it satisfies the following fully nonlinear equation
\begin{equation}
\label{eq:nonlineq1}
\begin{split}
F(D^2v, D v, v,y)&=-\sum_{i,j=1}^{n-1}\tilde{a}^{ij}\det\begin{pmatrix}
\p_{ij}v& \p_{in}v & \p_{i,n+1}v\\
\p_{jn}v& \p_{nn}v & \p_{n,n+1}v\\
\p_{j,n+1}v & \p_{n,n+1}v &\p_{n+1,n+1}v
\end{pmatrix}\\
&+2\sum_{i=1}^{n-1}\tilde{a}^{i,n}\det\begin{pmatrix}
\p_{in}v & \p_{i,n+1}v\\
\p_{n,n+1}v & \p_{n+1,n+1}v
\end{pmatrix}\\
& \quad+2 \sum_{i=1}^{n-1}\tilde{a}^{i,n+1}\det\begin{pmatrix}
\p_{i,n+1}v & \p_{in}v\\
\p_{n,n+1}v & \p_{nn}v
\end{pmatrix}\\
&+\tilde{a}^{nn}\p_{n+1,n+1}v+\tilde{a}^{n+1,n+1}\p_{nn}v-2\tilde{a}^{n,n+1}\p_{n,n+1}v\\
&-\det\begin{pmatrix}
\p_{nn}v &\p_{n,n+1}v\\
\p_{n,n+1}v &\p_{n+1,n+1}v
\end{pmatrix}\left(\sum_{j=1}^{n-1}\tilde{b}^j\p_j v+\tilde{b}^ny_n+\tilde{b}^{n+1}y_{n+1}\right)=0,
\end{split}
\end{equation}  
where
\begin{align*}
\tilde{a}^{ij}(y)&:=a^{ij}(x)\big|_{x=(y'',-\p_nv(y),-\p_{n+1}v(y))},\\
\tilde{b}^j(y)&:=\sum_{i=1}^{n+1}(\p_{x_i}a^{ij})(x)\big|_{x=(y'',-\p_nv(y),-\p_{n+1}v(y))}.
\end{align*}
Moreover, the following mixed Dirichlet-Neumann boundary conditions hold:
\begin{align*}
v=0\text{ on } U\cap \{y_n=0\}; \quad \p_{n+1}v=0 \text{ on } U\cap \{y_{n+1}=0\}.
\end{align*}
In particular, $\Gamma_w\cap B_{1/4}$ is parametrized by $x_n=-\p_{y_n}v(y'',0,0)$.
\end{prop}

\begin{rmk}
For convenience of notation, in the sequel we will also use the notation $F(v,y):=F(D^2v,Dv,v,y)$. We emphasize that the coefficients $\tilde{a}^{ij}(y)$ depend on $v$ nonlinearly.
\end{rmk}

The proof of Proposition \ref{prop:bulk_eq} follows by computing the corresponding changes of coordinates:

\begin{proof}
Due to the regularity of $T^{-1}$ and $w$ \eqref{eq:dual} directly entails that $v\in C^{ 1}(U)$. The condition $w=0$ on $\Gamma_w\cap B^+_{1/4}$ immediately translates into $v=0$ on $P$. Moreover, it is easy to check from \eqref{eq:dual} and the Signorini boundary condition of $w$, that $v=0$ on $U\cap \{y_n=0\}$ and $\p_{n+1}v=0$ on $U\cap \{y_{n+1}=0\}$.

Now we derive the equation for $v$. Recalling that 
\begin{align*}
y=T(x)=(x', \partial_{x_n}w, \partial_{x_{n+1}} w), \quad x=T^{-1}(y)=(y', -\partial_{y_n}v, -\partial_{y_{n+1}}v),
\end{align*}
and using \eqref{eq:dual}, we have 
\begin{align*}
DT=\begin{pmatrix}
I_{n-1} & 0\\
A(w)& H(w)
\end{pmatrix},\quad 
DT^{-1}=\begin{pmatrix}
I_{n-1} & 0\\
A(v)& H(v)
\end{pmatrix} \mbox{ in } U\setminus P,
\end{align*}
where
\begin{align*}
A(w)=\begin{pmatrix}
\partial_{x_{n}x_1} w & \ldots & \partial_{x_{n}x_{n-1}} w\\
\partial_{x_{n+1}x_1} w & \ldots & \partial_{x_{n+1}x_{n-1}} w
\end{pmatrix},\quad 
H(w)=\begin{pmatrix}
\partial_{x_{n}x_n} w & \partial_{x_{n}x_{n+1}} w\\
\partial_{x_{n+1}x_n} w & \partial_{x_{n+1}x_{n+1}}w
\end{pmatrix},\\
A(v)=-\begin{pmatrix}
\partial_{y_{n}y_1} v & \ldots & \partial_{y_{n}y_{n-1}} v\\
\partial_{y_{n+1}y_1} v & \ldots & \partial_{y_{n+1}y_{n-1}} v
\end{pmatrix},\quad
H(v)=-\begin{pmatrix}
\partial_{y_{n}y_n} v & \partial_{y_{n}y_{n+1}} v\\
\partial_{y_{n+1}y_n} v & \partial_{y_{n+1}y_{n+1}}v
\end{pmatrix}.
\end{align*}

Next we express $D^2w(x)$ in terms of $D^2v(y)$ if $y\in U\setminus P$. Since $(DT)^{-1}=DT^{-1}$, we immediately obtain 
\begin{equation}\label{eq:relation}
H(w)=H(v)^{-1},\quad A(v)=-H(w)^{-1}A(w).
\end{equation}
Moreover, the identities \eqref{eq:relation} and \eqref{eq:dual} together with a direct calculation give
\begin{align}
(\partial_{y_iy_j}v)_{(n-1)\times (n-1)}&= (\partial_{x_ix_j}w)_{(n-1)\times (n-1)} - A(w)^t H(w)^{-1} A(w), \label{eq:hessianv}\\
(\partial_{x_ix_j}w)_{(n-1)\times (n-1)}&= (\partial_{y_iy_j}v)_{(n-1)\times (n-1)} - A(v)^t H(v)^{-1} A(v)\label{eq:hessianw}.
\end{align}

In order to compute the equation for $v$, we assume that $a^{ij}\in C^{1,\gamma}$ for some $\gamma>0$ and rewrite the equation for $w$ in non-divergence form
\begin{equation}\label{eq:nondivw}
a^{ij}\partial_{ij}w + (\partial_i a^{ij})\partial_j w=0.
\end{equation}
For convenience and abbreviation, we set
\begin{align*}
\tilde{a}^{ij}(y)&:=a^{ij}(x)\big|_{x=(y'',-\p_nv(y),-\p_{n+1}v(y))},\\
\tilde{b}^j(y)&:=\sum_{i=1}^{n+1}(\p_{x_i}a^{ij})(x)\big|_{x=(y'',-\p_nv(y),-\p_{n+1}v(y))}.
\end{align*}
Plugging \eqref{eq:relation}-\eqref{eq:hessianw} into (\ref{eq:nondivw}), and multiplying the resulting equation by
\begin{align*}
-J(v):=-\det\begin{pmatrix}
\p_{nn}v &\p_{n,n+1}v\\
\p_{n,n+1}v &\p_{n+1,n+1}v
\end{pmatrix},
\end{align*}
leads to the equation~(\ref{eq:nonlineq1}) for $v$.
\end{proof}

We conclude this section by computing the Legendre function of a 3/2-homogeneous blow-up of a solution to the variable coefficient thin obstacle problem. 

\begin{lem}
\label{lem:asymp_profile}
Let $w:B_{1}^+ \rightarrow \R$ be a solution of the variable coefficient thin obstacle problem and let $x_0\in \Gamma_w\cap B_{1/2}$. Assume that $v$ is the Legendre function of $w$ under the Hodograph transformation $y=T^w(x)$. Then at $y_0=T^w(x_0)$, the Legendre function $v$ has the asymptotic expansion
$$v(y)= v_{y_0}(y)+ \max\{\epsilon_0,c_\ast\}O(|y-y_0|^{3+2\alpha}),$$
with the leading order profile  
\begin{align*}
v_{y_0}(y) &=- \frac{4}{27 a^2(x_0)}\left( \left(\frac{\nu_{x_0}\cdot A(x_0)\nu_{x_0}}{(\nu_{x_0})_n} y_n \right)^3 \right.\\
& \quad \left.  - 3 \left( \frac{\nu_{x_0}\cdot A(x_0)\nu_{x_0}}{(\nu_{x_0})_n} (a^{n+1,n+1}(x_0)) \right)  y_ny_{n+1}^2 \right)\\
&\quad -g(y_0)y_n  + y_{n}\frac{(y''-y_0)\cdot \nu_{x_0}''}{(\nu_{x_0})_n},
\end{align*}
where $\nu_{x_0}:= (\nu_{x_0}'', (\nu_{x_0})_n,0)=\frac{(-\nabla''g(x_0), 1, 0)}{\sqrt{1+|\nabla''g(x_0)|^2}}$ denotes the (in-plane) outer normal to $\Lambda_w$ at $x_0$.
\end{lem}

\begin{proof}
The claim follows from a straightforward calculation. Indeed, recall that $y(x)=(x'',\p_nw(x),\p_{n+1}w(x))$. From the asymptotics of $\p_nw, \p_{n+1}w$ around $x_0\in \Gamma_w$ in Proposition~\ref{prop:asym2}, we obtain the asymptotics of the inverse $x=x(y)$  around $y_0=T^w(x_0)$. Additionally, recalling that $v(y)=w(x(y))-x_n(y)y_n-x_{n+1}(y)y_{n+1}$, we obtain the claimed asymptotic expansion of $v$ around $y_0$.
\end{proof}

It turns out that the function $v_{y_0}(y)$ provides good intuition for the behavior of solutions to (\ref{eq:nonlineq1}).
In order to obtain a better idea about the structure of $F(v,y)$, we compute its linearization at $ v_0(y)$, which is the leading order expansion of $v$ at the origin. It is immediate from Lemma~\ref{lem:asymp_profile} (and using the normalization in Remark~\ref{rmk:normal} and Remark~\ref{rmk:convention}) that
$$v_0(y)=-\frac{1}{3}\left(y_n^3-3y_ny_{n+1}^2\right).$$

\begin{example}[Linearization at $v_0$]
\label{ex:linear}
Let $v_0$ be the Legendre function of the blow-up limit $\mathcal{W}_{0}$ at the origin, which itself is a global solution to the Signorini problem with constant metric $a^{ij}=\delta^{ij}$. Then, the Legendre function $v_0$ satisfies the nonlinear PDE
\begin{align*}
F(D^2v)=-\p_{nn}v-\p_{n+1,n+1}v+ \sum\limits_{i=1}^{n-1} \det
\begin{pmatrix}
\p_{ii}v & \p_{in}v & \p_{i,n+1}v\\
\p_{ni}v & \p_{nn}v & \p_{n,n+1}v\\
\p_{n+1,i}v & \p_{n+1,n}v & \p_{n+1,n+1}v
\end{pmatrix}=0.
\end{align*}
A direct computation leads to 
\begin{align*}
\frac{\p F(M)}{\p m_{ij}}\big|_{M=D^2 v_0}=-
\begin{pmatrix}
4(y_n^2+y_{n+1}^2) & 0 & 0\\
0 & 1 &0\\
0& 0& 1
\end{pmatrix}.
\end{align*}
As a consequence, the linearization $L_{v_0}=D_vF \big|_{ v_0}=4(y_n^2+y_{n+1}^2)\Delta''+\p^2_{n,n}+\p^2_{n+1,n+1}$ is a \emph{constant coefficient Baouendi-Grushin operator}.
\end{example}

The previous example and the observation that around the origin $v$ is a perturbation of $v_0$ and $a^{ij}$ is a perturbation of the identity matrix, indicates that the linearization $D_v F$ (and hence $F$) can be viewed as a perturbation of the Baouendi-Grushin Laplacian. 
Motivated by this, we introduce function spaces which are adapted to the Baouendi-Grushin operator in the next section.

\section{Function spaces}
\label{sec:holder}
In this section we introduce and discuss generalized H\"older spaces (c.f. Definition \ref{defi:spaces}, Proposition \ref{prop:decompI}) which are adapted to our equation (\ref{eq:nonlineq1}). These are the spaces in which we apply the implicit function theorem in Section~\ref{sec:fb_reg} to deduce the tangential regularity of the Legendre function $v$.
In order to define these spaces, we use the intrinsic geometry induced by the Baouendi-Grushin operator. In particular, we work with the intrinsic (or Carnot-Caratheodory) distance (c.f. Definition \ref{defi:Grushinvf}) associated with the Baouendi-Grushin operator and corresponding intrinsic Hölder spaces (c.f. Definitions \ref{defi:Hoelder}, \ref{defi:Hoelder1}).\\
Our function spaces are inspired by Campanato's characterization of the classical H\"older spaces \cite{Ca64} and are reminiscent of the function spaces used in \cite{DSS14}. They are constructed on the one hand to capture the asymptotics of the Legendre function and on the other hand to allow for elliptic estimates for the Baouendi-Grushin operator (c.f. Proposition \ref{prop:invert}).

\subsection{Intrinsic metric for Baouendi-Grushin Laplacian}
\label{sec:intrinsic}
In this section we define the geometry which is adapted to our equation (\ref{eq:nonlineq1}). This is motivated by viewing our nonlinear operator from (\ref{eq:nonlineq1}) as a variable coefficient perturbation of the constant coefficient \emph{Baouendi-Grushin} operator (c.f. Example \ref{ex:linear})
\begin{align*}
\D_G:= (y_n^2 + y_{n+1}^2)\D'' + \p_n^2 + \p_{n+1}^2.
\end{align*}

The Baouendi-Grushin operator is naturally associated with the \emph{Baouendi-Grushin vector fields} and an \emph{intrinsic metric}:

\begin{defi}
\label{defi:Grushinvf}
Let $Y_i:=\sqrt{y_n^2+y_{n+1}^2}\p_i$, $i\in\{1,\dots, n-1\}$, $Y_n:=\p_n$, $Y_{n+1}:=\p_{n+1}$ denote the \emph{Baouendi-Grushin vector fields}. The metric associated with the vector fields $Y_i$ is 
\begin{align}
\label{eq:metr}
ds^2=\sum_{j=1}^{n-1}\frac{dy_j^2}{y_n^2+y_{n+1}^2}+dy_n^2+dy_{n+1}^2.
\end{align}
More precisely it is defined by the following scalar product in the tangent space:
\begin{align*}
g_{y}(v,w):= (y_n^2 + y_{n+1}^2)^{-1}\left(\sum\limits_{j=1}^{n-1}v_j w_j  \right) + v_n w_n + v_{n+1} w_{n+1},
\end{align*}
for all $y \in \R^{n+1}$, $v,w \in \spa\{Y_i(y)| \  i\in \{ 1,\dots, n+1\}\}$.
Let $d_G$ be the distance function associated with this sub-Riemannian metric (or the associated \emph{Carnot-Caratheodory metric}): 
\begin{multline*}
d_G(x,y) := \inf \{ \ell(\gamma)| \ \gamma: [a,b] \subset \R \rightarrow \R^{n+1} \mbox{ joins } x \mbox{ and } y,\\
 \dot{\gamma}(t)\in \spa\{Y_i(\gamma(t))| \ i\in \{1,\dots, n+1\}\}\},
\end{multline*} 
where
\begin{align*}
\ell(\gamma) := \int_{a}^{b} \sqrt{g_{\gamma(t)}(\dot{\gamma}(t), \dot{\gamma}(t))}dt.
\end{align*}
\end{defi}

\begin{rmk}\label{rmk:equi_dist}
We remark that for the family of dilations $\delta_\lambda(\cdot)$ which is defined by $\delta_\lambda(y'',y_n,y_{n+1}):=(\lambda^2 y'',\lambda y_{n},\lambda y_{n+1})$, we have $d_G(\delta_\lambda( p), \delta_\lambda (q))=|\lambda|d_G(p,q)$ for $p,q\in \R^{n+1}$. Moreover, from \eqref{eq:metr} for $\sqrt{y_n^2+y_{n+1}^2}\sim 1$ we have $ds^2\sim dy_1^2+\dots+dy_{n+1}^2$. Using these, it is possible to directly verify that $d_G$ is equivalent to the following quasi-metric 
\begin{align*}
d(x,y)=|x_n-y_n|+|x_{n+1}-y_{n+1}|+\frac{|x''-y''|}{|x_n|+|x_{n+1}|+|y_n|+|y_{n+1}|+|x''-y''|^{1/2}}.
\end{align*}
\end{rmk}

\begin{rmk}
\label{rmk:original_variables}
In order to elucidate our choice of metric, we derive its form in our original $x$-coordinates. To this end, we consider the case of the flat model solution $w(x)=\mathcal{W}_{0}(x)$. Denoting the Euclidean inner product on $\R^{n+1}$ by $g_0$ and defining $g_{\mathcal{W}_{0}}$ as the Baouendi-Grushin inner product from Definition~\ref{defi:Grushinvf}, (\ref{eq:metr}) (up to constants), we obtain that $g_{\mathcal{W}_{0}}=(x_n^2+x_{n+1}^2)^{-\frac{1}{2}}T_\ast g_{0}$, where $T$ is the Legendre transformation associated with $\mathcal{W}_0$. 
\end{rmk}

The previously defined intrinsic metric induces a geometry on our space. In particular, it defines associated Baouendi-Grushin cylinders/balls:

\begin{defi}
\label{defi:Grushincylinder}
Let $0<r\leq 1$. We set 
$$\mathcal{B}_r:= \{y\in \R^{n+1}| \ |y''|\leq r^2, \ y_{n}^2 + y_{n+1}^2 \leq r^2 \}$$ 
to denote the closed \emph{non-isotropic Baouendi-Grushin cylinders}. For  
$$y_0\in P:=\{(y'',y_n,y_{n+1})| y_n=y_{n+1}=0\}$$ we further define $\mathcal{B}_r(y_0):=y_0+\mathcal{B}_r$. 
In the quarter space, we restrict the cylinders to the corresponding intersection 
$$\mathcal{B}_{r}^+(y_0):=\mathcal{B}_{r}(y_0)\cap Q_{+}, \quad \text{where }Q_+:= \{y\in \R^{n+1}| \ y_{n}\geq 0, y_{n+1}\leq 0\}.$$
\end{defi}

\begin{rmk}
Due to Remark~\ref{rmk:equi_dist}, there are constants $c,C>0$ such that for any $y_0\in P$ 
\begin{align*}
\mathcal{\tilde{B}}_{cr}(y_0)\subseteq \mathcal{B}_r(y_0)\subseteq \mathcal{\tilde{B}}_{Cr}(y_0),\quad \text{where }\mathcal{\tilde{B}}_r(y_0)=\{y| d_G(y,y_0)< r\}.
\end{align*}
In the sequel, with slight abuse of notation, for $y_0\in P$ we will not distinguish between $\mathcal{\tilde{B}}_r(y_0)$ and $\mathcal{B}_r(y_0)$ for convenience of notation.  
\end{rmk}

\subsection{Function spaces}
\label{sec:functions}
In the sequel, we consider the intrinsic H\"older spaces which are associated with the geometry introduced in Section \ref{sec:intrinsic}: 

\begin{defi}
\label{defi:Hoelder}
Let $\Omega$ be a subset in $\R^{n+1}$ and let $\alpha\in (0,1]$. Then
\begin{align*}
C^{0,\alpha}_\ast(\overline{\Omega}):=\left\{u:\overline{\Omega}\rightarrow \R| \ \sup_{x,y\in \overline{\Omega}}\frac{|u(x)-u(y)|}{d_G(x,y)^\alpha}<\infty\right\}.
\end{align*}
Let
$$[u]_{\dot{C}^{0,\alpha}_\ast(\overline{\Omega})}:=\sup_{x,y\in \overline{\Omega}}\frac{|u(x)-u(y)|}{d_G(x,y)^\alpha}.$$
For $u\in C^{0,\alpha}_\ast(\overline{\Omega})$ we define
$$\|u\|_{C^{0,\alpha}_\ast(\overline{\Omega})}:=\|u\|_{L^\infty(\Omega)}+[u]_{C^{0,\alpha}_\ast(\overline{\Omega})}.$$
\end{defi}

\begin{rmk}
The mapping $\| \cdot \|_{C_{\ast}^{0,\alpha}}: C_{\ast}^{0,\alpha} \rightarrow [0,\infty)$ is a norm. By Remark \ref{rmk:equi_dist} 
\begin{align*}
C^{0,\alpha}_{\ast}(\overline{\Omega}) \hookrightarrow C^{0,\frac{\alpha}{2}}(\overline{\Omega}).
\end{align*}
Hence, the pair $(C^{0,\alpha}_\ast(\overline{\Omega}), \| \cdot \|_{C^{0,\alpha}_{\ast}(\bar{\Omega})})$ is a Banach space.
\end{rmk}

Based on the spaces from Definition \ref{defi:Hoelder}, we can further define higher order H\"older spaces:

\begin{defi}
\label{defi:Hoelder1}
Let 
$$\tilde{Y}_1=y_n\p_1, \quad \tilde{Y}_2=y_{n+1}\p_1, \quad \dots, \quad \tilde{Y}_{2n-1}=\p_n,\quad \tilde{Y}_{2n}=\p_{n+1}.$$ 
For $k\in \mathbb{N}$, $k\geq 1$, we say that $u\in C^{k,\alpha}_\ast(\overline{\Omega})$, if for all $\sigma_i\in \{1,\ldots, 2n\}$, $1\leq i\leq k$, the functions $u, \tilde{Y}_{\sigma_1}\cdots \tilde{Y}_{\sigma_i}u$ are continuous and $\tilde{Y}_{\sigma_1}\cdots\tilde{Y}_{\sigma_k}u \in C^{0,\alpha}_\ast(\overline{\Omega})$.
We define
\begin{equation}
\label{eq:norm1}
\begin{split}
\|u\|_{C^{k,\alpha}_\ast(\overline{\Omega})}& =\|u\|_{L^{\infty}(\overline{\Omega})} \\
&+ \sum_{j=1}^{k-1}\sum_{\sigma_1,\ldots,\sigma_j\in \{1,\dots, 2n\}}\|\tilde{Y}_{\sigma_1}\cdots \tilde{Y}_{\sigma_j}u\|_{L^\infty(\overline{\Omega})} \\
& +\sum_{\sigma_1,\ldots,\sigma_k\in \{1,\dots,2n\}}\|\tilde{Y}_{\sigma_1}\cdots \tilde{Y}_{\sigma_k}u\|_{C^{0,\alpha}_\ast(\overline{\Omega})}.
\end{split}
\end{equation}
\end{defi}

\begin{rmk}
The space $C^{k,\alpha}_{\ast}(\overline{\Omega})$ equipped with $\| \cdot \|_{C^{k,\alpha}_{\ast}(\overline{\Omega})}$ is a Banach space.\\
\end{rmk}

Building on the previously introduced Hölder spaces, we proceed to define the function spaces which we use to prove the higher regularity of the Legendre function $v$. These spaces, their building blocks and their role in our argument are reminiscent of the higher regularity approach of De Silva and Savin \cite{DSS14}. In contrast to the approach of De Silva and Savin we however use them in the \emph{linear} set-up in the sense that the (regular) free boundary has been fixed by the Legendre-Hodograph transform (at the expense of working with a degenerate (sub)elliptic, fully nonlinear equation). In this situation the approximation approach of De Silva and Savin simply becomes a Taylor expansion of our solution at the straightened free boundary.
Moreover, we do not carry out the expansion up to arbitrary order, but only up to order less than five. Beyond this we work with the implicit function theorem (c.f. Theorem \ref{prop:hoelder_reg_a} in Section \ref{sec:IFT1}), which is more suitable to the variable coefficients set-up. In particular, this restriction to an essentially leading order expansion with respect to the non-tangential variables allows us to avoid dealing with \emph{regularity issues in the non-tangential directions}. Working in a conical domain and with metrics and inhomogeneities which are not necessarily symmetric with respect to the non-tangential directions, we thus ignore potential higher-order singularities in the non-tangential variables. This has the advantage of deducing the desired partial regularity result in the tangential directions, which then entails the free boundary regularity, without having to deal with potentially arising non-tangential singularities.\\

Roughly speaking, our spaces interpolate between the regularity of the function at 
$P=\{y_n=y_{n+1}=0\}$ (at which the Baouendi-Grushin operator is only degenerate elliptic)
and at $\{\frac{1}{2}<y_n^2+y_{n+1}^2< 2\}$ (in which the Baouendi-Grushin operator is uniformly elliptic region). In order to make this rigorous, we need the notion of an \emph{homogeneous polynomial}:

\begin{defi}[Homogeneous polynomials]
\label{defi:poly}
Let $k\in \N$. We define the \emph{space of homogeneous polynomials of degree less than or equal to $k$} as
\begin{align*}
\mathcal{P}_k=&\{p_k(y)| \  p_k(y)=\sum_{|\beta|\leq k}a_\beta y^{\beta},\\
&\text{ such that }a_\beta=0 \text{ whenever }\sum_{i=1}^{n-1}2\beta_i+\beta_n+\beta_{n+1}>k\}.
\end{align*}
Moreover, we define the \emph{space of homogeneous polynomials of degree exactly $k$} as 
\begin{align*}
\mathcal{P}_k^{hom}=&\{p_k(y)| \  p_k(y)=\sum_{|\beta|\leq k}a_\beta y^{\beta},\\
&\text{ such that }a_\beta=0 \text{ whenever }\sum_{i=1}^{n-1}2\beta_i+\beta_n+\beta_{n+1}\neq k\}.
\end{align*}
\end{defi}

The definition of the homogeneous polynomials is motivated by the scaling properties of our operator $\Delta_G$. More precisely, we note the following dilation invariance property: if $u$ solves $\D_G u = f$, then the function $v(y):= u(\delta_\lambda(y))$, where $\delta_\lambda(y)=(\lambda^2y'',\lambda y_n,\lambda y_{n+1})$, solves
\begin{align*}
\D_G v = \lambda^2 f_\lambda,
\end{align*}
where $f_{\lambda}(y)= f(\delta_\lambda(y))$.
This motivates to count the order of the tangential variables $y''$ and the normal variables $y_n$, $y_{n+1}$ differently and define the homogeneous polynomials with respect to the Grushin scaling:  $p_k(\delta_\lambda(y))=\lambda^k p_k(y)$ for $p_k\in \mathcal{P}_k^{hom}$.

\begin{rmk}
We observe that for instance $P\in \mathcal{P}_3$ is of the form
\begin{align*}
P(y)&=c_0+\sum_{i=1}^{n+1}a_iy_i+\sum_{k\in \{1,\dots,n-1\},\ell\in \{n,n+1\}}a_{k\ell}y_ky_\ell\\
&+ \left(c_1y_n^3+c_2y_n^2y_{n+1}+c_3y_ny_{n+1}^2+c_4y_{n+1}^3\right).
\end{align*}
\end{rmk}

Using the notion of homogeneous polynomials, we further define an adapted notion of differentiability at the co-dimension two hypersurface $P$:

\begin{defi}\label{defi:diff}
Let $k\in \N$ and $\alpha \in (0,1]$. Given a function $f$, we say that \emph{$f$ is $C^{k,\alpha}_\ast$ at $P$}, if at each $y_0\in P$ there exists an approximating polynomial $P_{y_0}(y)=\sum a_\beta(y_0)(y-y_0)^{\beta}\in \mathcal{P}_k$ such that
\begin{align*}
f(y)=P_{y_0}(y)+O(d_G(y,y_0)^{k+2\alpha}), \quad \text{as } y\rightarrow y_0.
\end{align*}
\end{defi}

\begin{rmk}
We note that for a multi-index $\beta$ satisfying $\sum_{i=1}^{n-1}2\beta_i+\beta_n+\beta_{n+1}\leq k$, the evaluation $\partial^{\beta}P_{y_0}(y_0) = \beta! a_{\beta}(y_0)$ corresponds to the (classical) $\beta$ derivative of $f$ at $y_0$, i.e. $\p^\beta f(y_0)=\beta! a_{\beta}(y_0)$.
\end{rmk}

With this preparation, we can finally give the definition of our function spaces:

\begin{defi}[Function spaces]
\label{defi:spaces}
Let $\epsilon, \alpha\in (0,1]$. Then,
\begin{align*}
X_{\alpha,\epsilon}:=&\{v\in C^{2,\epsilon}_\ast (Q_+) \cap C_0(Q_+)| \ \supp(\Delta_G v)\subset \mathcal{B}_1^+,\ v\text{ is } C^{3,\alpha}_\ast \text{ at } P, \\
& 
v=0 \text{ on } \{y_n=0\},\ \p_{n+1}v=0\text{ on } \{y_{n+1}=0\}, \ \p_{nn}v=0 \text{ on } P,\\
& \text{and }\| v \|_{X_{\alpha, \epsilon}}<\infty\},\\
Y_{\alpha,\epsilon}:=&\{f\in C^{0,\epsilon}_\ast(Q_+)| \ \supp(f) \subset \mathcal{B}_1^+, \  f \text{ is } C^{1,\alpha}_\ast \text{ at }P,  \  f=\p_{n+1}f=0 \text{ on }P,\\
& \text{and }\|f\|_{Y_{\alpha,\epsilon}}<\infty\}.
\end{align*}
The corresponding norms are defined as
\begin{align*}
\|f\|_{Y_{\alpha, \epsilon}}& : =\sup_{\bar y\in P}[d_G(\cdot,\bar y)^{-(1+2\alpha-\epsilon)}(f-P_{\bar y})]_{\dot{C}^{0,\epsilon}_\ast(\mathcal{B}_3^+(\bar y))}  ,\\
\text{ where } & P_{\bar y}(y)=y_n \p_nf(\bar y);\\
\| v \|_{X_{\alpha,\epsilon}} &:=\sup_{ \bar{y}\in P} \left(\|d_G(\cdot, \bar y)^{-(3+2\alpha)}(v-P_{\bar y})\|_{L^{\infty}(\mathcal{B}_3^+(\bar y))} \right.\\
& \left. +\sum\limits_{i,j=1}^{n+1}[d_G(\cdot, \bar y)^{-(1+2\alpha- \epsilon)}Y_i Y_j (v-P_{\bar y})]_{\dot{C}^{0,\epsilon}_{\ast}(\mathcal{B}_3^+(\bar y))}    
 + [v]_{C^{2,\epsilon}_{\ast}(Q_+ \setminus \mathcal{B}_3^+(\bar y))}\right),\\
\text{ where } & P_{\bar y}(y)=\p_nv(\bar y)y_n  +  \sum_{i=1}^{n-1}\p_{in}v(\bar y)(y_i-\bar y_i) y_n +\frac{1}{6}\p_{nnn}v(\bar y)y_n^3\\
& \qquad +\frac{1}{2}\p_{n,n+1,n+1}v(\bar y)y_ny_{n+1}^2.
\end{align*}
\end{defi}

Let us discuss these function spaces $X_{\alpha,\epsilon}$: They are subspaces of the Baouendi-Grushin H\"older spaces $C^{2,\epsilon}_\ast(Q_+)$, with the additional properties that these functions are $C^{3,\alpha}_\ast$ along the edge $P$ and that they satisfy the symmetry conditions $v=0$ on $\{y_n=0\}$ and $\p_{n+1}v=0$ on $\{y_{n+1}=0\}$. The condition $\p_{nn}v=0$ on $P$ is a necessary compatibility condition which ensures that $\Delta_G$ maps $X_{\alpha,\epsilon}$ to $Y_{\alpha,\epsilon}$. The boundary conditions together with the $C^{3,\alpha}_\ast$ regularity allow us to conclude that any function in $X_{\alpha,\epsilon}$ has the same type of asymptotic expansion at $P$ as the Legendre function $v$. The support condition on $\Delta_G v$ together with the decay condition at infinity ($v\in C_0(Q_+)$ is a continuous function in $Q_+$ vanishing at infinity) is to ensure that $(X_{\alpha,\epsilon},\|\cdot\|_{X_{\alpha,\epsilon}})$ is a Banach space. \\
The spaces are the ones in which we apply the Banach implicit function theorem later in Section~\ref{sec:IFT1}. They are constructed in such a way as to
\begin{itemize}
\item[(i)] mimic the asymptotics behavior of our Legendre functions (which are defined in (\ref{eq:legendre})) around the straightened regular free boundary $P$. In particular, the Legendre functions $v$ associated with solutions $w$ of (\ref{eq:varcoeff}) are contained in the spaces $X_{\alpha,\epsilon}$ for a suitable range of $\alpha,\epsilon$ (c.f. Proposition \ref{prop:error_gain2}).
\item[(ii)] The spaces are compatible with the mapping properties of the fully nonlinear, degenerate, (sub)elliptic operator $F$ from (\ref{eq:nonlineq1}) (c.f. Proposition \ref{prop:nonlin_map}).
\item[(iii)] They are compatible with the linearization of the operator $F$ (c.f. Proposition \ref{prop:linear}). In particular they allow for ``Schauder type'' estimates for the Baouendi-Grushin Laplacian.
\end{itemize}

\begin{rmk}
\label{rmk:homo}
\begin{itemize}
\item[(i)] We note that by our support assumptions 
\begin{align*}
& \|d_G(\cdot,\bar y)^{-(1+2\alpha)}(f-P_{\bar y})\|_{L^{\infty}(\mathcal{B}_3^+(\bar y))} \\
& \quad \leq C [d_G(\cdot,\bar y)^{-(1+2\alpha-\epsilon)}(f-P_{\bar y})]_{\dot{C}^{0,\epsilon}_\ast(\mathcal{B}_3^+(\bar y))} .
\end{align*}
Similarly, by interpolation, we control all intermediate Hölder norms of $v-P_{\bar{y}}$ by $\| v\|_{X_{\alpha,\epsilon}}$.
\item[(ii)] We remark that the norms of $X_{\alpha,\epsilon}$ and $Y_{\alpha,\epsilon}$ only contain homogeneous contributions and do not include the lower order contributions which would involve the norms of the approximating polynomials $p(y):= \sum\limits_{|\alpha|\leq k} a_{\alpha}y^{\alpha}\in \mathcal{P}_k^{hom}$:
$$|p|_{k}:= \sum\limits_{\beta}|a_{\beta}| .$$ Yet, this results in Banach spaces as additional support conditions are imposed on $f, \D_G v$. The Banach space property is shown in Lemma~\ref{lem:Banach} in the Appendix.
\end{itemize}
\end{rmk}

For locally defined functions we use the following spaces:

\begin{defi}[Local function spaces]
\label{defi:spaces_loc}
Given $\alpha,\epsilon\in (0,1]$ and $R>0$.
\begin{align*}
X_{\alpha,\epsilon}(\mathcal{B}_R^+):=&\{v\in C^{2,\epsilon}_\ast(\mathcal{B}_R^+)| v\text{ is } C^{3,\alpha}_\ast \text{ at } P\cap \mathcal{B}_R, \\
&v=0 \text{ on } \{y_n=0\}\cap \mathcal{B}_R,\ \p_{n+1}v=0\text{ on }\{y_{n+1}=0\}\cap \mathcal{B}_R, \\  
&\p_{nn}v=0\text{ on } P\cap \mathcal{B}_R \text{ and } \|v\|_{X_{\alpha,\epsilon}(\mathcal{B}_R^+)}<\infty\},
\end{align*}
where
\begin{align*}
\|v\|_{X_{\alpha,\epsilon}(\mathcal{B}_R^+)}:=\sup_{\bar y\in P\cap \mathcal{B}_R}\left(\sum_{i,j=1}^{n+1}[d_G(\cdot, \bar y)^{-(1+2\alpha-\epsilon)}Y_iY_j(v-P_{\bar y})]_{\dot{C}^{0,\epsilon}_\ast(\mathcal{B}_3^+(\bar y)\cap \mathcal{B}_R^+)}\right.\\
\left.+\|d_G(\cdot, \bar y)^{-(3+2\alpha)}(v-P_{\bar y})\|_{L^{\infty}(\mathcal{B}_3^+(\bar y)\cap \mathcal{B}_R^+)}
+ |P_{\bar y}|_{3}\right),
\end{align*}
with $P_{\bar y}$ being as in Definition~\ref{defi:spaces}.\\ 
Similarly,
\begin{align*}
Y_{\alpha,\epsilon}(\mathcal{B}_R^+):=&\{f\in C^{0,\epsilon}_\ast(\mathcal{B}_R^+)| f \text{ is } C^{1,\alpha}_\ast \text{ at } P\cap \mathcal{B}_R,\\
&f=\p_{n+1}f=0\text{ on }P\cap \mathcal{B}_R \text{ and } \|f\|_{Y_{\alpha,\epsilon}(\mathcal{B}_R^+)}<\infty\},
\end{align*}
where
\begin{align*}
\|f\|_{Y_{\alpha,\epsilon}(\mathcal{B}_R^+)}:=\sup_{\bar y\in P\cap \mathcal{B}_R}\left(\| d_G(\cdot, \bar y)^{-(1+2\alpha)}(f-P_{\bar y})\|_{L^{\infty}(\mathcal{B}_3^+(\bar y)\cap \mathcal{B}_R^+)} \right.\\
\left. + [d_G(\cdot, \bar y)^{-(1+2\alpha)}(f-P_{\bar y})]_{\dot{C}^{0,\epsilon}_\ast(\mathcal{B}_3^+(\bar y)\cap \mathcal{B}_R^+)} +|P_{\bar y}|_1 \right),
\end{align*}
with $P_{\bar y}$ being as in Definition~\ref{defi:spaces}.
\end{defi}

For the functions in $X_{\alpha,\epsilon}$ and $Y_{\alpha,\epsilon}$, the following characterization will be useful. We postpone the proof to the Appendix, Section \ref{sec:decomp}.

\begin{prop}[Characterization of $X_{\alpha,\epsilon}$ and $Y_{\alpha,\epsilon}$]
\label{prop:decompI}
Let $v\in X_{\alpha,\epsilon}$ and $f\in Y_{\alpha,\epsilon}$ and $2\alpha>\epsilon$. Let $r=r(y):=\sqrt{y_n^2+y_{n+1}^2}$ denote the distance from $y$ to $P$. Let $y'':=(y'',0,0)\in P$.
\begin{itemize}
\item[(i)] Then $\p_n f(y'')\in C^{0,\alpha}(P)$. Moreover, there exists $f_1(y)\in C^{0,\epsilon}_{\ast}(Q_+)$ vanishing on $P$, such that for $y\in \mathcal{B}_3^+$
\begin{align*}
f(y)=\p_n f(y'')y_n+ r^{1+2\alpha-\epsilon}f_1(y).
\end{align*}
\item[(ii)] Then $\p_nv(y'')\in C^{1,\alpha}(P)$, $\p_{nnn}v(y''), \p_{n,n+1,n+1}v(y'') \in C^{0,\alpha}(P)$. Moreover, there exist functions $C_1,V_i, C_{ij}\in C^{0,\epsilon}_{\ast}(Q_+)$, $i,j\in\{1,\dots,n+1\}$, vanishing on $P$, such that for $y\in \mathcal{B}_3^+$
\begin{align*}
v(y) &= \p_n v(y'') y_n  + \frac{\p_{nnn}v(y'')}{6}y_n^3 +\frac{\p_{n,n+1,n+1}v(y'')}{2}y_ny_{n+1}^2+ r^{3+2\alpha-\epsilon}C_1(y), \\
\p_{i}v(y)&= \p_{in}v(y'') y_n + r^{1+2\alpha-\epsilon}V_i(y),\quad i\in \{1,\dots, n-1\},\\
\p_{n}v(y)& = \p_{n}v(y'') + \frac{\p_{nnn}v(y'')}{2}y_n^2 + \frac{\p_{n,n+1,n+1}v(y'')}{2}y_{n+1}^2 + r^{2+2\alpha-\epsilon}V_n(y),\\
\p_{n+1}v(y) & = \p_{n,n+1,n+1}v(y'') y_n y_{n+1} + r^{2+2\alpha-\epsilon}V_{n+1}(y),\\
\p_{ij}v(y)&=r^{-1+2\alpha-\epsilon}C_{ij}(y),\\
\p_{in}v(y)&= \p_{in}v(y'')  + r^{2\alpha-\epsilon}C_{in}(y),\\
\p_{i,n+1}v(y)&=r^{2\alpha-\epsilon}C_{i,n+1}(y),\\
\p_{n,n}v(y)&= \p_{nnn}v(y'') y_n + r^{1+2\alpha-\epsilon}C_{n,n}(y),\\
\p_{n,n+1}v(y)&=\p_{n,n+1,n+1}v(y'')y_{n+1}+r^{1+2\alpha-\epsilon}C_{n,n+1}(y),\\
\p_{n+1,n+1}v(y)&=\p_{n,n+1,n+1}v(y'')y_{n} + r^{1+2\alpha-\epsilon}C_{n+1,n+1}(y).
\end{align*}
\end{itemize}
Moreover, $C_1(y)=0=V_{n+1}(y)$ on $\{y_n=0\}$.
For the decompositions in (i) and (ii) we have
\begin{align*}
&[\p_{n}f]_{\dot{C}^{0,\alpha}(P\cap \mathcal{B}_3^+)}+ [f_1]_{\dot{C}^{0,\epsilon}_\ast(\mathcal{B}_3^+)}\leq C \|f\|_{Y_{\alpha,\epsilon}},\\
&[\p_{in}v]_{\dot{C}^{0,\alpha}(P\cap \mathcal{B}_3^+)} +[\p_{nnn}v]_{\dot{C}^{0,\alpha}(P\cap \mathcal{B}_3^+)}+[\p_{n,n+1,n+1}v]_{\dot{C}^{0,\alpha}(P\cap \mathcal{B}_3^+)}\\
& \quad + \sum\limits_{i,j=1}^{n+1}[C_{ij}]_{\dot{C}^{0,\epsilon}_\ast(\mathcal{B}_3^+)}
+ \sum\limits_{j=1}^{n+1}[V_{j}]_{\dot{C}^{0,\epsilon}_\ast(\mathcal{B}_3^+)}
 + [C_1]_{\dot{C}^{0,\epsilon}_\ast(\mathcal{B}_3^+)} \leq C \|v\|_{X_{\alpha,\epsilon}}.
\end{align*}
\end{prop}

\begin{rmk}
\label{rmk:characterize}
It is immediate that any $f\in C^{0,\epsilon}_\ast(Q_+)$ with $\supp(f)\subset \mathcal{B}_3^+$ which satisfies the decomposition in (i) is in  $Y_{\alpha,\epsilon}$. Moreover, 
$$\|f\|_{Y_{\alpha,\epsilon}}\leq C\left([\p_{n}f]_{\dot{C}^{0,\alpha}(P)}+ [f_1]_{\dot{C}^{0,\epsilon}_\ast(\mathcal{B}_3^+)}\right).$$ 
Similarly, it is not hard to show that functions $v$ satisfying the decomposition in (ii) with $\supp(\Delta_Gv)\subset \mathcal{B}_3^+$ are in $X_{\alpha,\epsilon}$. In particular, this implies that Proposition \ref{prop:decompI} gives an equivalent characterization of the spaces $X_{\alpha,\epsilon}$ and $Y_{\alpha,\epsilon}$.\\
Motivated by the decomposition of Proposition \ref{prop:decompI}, we sometimes also write 
\begin{align*}
Y_{\alpha,\epsilon} = y_n C^{0,\alpha} + r^{1+2\alpha- \epsilon} C_{\ast}^{0,\epsilon}.
\end{align*} 
\end{rmk}

At the end of this section, we state the following a priori estimate (which should be viewed as a Schauder type estimate for the Baouendi-Grushin operator):

\begin{prop}
\label{prop:invert}
Let $\alpha \in (0,1)$, $\epsilon \in (0,1)$, $v\in X_{\alpha,\epsilon}$, $f\in Y_{\alpha,\epsilon}$ and 
\begin{align*}
\D_G v = f.
\end{align*}
Then we have
\begin{align*}
\| v\|_{X_{\alpha,\epsilon}} \leq C \|f\|_{Y_{\alpha,\epsilon}}.
\end{align*}
\end{prop}

\begin{rmk}
We note that the a priori estimates exemplify a scaling behavior which depends on the support of the respective function. More precisely, let $=\D_G v$ be supported in $\mathcal{B}_{\mu}^+$ for some $\mu>0$. Then,
\begin{align*}
\|v\|_{X_{\alpha,\epsilon}} \leq C(1+|\mu|^{1+2\alpha-\epsilon}) \|\D_G v\|_{Y_{\alpha,\epsilon}}.
\end{align*}
\end{rmk}

The proof of Proposition \ref{prop:invert} follows by exploiting the scaling properties of our operator and polynomial approximations. This method is in analogy to the Campanato approach (c.f. \cite{Ca64}) to prove Schauder estimates for the Laplacian (c.f. also \cite{Gia}, \cite{Wang92} and \cite{Wa03} for generalizations to elliptic systems, fully nonlinear elliptic and parabolic equations and certain subelliptic equations). Our generalization of these spaces is adapted to our thin free boundary problem. We postpone the proof of Proposition \ref{prop:invert} to the Appendix, Section \ref{sec:quarter_Hoelder}.

\section[Regularity of the Legendre Function]{Regularity of the Legendre Function for $C^{k,\gamma}$ Metrics}
\label{sec:improve_reg}

In this Section we return to the investigation of the Legendre function. We recall that in Section~\ref{sec:Legendre} we transformed the free boundary problem (\ref{eq:varcoeff}) into a fully nonlinear Baouendi-Grushin type equation for the Legendre function $v$. In this section, we will study the regularity of the Legendre function $v$ in terms of the function spaces $X_{\alpha,\epsilon}$ from Section~\ref{sec:holder}.\\
In the whole section we assume that the metrics $a^{ij}$ are $C^{1,\gamma}$ Hölder regular for some $\gamma\in (0,1)$. We start by showing that the Legendre function $v$ (associated with a solution $w$ to \eqref{eq:thin_obst}) is in the space $X_{\alpha,\gamma}$ introduced in Section~\ref{sec:holder}.  This is a consequence of transferring the asymptotics of $w$ (which were derived in Propositions \ref{prop:asym2}, \ref{prop:improved_reg} in Section~\ref{sec:asymp}) to $v$. Here $\alpha$ is the (a priori potentially very small) H\"older exponent of the free boundary from Section~\ref{sec:asymp}.\\
In Section~\ref{subsec:improvement} (c.f. Propositions \ref{prop:error_gain}, \ref{prop:error_gain2}), we exploit the structure of our nonlinear equation to improve the regularity of $v$ from $X_{\alpha,\epsilon}$ to $X_{\delta,\epsilon}$ for any $\delta\in (0,1)$ and $\epsilon \in (0,\gamma]$. In particular this implies that the free boundary $\Gamma_w$ is $C^{1,\delta}$ regular for any $\delta\in (0,1)$. 

\subsection{Asymptotics of the Legendre function}
\label{sec:Leg}
Throughout this section we assume that $w$ solves the thin obstacle problem \eqref{eq:varcoeff} with coefficients $a^{ij}\in C^{k,\gamma}$ for $k\geq 1$ and $\gamma\in (0,1]$. Moreover, we always assume that the conditions (A1)-(A7) are satisfied. We further recall from Section~\ref{sec:Legendre} that the Legendre function $v$ is originally defined on $U=T(B_{1/2}^+)$. However, after a  rescaling procedure we  may assume that $v$ is defined in $\mathcal{B}_2^+$. \\

We first rewrite the asymptotics of $w$ (stated in Corollary~\ref{cor:improved_reg}) in terms of the corresponding Legendre function $v$.

\begin{prop}\label{prop:holder_v}
Let $a^{ij}\in C^{k,\gamma}$ with $k\geq 1$ and $\gamma\in (0,1)$. Let $w$ be a solution to the variable coefficient thin obstacle problem and let $v$ be its Legendre function as defined in \eqref{eq:legendre}. Assume that $w$ satisfies the asymptotic expansion in Proposition~\ref{prop:improved_reg} with some $\alpha\in (0,1)$.
Then for any $\hat y\in \mathcal{B}_1^+$ with $\sqrt{\hat y_n^2+\hat y_{n+1}^2}=:\lambda \in (0,1)$, and any multi-index $\beta$ with $|\beta|\leq k+1$
\begin{align*}
\left[D^{\beta}v-D^{\beta} v_{\hat y''}\right]_{\dot{C}^{0,\gamma}_\ast(\mathcal{B}_{\lambda/4}^+(\hat y))} \leq C(|\beta|,n,p) \max\{\epsilon_0, c_*\} \lambda^{3+2\alpha-\gamma-2|\beta''|-|\beta_n|-|\beta_{n+1}|},
\end{align*}
where $v_{\hat y''}$ is the leading order expansion of $v$ at $\hat y''$ as defined in Lemma~\ref{lem:asymp_profile}.
\end{prop}

\begin{proof}
We argue in three steps in which we successively simplify the problem:\\

\emph{Step 1: First reduction -- Scaling.} Given $\hat y\in \mathcal{B}_1^+$ with $\sqrt{\hat y_n^2+\hat y_{n+1}^2}=\lambda>0$, we project $\hat y$ onto $P=\{y_n=y_{n+1}=0\}$ and let $\hat y''=(\hat y'',0,0)$ denote the projection point.
Let 
$$\hat v_{\hat y'',\lambda}(\zeta):=\frac{v(\hat y''+\delta_\lambda( \zeta))-\p_nv(\hat y'')(\lambda \zeta_n)}{\lambda^3}, $$
with $\delta_\lambda (\zeta)=(\lambda^2\zeta'',\lambda \zeta_n,\lambda \zeta_{n+1})$. 
We note that 
$\hat v_{\hat y'',\lambda}$ is the Legendre function for $w_{x_0,\lambda^2}(\xi):=w(x_0+\lambda^2\xi)/\lambda^3$ with $x_0=T^{-1}(\hat y'')$ . We set 
$$\hat \zeta=\left(0,\frac{\hat y_n}{\lambda}, \frac{\hat y_{n+1}}{\lambda}\right)\in B''_1\times \mathcal{S}^1.$$ 
With this rescaling, it suffices to show that
\begin{align}\label{eq:deri_v}
\left[ D^{\beta} \hat v_{\hat y'',\lambda}- D^{\beta} \hat v_{\hat y''} \right]_{\dot{C}^{0,\gamma}_\ast(\mathcal{B}_{1/4}^+(\hat \zeta))}\leq C(|\beta|,n,p)  \max\{\epsilon_0, c_*\} \lambda^{2\alpha},
\end{align}
where $\hat v_{\hat{y}''}$ denotes the Legendre function for $w_{x_0}(\xi)=\lim_{\lambda\rightarrow 0_+} w_{x_0,\lambda^2}(\xi)$. 
Indeed, the conclusion of Proposition \ref{prop:holder_v} then follows by undoing the rescaling in \eqref{eq:deri_v}.\\

\emph{Step 2: Second reduction.} Given any multi-index $\beta$ with $|\beta |\leq k+1$, it is possible to express $D^{\beta}\hat v_{\hat y'',\lambda}$ as a function of $w_{x_{0},\lambda^2}$ and its derivatives: 
$$D^\beta \hat v_{\hat y'',\lambda}(y)= F_\beta(D^{\tilde{\alpha}} w_{x_{0},\lambda^2}(x))\big|_{x=(T^{w_{x_0,\lambda^2}})^{-1}(y)},\quad |\tilde{\alpha}|\leq |\beta|.$$ 
Here $F_\beta$ is an analytic function on the open set $\{J(w_{x_{0},\lambda^2})\neq 0\}$. Let $\hat \xi=(T^{w_{x_{0},\lambda^2}})^{-1}(\hat \zeta)$. We note that a sufficiently small choice of $\lambda_0\in(0,1)$ implies that our change of coordinates is close to the square root mapping (c.f. Proposition \ref{prop:asym2}). Combining this with the observation that on scales of order one (i.e. when $\zeta_n^2+\zeta_{n+1}^2\sim 1$) the Baouendi-Grushin metric is equivalent to the Euclidean metric, results in the inclusions $(T^{w_{x_{0},\lambda^2}})^{-1}(\mathcal{B}_{1/4}(\hat \zeta)), (T^{w_{x_{0}}})^{-1}(\mathcal{B}_{1/4}(\hat \zeta))\subset B_{1/2}(\hat \xi)$ and $B_{3/4}(\hat \xi)\cap \Gamma_{w_{x_{0},\lambda^2}}=\emptyset$, $B_{3/4}(\hat \xi)\cap \Gamma_{w_{x_{0}}}=\emptyset$. Thus, to show \eqref{eq:deri_v}, it then suffices to prove
\begin{align}\label{eq:multi_deri}
\left[F_\beta(D^{\tilde{\alpha}} w_{x_{0},\lambda^2} )- F_\beta(D^{\tilde{\alpha}} \mathcal{W}_{x_0}(x_0 +\cdot))\right]_{\dot{C}^{0,\gamma}(B_{1/2}^+(\hat \xi))} \leq C(|\beta|,n,p)  \max\{\epsilon_0, c_*\} \lambda^{2\alpha},
\end{align}
and 
\begin{align}\label{eq:inverse_T}
\|(T^{w_{x_0,\lambda^2}})^{-1}-(T^{w_{x_0}})^{-1}\|_{C^{1}(\mathcal{B}^+_{3/4})}\leq C  \max\{\epsilon_0, c_*\} \lambda^{2\alpha}
\end{align}
for $\lambda \in (0,\lambda_0)$.
Indeed, once we have obtained \eqref{eq:multi_deri}-\eqref{eq:inverse_T}, \eqref{eq:deri_v} follows by the equivalence of the Euclidean and Baouendi-Grushin geometries at scales of order one and a triangle inequality.\\

\emph{Step 3: Proof of (\ref{eq:multi_deri}) and conclusion.}
To show \eqref{eq:multi_deri}, we use a Taylor expansion and write
$$F_\beta(D^{\tilde{\alpha}} w_{x_{0},\lambda^2})-F_\beta(D^{\tilde{\alpha}} \mathcal{W}_{x_0}(x_0 + \cdot))=R_{\tilde{\alpha}} (w_{x_{0},\lambda^2})(D^{\tilde{\alpha}} w_{x_{0},\lambda^2}-D^{\tilde{\alpha}} \mathcal{W}_{x_0}(x_0 + \cdot)),$$
where
$$R_{\tilde{\alpha}}(w_{x_{0},\lambda^2})(x)=\int_0^1\p_{m_{\tilde{\alpha}}}F_\beta(tD^{\tilde{\alpha}}w_{x_{0},\lambda^2}(x)+(1-t)D^{\tilde{\alpha}}\mathcal{W}_{x_0}(x_0 + x))dt.$$
Since $-C\leq J(tw_{x_{0},\lambda^2}+(1-t)\mathcal{W}_{x_0})\leq -c$ in $B_{1/2}^+(\hat \xi)$ and since $w_{x_{0},\lambda^2}, \mathcal{W}_{x_0}(x_0 + \cdot) \in C^{k+1,\gamma}(B_{1/2}^+(\hat \xi))$, we have that $R_{\tilde{\alpha}}(w_{x_{0},\lambda^2})\in C^{0, \gamma}(B_{1/2}^+(\hat \xi))$ with uniform bounds in $\lambda$. 
Next we recall that by Corollary \ref{cor:improved_reg} 
$$\left[D^{\tilde{\alpha}}w_{x_{0},\lambda^2}-D^{\tilde{\alpha}}\mathcal{W}_{x_0}(x_0 + \cdot) \right]_{\dot{C}^{0,\gamma}(B_{1/4}^+(\hat{\xi}))} \leq C(\beta)\max\{ \epsilon_0, c_{\ast}\} \lambda^{2\alpha}.$$ Combining this with the fact that $R_{\tilde{\alpha}}(w_\tau)(\xi)\in C^{0, \gamma}(B_{1/4}^+(\hat \xi))$, yields \eqref{eq:multi_deri}.
To show \eqref{eq:inverse_T}, we first observe that $\|T^{w_{x_0,\lambda^2}}-T^{w_{x_0}}\|_{C^1(B_{3/4}(\hat \xi))}\leq C  \max\{\epsilon_0, c_*\}\lambda^{2\alpha}$ by Corollary~\ref{cor:improved_reg} and the definition of $T$. Then using the uniform boundedness of $\|D (T^{w_{x_0,\lambda^2}})^{-1}\|_{L^\infty(\mathcal{B}^+_{1/2}(\hat \xi))}$ and $\|D (T^{w_{x_0}})^{-1}\|_{L^\infty(\mathcal{B}^+_{1/2}(\hat \xi))}$ we obtain the desired estimate.
\end{proof}

Using the spaces from Definition \ref{defi:spaces}, we apply Proposition \ref{prop:holder_v} to quantify the regularity and asymptotics of our Legendre function:

\begin{prop}
\label{prop:regasymp}
Under the assumptions of Proposition~\ref{prop:holder_v} we have $v\in X_{\alpha, \mu}(\mathcal{B}_{1}^+)$ 
for all $\mu\in(0,\gamma]$. In particular, for $y_0\in P\cap \mathcal{B}_{1}$ there exist functions $C_{k\ell}\in C^{0,\gamma}_{\ast}(\mathcal{B}_{1}^+(y_0))$ with $C_{k l}(y'',0,0)=0$ for all $k,l \in \{1,\dots,n+1\} $ such that the following asymptotics are valid:
\begin{align*}
\p_{ij}v(y)& = (y_n^2+y_{n+1}^2)^{-1}d_G(y,y_0)^{1+2\alpha-\gamma}C_{ij}(y),\\
\p_{in}v(y)&=\frac{(e_i\cdot \nu_{T^{-1}(y_0)})}{(e_n\cdot \nu_{T^{-1}(y_0)})}+d_G(y,y_0)^{2\alpha-\gamma}C_{in}(y),\\
\p_{i,n+1}v(y)&= d_G(y,y_0)^{2\alpha-\gamma}C_{i,n+1}(y),\\
\begin{pmatrix}
\partial_{nn} v(y) & \partial_{n,n+1} v(y)\\
\partial_{n+1,n} v(y) & \partial_{n+1,n+1}v(y)
\end{pmatrix}
&=\begin{pmatrix} a_0(y_0) y_n & a_1(y_0)y_{n+1}\\
a_1(y_0)y_{n+1}& a_1(y_0)y_n
\end{pmatrix} \\
& \quad + d_G(y,y_0)^{1+2\alpha- \gamma}\begin{pmatrix} C_{nn}(y)& C_{n,n+1}(y)\\
C_{n,n+1}(y)& C_{n+1,n+1}(y)
\end{pmatrix} .
\end{align*}
Here $i,j\in\{1,\dots,n-1\}$.
\end{prop}

\begin{rmk}
We emphasize that in the expression for $\p_{ij}v$ it is not possible to replace $(y_n^2 + y_{n+1}^2)^{-1}$ by $d_G(y,y_0)^{-2}$. This is in analogy with Remark \ref{rmk:improved_reg}.
\end{rmk}

\begin{proof}
For each $y_0\in P$ the proof of the asymptotics in the non-tangential cone $\mathcal{N}_G(y_0)$ follows directly from Proposition~\ref{prop:holder_v}, Lemma \ref{lem:asymp_profile} (which yields the explicit expressions of the leading order asymptotic expansions) and a chain of balls argument. In order to obtain the asymptotic expansion in the whole of $\mathcal{B}_{1}^+(y_0)$, we use the regularity of the coefficient functions and the triangle inequality. We only present the argument for $\p_{in}v$, since the reasoning for the other partial derivatives is analogous. Hence, let $y\in \mathcal{B}_{1}^+(y_0)$ but $y\notin \mathcal{N}_G(y_0)$. Let $\bar{y}\in P\cap \mathcal{B}_1$ denote the projection of $y$ onto $P$. Then, by the triangle inequality, we may assume that $d_G(y,\bar{y}), d_G(\bar{y},y_0)\leq C d_G(y,y_0)$. Thus, by virtue of the regularity of $\nu_{T^{-1}(y_0)}$, we have that
\begin{align*}
\left|\p_{in}v(y)- \frac{(e_i\cdot \nu_{T^{-1}(y_0)})}{(e_n\cdot \nu_{T^{-1}(y_0)})} \right| &\leq \left| \p_{in}v(y) -\frac{(e_i\cdot \nu_{T^{-1}(\bar{y})})}{(e_n\cdot \nu_{T^{-1}(\bar{y})})}\right| \\
& \quad + \left| \frac{(e_i\cdot \nu_{T^{-1}(\bar{y})})}{(e_n\cdot \nu_{T^{-1}(\bar{y})})} -\frac{(e_i\cdot \nu_{T^{-1}({y_0})})}{(e_n\cdot \nu_{T^{-1}({y_0})})} \right|\\
& \leq C d_G(y,\bar{y})^{2\alpha} +C|\bar{y}- y_0|^{\alpha} \leq C d_G(y,y_0)^{2\alpha}.
\end{align*}
The Hölder estimates are analogous.
\end{proof}

\begin{rmk}
\label{rmk:close}
For later reference we conclude this section by noting that the closeness condition (A5) (and the asymptotics from Proposition \ref{prop:holder_v}) implies that 
$$\left\| v- v_{0}\right\|_{X_{\alpha,\epsilon}(\mathcal{B}_{1}^+)} \leq C\max\{\epsilon_0,c_\ast\},$$ 
where $v_0(y)= -\frac{1}{3}(y_n^3-3y_ny_{n+1}^2)$ is the leading order expansion of $v$ at the origin.
This will be used in the perturbation argument in Section \ref{sec:grushin} and in the application of the implicit function theorem in Section \ref{sec:IFT1}.
\end{rmk}

\subsection{Improvement of regularity}
\label{subsec:improvement}
In this section we present a bootstrap argument to infer higher regularity of the Legendre function. By virtue of the previous section, we have that $v\in X_{\alpha,\epsilon}(\mathcal{B}_1^+)$ for some potentially very small value of $\alpha\in(0,\gamma]$. In this section we improve this regularity modulus further by showing that $\alpha$ can be chosen arbitrarily close to one. To this end we argue in two steps: By an expansion, we first identify the structure of $F$ in terms of a leading order linear operator and additional higher order controlled contributions (c.f. Proposition \ref{prop:error_gain}). Then in a second step, we use this to bootstrap regularity (c.f. Proposition \ref{prop:error_gain2}). \\

In the sequel we use the following abbreviations:
\begin{align*}
G^{ij}(v)&:=-\det\begin{pmatrix}
\p_{ij}v& \p_{in}v & \p_{i,n+1}v\\
\p_{jn}v& \p_{nn}v & \p_{n,n+1}v\\
\p_{j,n+1}v & \p_{n,n+1}v &\p_{n+1,n+1}v
\end{pmatrix}, \ i,j\in\{1,\dots,n-1\},\\
G^{i,n}(v)&:=2\det\begin{pmatrix}
\p_{in}v & \p_{i,n+1}v\\
\p_{n,n+1}v & \p_{n+1,n+1}v
\end{pmatrix}, \ i\in\{1,\dots,n-1\},\\
G^{i,n+1}(v)&:=2\det\begin{pmatrix}
\p_{i,n+1}v & \p_{in}v\\
\p_{n,n+1}v & \p_{nn}v
\end{pmatrix}, \ i\in\{1,\dots,n-1\},\\
G^{n,n}(v)&:=\p_{n+1,n+1}v,\\
G^{n+1,n+1}(v)&:=\p_{nn}v,\\
G^{n,n+1}(v)&:=-\p_{n,n+1} v,\\
J(v)&:=\det\begin{pmatrix}
\p_{nn}v &\p_{n,n+1}v\\
\p_{n,n+1}v &\p_{n+1,n+1}v
\end{pmatrix}.
\end{align*}
With slight abuse of notation, we thus interpret $G^{ij}$ and $J$ as functions from the symmetric matrices $\R^{(n+1)\times(n+1)}_{sym}$ to $\R$ and recall the notation for partial derivatives of $G^{ij}$ with respect to the components $m_{k\ell}$ from Section \ref{sec:notation}:
\begin{align*}
\p_{m_{k\ell}}G^{ij}(M)=\frac{\p G^{ij}(M)}{\p m_{k\ell}}, \ M=(m_{k\ell})\in \R^{(n+1)\times (n+1)}_{sym}.
\end{align*}
With these conventions, the nonlinear equation \eqref{eq:nonlineq1} from Section~\ref{sec:Legendre}, which is satisfied by $v$, turns into
\begin{align}\label{eq:nonlin_2}
F(v,y):=\sum_{i,j=1}^{n+1}\tilde{a}^{ij}(y)G^{ij}(v)-J(v)\left(\sum_{j=1}^{n-1}\tilde{b}^j\p_j v+\tilde{b}^ny_n+\tilde{b}^{n+1}y_{n+1}\right)=0.
\end{align}

Relying on this structure, we derive a (self-improving) linearization. More precisely,
for each $y_0\in P\cap \mathcal{B}_{1}^+$, we will linearize the equation at $v_{y_0}$, where
\begin{equation}\label{eq:v0}
\begin{split}
v_{y_0}(y)&=\p_nv(y_0)y_n+\sum_{i=1}^{n-1}\p_{in}v(y_0)(y_i-(y_0)_i)y_n\\
&+\frac{\p_{nnn}v(y_0)}{6}y_n^3+\frac{\p_{n,n+1,n+1}v(y_0)}{2}y_ny_{n+1}^2,
\end{split}
\end{equation}
is the up to order three asymptotic expansion of $v$ at $y_0$. 
The linearization then leads to the following self-improving structure:

\begin{prop}\label{prop:error_gain}
Let $a^{ij}(x)\in C^{1,\gamma}$ for some $\gamma\in (0,1]$.
Assume that $v\in X_{\alpha,\epsilon}(\mathcal{B}_{1}^+)$ with $\alpha\in (0,1]$ solves $F(v,y)=0$. Then at each point $y_0\in P\cap \mathcal{B}_{1/2}^+$, we have the following expansion in $\mathcal{B}_r^+(y_0)$ with $0<r<1/2$: 
\begin{align*}
F(v,y)=L_{y_0}v + P_{y_0}(y)+E_{y_0}(y).
\end{align*}
Here 
$$L_{y_0}v=\tilde{a}^{ij}(y_0)\p_{m_{k\ell}}G^{ij}(v_{y_0})\p_{k\ell}v=D_vF\big|_{(v,y)=(v_{y_0},y_0)} v,$$  
$$P_{y_0}(y)=\tilde{a}^{ij}(y_0)\left(G^{ij}(v_{y_0})-\p_{m_{k\ell}}G^{ij}(v_{y_0})\p_{k\ell}v_{y_0}\right)\in \mathcal{P}_{1}^{hom},$$
$P_{y_0}(y)$ is of the form $c_0(y_0)y_n$ and $E_{y_0}(y)$ is an error term satisfying
\begin{equation*}
\begin{split}
&\left\|d_G(\cdot,y_0)^{-\eta_0}E_{y_0}\right\|_{L^\infty(\mathcal{B}_{1/2}^+(y_0))}+\left[d_G(\cdot,y_0)^{-(\eta_0-\epsilon)}E_{y_0}\right]_{\dot{C}^{0,\epsilon}_\ast(\mathcal{B}_{1/2}^+(y_0))}\leq C,
\end{split}
\end{equation*}
for $\eta_0=\min\{1+4\alpha,3\}$.
\end{prop}

\begin{rmk}[The role of $\alpha$, $2\alpha$ and $4\alpha$]
As already seen in Proposition~\ref{prop:decompI}, the parameter $\alpha$ in the space $Y_{\alpha,\epsilon}$ refers to the (tangential) H\"older regularity (w.r.t. the Euclidean metric) of the quotient $\frac{f}{y_n}\big|_P$ for any function $f\in Y_{\alpha,\epsilon}$. The parameter $2\alpha$ originates from the different scalings of the Euclidean metrics and Baouendi-Grushin metrics (c.f. Remark~\ref{rmk:equi_dist}). More precisely, for any $y''_1,y''_2\in P$, $|y''_1-y''_2|\sim d_G(y''_1,y''_2)^{2}$ by Remark~\ref{rmk:equi_dist}, which accounts for the $2\alpha$ in the definition of the norm $\|f\|_{Y_{\alpha,\epsilon}}=\sup_{\bar y\in P}[d_G(\cdot, \bar y)^{-(1+2\alpha-\epsilon)}(f-P_{\bar y})]_{\dot{C}^{0,\epsilon}_\ast}$. The parameter $4\alpha$ in Proposition~\ref{prop:error_gain} indicates an \emph{improvement} of the tangential regularity of $L_{y_0}v/y_n$ at $y_0$ from $C^{0,\alpha}$ to $C^{0,2\alpha}$. 
\end{rmk}

\begin{rmk}\label{rmk:error_gain3}
By using the explicit expression of $v_{y_0}$ and $G^{ij}(v)$, it is possible to compute the form of the leading order operator:
\begin{align*}
L_{y_0}&:=\sum_{i,j=1}^{n+1}\tilde{a}^{ij}(y_0)\p_{m_{k\ell}}G^{ij}(v_{y_0})\p_{k\ell}\\
&=\sum_{i,j=1}^{n-1}\tilde{a}^{ij}\left((A_1)^2y_{n+1}^2-A_0A_1y_n^2\right)\p_{ij}\\
&\quad +2\sum_{i,j=1}^{n-1}\tilde{a}^{ij}\left(B_jA_1y_n\p_{in}-B_jA_1y_{n+1}\p_{i,n+1}\right)+\sum_{i,j=1}^{n-1}\tilde{a}^{ij}B_iB_j\p_{n+1,n+1}\\
&\quad +2\sum_{i=1}^{n-1}\tilde{a}^{in}\left(A_1y_n\p_{in}-A_1y_{n+1}\p_{i,n+1}+B_i\p_{n+1,n+1}\right)\\
& \quad +\tilde{a}^{nn}\p_{n+1,n+1}+\tilde{a}^{n+1,n+1}\p_{n,n}.
\end{align*}
Here the coefficients $\tilde{a}^{ij}$ are evaluated at $y_0$ and $A_0,A_1,B_j$ are constants depending on $y_0$:
\begin{align*}
A_0&:=\p_{nnn}v(y_0),\quad A_1:=\p_{n,n+1,n+1}v(y_0),\\
B_j&:=\p_{jn}v(y_0),\quad j\in \{1,\dots, n-1\}.
\end{align*}
To obtain this, we have used the off-diagonal assumption (A3) for the metric $a^{ij}$, i.e. $a^{i,n+1}(x',0)=0$ for $i\in\{1,\dots, n\}$.\\
We note that the operator $L_{y_0}$ is a self-adjoint, constant coefficient Baouendi-Grushin type operator. It is hypoelliptic as an operator on $\R^{n+1}$ after an odd reflection in the $y_n$ variable and an even reflection in the $y_{n+1}$ variable (c.f. \cite{JSC87}).
\end{rmk}

\begin{proof}[Proof of Proposition~\ref{prop:error_gain}]
The proof of this result relies on a successive expansion of the coefficients and the nonlinearities.
Thus, we first expand $\tilde{a}^{ij}(y)$, $G^{ij}(v)$ and the lower order term $J(v)\left(\cdots\right)$ in \eqref{eq:nonlin_2} separately and then combine the results to derive the desired overall expansion. \\

\emph{Step 1: Expansion of the leading term.}

\emph{Step 1 a: Expansion of the coefficients.} For the coefficients $\tilde{a}^{ij}$ we have
\begin{align*}
\tilde{a}^{ij}(y)=\tilde{a}^{ij}(y_0)+E_2^{y_0,ij}(y),
\end{align*}
where
the error term $E^{y_0,ij}_2(y)$ satisfies
\begin{equation}
\label{eq:Holderweight0}
\begin{split}
\left\|d_G(\cdot,y_0)^{-2}E^{y_0,ij}_2\right\|_{L^\infty(\mathcal{B}_{1}^+(y_0))}
+\left[d_G(\cdot,y_0)^{-(2-\epsilon)}E^{y_0,ij}_2\right]_{\dot{C}^{0,\epsilon}_\ast(\mathcal{B}_{1}^+(y_0))}\leq C.
\end{split}
\end{equation}

\emph{Proof of Step 1a:}
The claim follows from the differentiability of $a^{ij}(x)$ and the asymptotics of $\nabla v$. Using the abbreviations $\xi(y):=(y'',-\p_nv(y),-\p_{n+1}v(y))$ and $\xi_0:=(y_0'',-\p_{n}v(y_0),-\p_{n+1}v(y_0))=(y_0'',-\p_nv(y_0),0)$, we obtain
\begin{equation}
\label{eq:expansion_1}
\begin{split}
\tilde{a}^{ij}(y)=a^{ij}(\xi(y))&=a^{ij}(\xi_0)+ (a^{ij}(\xi)-a^{ij}(\xi_0))\\
&=a^{ij}(\xi_0) + \int\limits_{0}^{1}\nabla_x a^{ij}((1-t)\xi_0 + t \xi(y))dt \cdot (\xi_0 - \xi(y)) .
\end{split} 
\end{equation}
Hence, (\ref{eq:expansion_1}) turns into 
\begin{equation}
\label{eq:expansion_2}
\begin{split}
\tilde{a}^{ij}&=\tilde{a}^{ij}(y_0) +E_2^{y_0,ij}(y),
\end{split} 
\end{equation}
where
\begin{align*}
E_2^{y_0,ij}(y)&:=\int\limits_{0}^{1}\nabla_x'' a^{ij}((1-t)\xi_0 + t \xi(y)) dt \cdot (y'' - y_0'')\\
&  \quad  + \sum\limits_{k=n}^{n+1}\int\limits_{0}^{1}\nabla_{x_k} a^{ij}((1-t)\xi_0 + t \xi(y)) dt \cdot (\p_k v(y) - \p_k v(y_0)).
\end{align*}
Recalling the asymptotics of $v$ from Definition~\ref{defi:spaces} for $v$, we infer that for all $y\in \mathcal{B}_{1/2}^+(y_0)$
\begin{align*}
|E_2^{y_0,ij}(y)|&\leq C \left|(y''-y''_0,-\p_nv(y) + \p_{n}v(y_0),-\p_{n+1}v(y))\right|\leq Cd_G(y,y_0)^{2}.
\end{align*}
Thus, we have shown that $d_G(y,y_0)^{-2}E_2^{y_0,ij}(y)\in L^\infty(\mathcal{B}_{1/2}^+(y_0))$.
Similarly, we also obtain $[d_G(y,y_0)^{-(2-\epsilon)}E_2^{y_0, ij}]_{\dot{C}^{0,\epsilon}(\mathcal{B}_{1/2}^+(y_0))}\leq C$. This, together with Definition~\ref{defi:spaces}, yields the second estimate in (\ref{eq:Holderweight0}).\\

\emph{Step 1 b: Expansion of the functions $G^{ij}$.} 
For the (nonlinear) functions $G^{ij}(v)$ we have for all $i,j\in\{1,\dots,n+1\}$
\begin{align*}
G^{ij}(v) = G^{ij}(v_{y_0})+\p_{m_{k \ell}}G^{ij}(v_{y_0}) \p_{k \ell} (v-v_{y_0}) + E^{y_0,ij}_1(y),
\end{align*}
where
the error $E^{y_0,ij}_{1}(y)$ satisfies the bounds
\begin{equation}
\label{eq:Holderweight}
\begin{split}
\left\|d_G(\cdot,y_0)^{-(1+4\alpha)}E^{y_0,ij}_1\right\|_{L^\infty(\mathcal{B}_{1}^+(y_0))}
+\left[d_G(\cdot,y_0)^{-(1+4 \alpha -\epsilon)}E^{y_0,ij}_1\right]_{\dot{C}^{0,\epsilon}_\ast(\mathcal{B}_1^+(y_0))}\leq C.
\end{split}
\end{equation}

\emph{Proof of Step 1b:}
To show the claim, we first expand
\begin{equation}
\label{eq:expand_v}
\begin{split}
G^{ij}(v)&=G^{ij}(v_{y_0})+\p_{m_{k\ell}}G^{ij}(v_{y_0})\p_{k\ell}(v -v_{y_0})\\
&\quad +\frac{1}{2}\p^2_{m_{k\ell}m_{\xi\eta}}G^{ij}(v_{y_0})\p_{k\ell}(v-v_{y_0})\p_{\xi\eta}(v-v_{y_0})\\
&\quad +\frac{1}{6}\p^3_{m_{k\ell}m_{\xi\eta}m_{hs}}G^{ij}(v_{y_0})\p_{k\ell}(v-v_{y_0})\p_{\xi\eta}(v-v_{y_0})\p_{hs}(v-v_{y_0})\\
&=G^{ij}(v_{y_0})+\p_{m_{k \ell}}G^{ij}(v_{y_0}) \p_{k \ell} (v-v_{y_0})+ E^{y_0,ij}_1(y).
\end{split}
\end{equation}
Here the error term is given by
\begin{align*}
E^{y_0,ij}_1(y)&=\frac{1}{2}\p^2_{m_{k\ell}m_{\xi\eta}}G^{ij}(v_{y_0})\p_{k\ell}(v-v_{y_0})\p_{\xi\eta}(v-v_{y_0})\\
&+\frac{1}{6}\p^3_{m_{k\ell}m_{\xi\eta}m_{hs}}G^{ij}(v_{y_0})\p_{k\ell}(v-v_{y_0})\p_{\xi\eta}(v-v_{y_0})\p_{hs}(v-v_{y_0}).
\end{align*}
Hence, it remains to prove the error estimate (\ref{eq:Holderweight}). To this end we estimate each term from the expression for $E_1^{y_0,ij}$ separately. We begin by observing that
\begin{equation}
\label{eq:det}
\begin{split}
 e^{y_0,ij}(y):&=\p_{m_{k\ell}m_{\xi\eta}}G^{ij}(v_{y_0})\p_{k\ell}(v-v_{y_0})\p_{\xi\eta}(v-v_{y_0})\\
&= \det \begin{pmatrix} \p_{ij}v_{y_0} & \p_{in}(v-v_{y_0}) & \p_{i,n+1}(v-v_{y_0}) \\
\p_{jn}v_{y_0} & \p_{nn}(v-v_{y_0}) & \p_{n,n+1}(v-v_{y_0}) \\
\p_{j,n+1}v_{y_0} & \p_{n,n+1}(v-v_{y_0}) & \p_{n+1,n+1}(v-v_{y_0})
  \end{pmatrix}\\
& \quad + \det \begin{pmatrix} \p_{ij}(v-v_{y_0}) & \p_{in}v_{y_0} & \p_{i,n+1}(v-v_{y_0}) \\
\p_{jn}(v-v_{y_0}) & \p_{nn}v_{y_0} & \p_{n,n+1}(v-v_{y_0}) \\
\p_{j,n+1}(v-v_{y_0}) & \p_{n,n+1}v_{y_0} & \p_{n+1,n+1}(v-v_{y_0})
  \end{pmatrix}\\
&\quad +\det \begin{pmatrix} \p_{ij}(v-v_{y_0}) & \p_{in}(v-v_{y_0}) & \p_{i,n+1}v_{y_0} \\
\p_{jn}(v-v_{y_0}) & \p_{nn}(v-v_{y_0}) & \p_{n,n+1}v_{y_0} \\
\p_{j,n+1}(v-v_{y_0}) & \p_{n,n+1}(v-v_{y_0}) & \p_{n+1,n+1}v_{y_0}
  \end{pmatrix},
  \end{split}
  \end{equation}
  and 
  \begin{align*}
 \tilde{e}^{y_0,ij}(y)&:=\p^3_{m_{k\ell}m_{\xi\eta}m_{hs}}G^{ij}(v_{y_0})\p_{k\ell}(v-v_{y_0})\p_{\xi\eta}(v-v_{y_0})\p_{hs}(v-v_{y_0})\\
 &=\left\{ 
\begin{array}{ll}
3G^{ij}(v-v_{y_0})\text{ if } i,j\in\{1,\dots, n-1\},\\
0 \quad \text{ if } i \text{ or }j\in\{n,n+1\}.
\end{array}
\right.
\end{align*}
For simplicity we only present the estimate for $e^{y_0,ij}(y)$ in detail.
To estimate the difference $\p_{k\ell}(v-v_{y_0})$ for $\ell,k\in\{1,\dots, n+1\}$ in $\mathcal{B}_{1/2}^+(y_0)$, we use the definition of our function space $X_{\alpha,\epsilon}$ (in the form of Definition~\ref{defi:spaces} or in the form of the decomposition from Proposition~\ref{prop:decompI}) to obtain that for $y\in \mathcal{B}_{1/2}^+(y_0)$
\begin{align*}
|\p_{ij}(v-v_{y_0})(y)|&\leq C d_G(y,P)^{-1+2\alpha},\\
|\p_{in}(v-v_{y_0})(y)|&\leq C d_G(y,y_0)^{2\alpha},\\
|\p_{nn}(v-v_{y_0})(y)|&\leq C d_G(y,y_0)^{1+2\alpha}.
\end{align*}
Using this and plugging the explicit expression for $\p_{ij}v_{y_0}$ (c.f. \eqref{eq:v0}) into \eqref{eq:det} gives
$
|e^{y_0,ij}(y)| \leq Cd_G(y,y_0)^{1+4\alpha}.
$
Similarly,  $|\tilde{e}^{y_0,ij}(y)|\leq Cd_G(y,y_0)^{1+6\alpha}$.
Hence, we obtain 
\begin{align*}
|E^{y_0,ij}_1(y)|\leq C|e^{y_0,ij}(y)|+C|\tilde{e}^{y_0,ij}(y)|\leq d_G(y,y_0)^{1+4\alpha}.
\end{align*}
Moreover, it is not hard to deduce that $d_G(y,y_0)^{-(1+4\alpha-\epsilon)}E_1^{y_0,ij}(y)\in \dot{C}^{0,\epsilon}_\ast(\mathcal{B}_{1/2}^{+}(y_0))$. This concludes the proof of Step 1b.\\

\emph{Step 1c: Concatenation.}
We show that the leading order term $\tilde{a}^{ij}G^{ij}(v)$ has the expansion
\begin{align*}
\tilde{a}^{ij}G^{ij}(v)=\tilde{a}^{ij}(y_0)\p_{m_{k\ell}}G^{ij}(v_{y_0})\p_{k\ell}v + P_{y_0}(y)+E^{y_0}_3(y),
\end{align*}
where 
\begin{align}
\label{eq:P3}
P_{y_0}(y)=\tilde{a}^{ij}(y_0)\left(G^{ij}(v_{y_0})-\p_{m_{k\ell}}G^{ij}(v_{y_0})\p_{k\ell}v_{y_0}\right)\in  \mathcal{P}_1^{hom},
\end{align}
is a polynomial of the form $c(y_0)y_n$, and $E^{y_0}_3(y)$ satisfies for $\eta_0=\min\{3,1+4\alpha\}$
\begin{equation}
\label{eq:err_3}
\|d_G(\cdot,y_0)^{-\eta_0}E^{y_0}_3\|_{L^\infty(\mathcal{B}_{1/2}^+(y_0))}+\left[d_G(\cdot,y_0)^{-(\eta_0-\epsilon)}E^{y_0}_3\right]_{\dot{C}^{0,\epsilon}_\ast(\mathcal{B}_{1/2}^+(y_0))}\leq C.
\end{equation}

\emph{Proof of Step 1c:} Using the expansions for $\tilde{a}^{ij}$ and for $G^{ij}(v)$ from Steps 1a and 1b, we obtain
\begin{align*}
\tilde{a}^{ij}(y)G^{ij}(v)&=\tilde{a}^{ij}(y_0)G^{ij}(v)+E_{2}^{y_0,ij}(y)G^{ij}(v)\\
&=\tilde{a}^{ij}(y_0)\left(G^{ij}(v_{y_0})+\p_{m_{k\ell}}G^{ij}(v_{y_0})\p_{k\ell}(v-v_{y_0}) +E_1^{y_0,ij}(y)\right)\\
&\quad +E_2^{y_0,ij}(y)G^{ij}(v)\\
&=\tilde{a}^{ij}(y_0)\p_{m_{k\ell}}G^{ij}(v_{y_0})\p_{k\ell}v +P_{y_0}(y) +E_3^{y_0}(y),
\end{align*}
where 
\begin{align*}
E_3^{y_0}(y) := \tilde{a}^{ij}(y_0)E_1^{y_0,ij}(y)+E^{y_0,ij}_2(y)G^{ij}(v).
\end{align*}
Recalling the error bounds from Steps 1a, 1b and further observing 
\begin{align*}
\left| G^{ij}(v) \right| \leq Cd_G(y,y_0),
\end{align*}
entails (\ref{eq:err_3}).\\

\emph{Step 2: Expansion of the lower order contributions.} 
For the lower order contribution the asymptotics of $v$ immediately yield
\begin{align*}
\left|J(v)(y)\left( \sum\limits_{j=1}^{n-1} \tilde{b}^j(y) \p_j v(y) + \tilde{b}^n(y)y_n + \tilde{b}^{n+1}(y)y_{n+1}\right)\right|\leq Cc_{\ast} d_G(y,P)^3.
\end{align*}
Here we used that $\|\nabla a^{ij}\|_{L^{\infty}}\leq C c_{\ast}$. Hence, this error is small compared with the error term $E^{y_0}_3(y)$ from the leading order expansion (c.f. \eqref{eq:err_3}). 
\end{proof}

The previous proposition allows us to apply an iterative bootstrap argument to obtain higher regularity for $v$.

\begin{prop}
\label{prop:error_gain2}
Assume that $v\in X_{\alpha,\epsilon}(\mathcal{B}_{1}^+)$ for some $\alpha\in (0,1]$, $\epsilon \in (0,\gamma]$, and that it satisfies $F(v,y)=0$ with $a^{ij}(x)\in C^{1,\gamma}$ for some $\gamma\in (0,1]$. Then 
 $v\in X_{\delta,\epsilon}(\mathcal{B}_{1/2}^+)$ for any $\delta\in (0,1)$. 
\end{prop}

\begin{proof}
If $1+4\alpha<3$, i.e. $0\leq \alpha<1/2$, Proposition~\ref{prop:error_gain} and Remark~\ref{rmk:error_gain3} yield that for each fixed $y_0\in P\cap \mathcal{B}_{1/2}$ the Legendre function $v$ solves  
\begin{align}
\label{eq:eq}
L_{y_0}v=L_{y_0}v_{y_0}+\tilde{f} \text{ in } \mathcal{B}_{1/2}^+(y_0),
\end{align}
where $L_{y_0}$ is the ``constant coefficient" Baouendi-Grushin type operator from Remark~\ref{rmk:error_gain3}, $L_{y_0}v_{y_0}=c(y_0)y_n$ and the function $\tilde{f}(y)$ is such that
$$\|d_G(\cdot,y_0)^{-(1+4\alpha)}\tilde{f}\|_{L^\infty(\mathcal{B}_{1/2}^+(y_0))}+[d_G(\cdot,y_0)^{-(1+4\alpha-\epsilon)}\tilde{f}]_{\dot{C}^{0,\epsilon}_\ast(\mathcal{B}_{1/2}^+(y_0))}\leq C,$$ 
where $C$ depends on $\|v\|_{X_{\alpha,\epsilon}}$ and $[D a^{ij}]_{\dot{C}^{0,\gamma}}$ and is in particularly independent of $y_0$. We apply the compactness argument from the Appendix (c.f. the proof of Proposition~\ref{prop:Hoelder0}) at each point $y_0\in P\cap \mathcal{B}_{1/2}$. This is possible as $L_{y_0}$ is a self-adjoint, constant coefficient subelliptic operator of Baoundi-Grushin type which is hypoelliptic after suitable reflections (c.f. Remark \ref{rmk:error_gain3} and \cite{JSC87}). We note that as in the case of the Grushin operator there are no fourth order homogeneous polynomials with symmetry (even about $y_{n+1}$ and odd about $y_n$) which are solutions to the equation $L_{y_0}v=0$. Combining the above approximation result along $P\cap \mathcal{B}_{1/2}$ with the $C^{2,\epsilon}_\ast$, $\epsilon\leq \gamma$, estimate in the corresponding non-tangential region (with respect to $P$) leads to $v\in X_{2\alpha,\epsilon}(\mathcal{B}_{1/4}^+)$. 
We repeat the above procedure until after finitely many, say, $k$, steps, $1+4k\alpha>3$. This results in $v\in X_{\delta,\epsilon}(\mathcal{B}_{1/2^{k}}^+)$ for every $\delta\in (0,1)$ (where we used the nonexistence of homogeneous fourth order approximating polynomials). Repeating this procedure in $\mathcal{B}^+_{1/2}(\bar y)$ for $\bar y\in P\cap \mathcal{B}_{1/2}^+$ and by a covering argument, we obtain that $v\in X_{\delta,\epsilon}(\mathcal{B}_{1/2}^+)$ for every $\delta\in (0,1)$.
\end{proof}

\section[Free Boundary Regularity]{Free Boundary Regularity for $C^{k,\gamma}$ Metrics, $k\geq 1$}
\label{sec:fb_reg}
In this section we apply the implicit function theorem to show that  the regular free boundary is locally in $C^{k+1,\gamma}$, if $a^{ij}\in C^{k,\gamma}$ with $k\geq 1$ and $\gamma\in (0,1)$. Moreover, we also argue that the regular free boundary is locally real analytic, if $a^{ij}$ is real analytic. \\

In order to invoke the implicit function theorem, we discuss the mapping properties of the nonlinear function $F$ in the next two sections. More precisely, we prove that 
\begin{itemize}
\item the nonlinearity $F$ maps $X_{\delta,\epsilon}(\mathcal{B}_1^+)$ to $Y_{\delta,\epsilon}(\mathcal{B}_1^+)$ for any $\delta\in (0,1)$ and $\epsilon$ sufficiently small (c.f. Section \ref{sec:nonlinmap}),
\item and that its linearization in a neighborhood of $v_0$  is a perturbation of the Baouendi-Grushin Laplacian and is hence invertible (c.f. Section \ref{sec:grushin}). 
\end{itemize}
Then in Section \ref{subsec:IFT0} we introduce an one-parameter family of diffeomorphisms which will form the basis of our application of the implicit function theorem in Section \ref{sec:IFTAppl}. In Section \ref{sec:IFTAppl} we apply the implicit function theorem argument to show the regularity of the free boundary in $C^{k,\gamma}$, $k\geq 1$ metrics and analytic metrics, which   yields the desired proof of Theorem \ref{thm:higher_reg}.

\subsection{Mapping properties of $F$}
\label{sec:nonlinmap}
As a consequence of the representation of $F$ which was derived in Proposition \ref{prop:error_gain} we obtain the following mapping properties for our nonlinear function $F$. 

\begin{prop}
\label{prop:nonlin_map}
Let $a^{ij}:B_1^+ \rightarrow \R^{(n+1)\times (n+1)}_{sym}$ be a $C^{k,\gamma}$ tensor field with $k\geq 1$ and $\gamma \in (0,1]$. Assume that  the nonlinear function $F$ is as in (\ref{eq:nonlin_2}). Then for any $\delta\in (0,1]$ and $\epsilon\in (0,\gamma)$, we have that
\begin{align*}
F:X_{\delta,\epsilon}(\mathcal{B}_1^+) \rightarrow Y_{\delta,\epsilon}(\mathcal{B}_1^+).
\end{align*}
\end{prop}

These properties will be used in Propositions \ref{prop:reg_a} and \ref{prop:invertible} to establish the mapping properties of the nonlinear function to which we apply the implicit function theorem in Section \ref{sec:IFT1}

\begin{proof}
The mapping properties of $F$ are an immediate consequence of the representation for $F$ that was obtained in Proposition \ref{prop:error_gain}. Indeed, given any $u\in X_{\delta,\epsilon}$, by Proposition \ref{prop:error_gain}, we have
\begin{align*}
F(u, y) = \sum\limits_{i,j=1}^{n+1} \tilde{a}^{ij}(y_0) \p_{m_{k \ell}} G^{ij}(u_{y_0}) \p_{k \ell } u + P_{y_0}(y) + E_{y_0}(y).
\end{align*}
Due to Proposition~\ref{prop:decompI}, Proposition~\ref{prop:error_gain} and Remark~\ref{rmk:error_gain3} we infer that $P_{y_0}(y) + E_{y_0}(y)\in Y_{\delta,\epsilon}(\mathcal{B}_1^+)$. For the remaining linear term we note that similarly as in the proof of Proposition \ref{prop:error_gain}
\begin{align*}
\tilde{a}^{ij}(y_0) \p_{m_{k \ell}} G^{ij}(u_{y_0}) \p_{k \ell } u & = \tilde{a}^{ij}(y_0) \p_{m_{k \ell}} G^{ij}(u_{y_0}) \p_{k \ell } u_{y_0} \\
& \quad + \tilde{a}^{ij}(y_0) \p_{m_{k \ell}} G^{ij}(u_{y_0}) \p_{k \ell } (u-u_{y_0}) \\
&  =  c(y'')y_n + r^{1+2\alpha-\epsilon}f(y),
\end{align*}
with $c(y'')\in C^{0,\delta}(\R^{n-1}\cap \mathcal{B}_1^+)$ and $f(y)\in C^{0,\epsilon}_{\ast}(\mathcal{B}_1^+)$.
This implies the result.
\end{proof}

\subsection{Linearization and the Baouendi-Grushin operator}
\label{sec:grushin}
In this section we compute the linearization of $F$ and show that it can be interpreted as a perturbation of the Baouendi-Grushin Laplacian. We treat the cases $a^{ij}\in C^{k,\gamma}$ with $k=1$ and $k\geq 2$ simultaneously as only small modifications are needed in the argument.
If $a^{ij}\in C^{k,\gamma}$ with $k\geq 2$, by using the notation in \eqref{eq:nonlin_2}, 
\begin{align*}
F(D^2v, Dv,y):=&\sum_{i,j=1}^{n+1}a^{ij}(y'',-\p_nv,-\p_{n+1}v)G^{ij}(v)-J(v)\sum_{j=1}^{n-1}b^j(y'',-\p_nv,-\p_{n+1}v)\p_j v\\
&-J(v)b^n(y'',-\p_nv,-\p_{n+1}v)y_n-J(v)b^{n+1}(y'',-\p_nv,-\p_{n+1}v)y_{n+1},
\end{align*} 
In the case of $a^{ij}\in C^{1,\gamma}$ we view $F$ as 
\begin{align*}
F(D^2v, Dv,y)&=\sum_{i,j}^{n+1}a^{ij}(y'',-\p_nv,-\p_{n+1}v)G^{ij}(v)+ f(y),\\
&\text{where } f(y)=-J(v(y))\left(\sum_{j=1}^{n-1}\tilde{b}^j(y)\p_j v(y)+\tilde{b}^n(y)y_n+\tilde{b}^{n+1}(y)y_{n+1}\right).
\end{align*}
Here we view $F$ as a mapping $F:\R^{(n+1)\times(n+1)}_{sym}\times \R^{n+1}\times U\rightarrow \R$ and introduce the abbreviations
\begin{align*}
F_{k\ell}(M,P,y)&:=\frac{\partial F(M,P,y)}{\partial m_{k\ell}}, \quad  M=(m_{k\ell})\in \R^{(n+1)\times(n+1)}_{sym},\ P\in \R^{n+1},\ y\in \R^{n+1},\\
F_k(M,P,y)&:=\frac{\partial F(M,P,y)}{\partial p_{k}}.
\end{align*}
For notational convenience, we use the conventions $$F_{k\ell}(v,y):=F_{k\ell}(D^2v, Dv, y), \quad F_k(v,y):=F_{k}(D^2v, Dv,y), \quad F(v,y):=F(D^2v, Dv,y).$$
The linearization of $F$ at $v$ is 
\begin{align}
\label{eq:op_lin}
L_v:=F_{k\ell}(v,y) \p_{k \ell}  + F_k (v,y) \p_k.
\end{align}
In the case $a^{ij}\in C^{k,\gamma}$ with $k\geq 2$, we have
\begin{align*}
F_{k\ell}(v,y) &= \sum\limits_{i,j=1}^{n+1} \tilde{a}^{ij} \p_{m_{k \ell}} G^{ij}(v) - \p_{m_{k \ell}} J(v) \left(\sum_{j=1}^{n-1}\tilde{b}^j\partial_jv+\tilde{b}^ny_n+\tilde{b}^{n+1}y_{n+1}\right),\\
F_k(v,y) &= -J(v)\sum_{j=1}^{n-1}\tilde{b}^j, \mbox{ for } k\in\{1,\dots,n-1\},\\
F_{k}(v,y) &= \sum\limits_{i,j=1}^{n+1} b^{ij}_k G^{ij}(v)- J(v)\left(\sum_{j=1}^{n-1}b^j_k\partial_jv+b^n_k y_n+b^{n+1}_ky_{n+1}\right), \mbox{ for } k\in\{n,n+1\}.
\end{align*}
Here
\begin{align*}
\tilde{a}^{ij}&:= a^{ij}|_{(y'',-\p_n v(y), - \p_{n+1} v(y))}, \quad \tilde{b}^j=\sum_{i=1}^{n+1}\p_{x_i}a^{ij}|_{(y'',-\p_nv(y),-\p_{n+1}v(y))}\\
b^{ij}_k &:= \p_{x_k} a^{ij}|_{(y'',-\p_n v(y), - \p_{n+1} v(y))},\\
b^{j}_k &:= \sum_{i=1}^{n+1}\p_{x_k x_i} a^{ij}|_{(y'',-\p_n v(y), - \p_{n+1} v(y))}.
\end{align*}
In particular, we note that the linearization of $F$ already involves second order derivatives of our metric $a^{ij}$. \\
In the case $a^{ij}\in C^{1,\gamma}$, we have
\begin{align*}
F_{k\ell}(v,y)&=\sum_{i,j=1}^{n+1}\tilde{a}^{ij}\p_{m_{k\ell}}G^{ij}(v), \quad k,\ell\in\{1,\dots,n+1\},\\
F_k(v,y)&=0,\ k\in\{1,\dots, n-1\}; \quad F_k(v,y)=\sum\limits_{i,j=1}^{n+1} b^{ij}_k G^{ij}(v), \quad k\in\{n,n+1\}.
\end{align*} 
Let 
$v_0(y):=-\frac{1}{6}\left(y_n^3-3y_ny_{n+1}^2\right)
$ be the (scaled) blow-up of $v$ at $0$.
A direct computation shows that $F_{k\ell}(v_0,0)\p_{k\ell}=\Delta_G$, where $\Delta_G$ is the standard Baouendi-Grushin Laplacian in Section~\ref{sec:intrinsic}. Thus, we write
\begin{align*}
L_v&=\Delta_G+\left(F_{k \ell}(v,y)-F_{k \ell}(v_0,0)\right)\p_{k \ell} +F_k(v,y)\p_k\\
&=:\Delta_G+\mathcal{P}_v.
\end{align*}

With this at hand, we can prove the following mapping properties for $L_v$:

\begin{prop}
\label{prop:linear}
Let $L_v, \mathcal{P}_v$ be as above. Assume furthermore in the case of $a^{ij}\in C^{k,\gamma}$ with $k\geq 2$ that $[D^2_xa^{ij}]_{\dot{C}^{0,\gamma}}\leq c_\ast$. Given $\delta\in (0,1]$ and $\epsilon\in (0,\min\{\delta,\gamma\})$, let $X_{\delta,\epsilon}(\mathcal{B}_1^+)$ and $Y_{\delta,\epsilon}(\mathcal{B}_1^+)$ be the spaces from Definition~\ref{defi:spaces_loc}. Then 
\begin{align*}
L_v : X_{\delta,\epsilon}(\mathcal{B}_1^+) \rightarrow Y_{\delta,\epsilon}(\mathcal{B}_1^+).
\end{align*}
Moreover, if $\|v-v_0\|_{X_{\delta,\epsilon}(\mathcal{B}_1^+}\leq \delta_0$, then 
$$\|\mathcal{P}_v w\|_{Y_{\delta,\epsilon}(\mathcal{B}_1^+)}\lesssim \max\{\delta_0,c_\ast\}\|w\|_{X_{\delta,\epsilon}(\mathcal{B}_1^+)}, \quad \text{for all } w\in X_{\delta,\epsilon}(\mathcal{B}_1^+).$$  
\end{prop}

\begin{proof}
We first show the claims of the proposition in the case of $a^{ij}\in C^{k,\gamma}$ with $k\geq 2$.  We begin by arguing that $L_v w \in Y_{\delta,\epsilon}(\mathcal{B}_1^+)$ if $w\in X_{\delta, \epsilon}(\mathcal{B}_1^+)$. 
\begin{itemize}
\item[(a)] To show this, we observe that
\begin{align*}
&\sum_{k,\ell}\p_{m_{k\ell}}G^{ij}(v)\p_{k\ell}w
=\det\begin{pmatrix}
\p_{ij}w & \p_{in}w & \p_{i,n+1}w\\
\p_{in}v & \p_{nn}v & \p_{n,n+1}v\\
\p_{i,n+1}v & \p_{n+1,n}v & \p_{n+1,n+1}v
\end{pmatrix} \\
&+ \det\begin{pmatrix}
\p_{ij}v & \p_{in}v & \p_{i,n+1}v\\
\p_{in}w & \p_{nn}w & \p_{n,n+1}w\\
\p_{i,n+1}v & \p_{n+1,n}v & \p_{n+1,n+1}v
\end{pmatrix}
+\det\begin{pmatrix}
\p_{ij}v & \p_{in}v & \p_{i,n+1}v\\
\p_{in}v & \p_{nn}v & \p_{n,n+1}v\\
\p_{i,n+1}w & \p_{n+1,n}w & \p_{n+1,n+1}w
\end{pmatrix},
\end{align*}
if $i,j\in\{1,\dots, n-1\}$ and for the remaining indices $(i,j)$ the expression $\sum_{k,\ell}G^{ij}(v)\p_{k\ell}w$ is similar. Thus, the mapping property of $\sum_{i,j}\sum_{k,\ell}\tilde{a}^{ij}\p_{m_{k\ell}}G^{ij}(v)\p_{k\ell}w$ follows along the lines of the proof of Proposition~\ref{prop:nonlin_map} and Proposition~\ref{prop:error_gain}.
\item[(b)] We discuss the term $\sum\limits_{i,j=1}^{n+1}\sum\limits_{\ell=n}^{n+1}b^{ij}_\ell\p_\ell w G^{ij}(v)$. As $b^{ij}_\ell$ is $C^{1,\gamma}$, it satisfies the same decomposition as $\tilde{a}^{ij}$ in \eqref{eq:Holderweight0}. Using the characterization for $\p_nw$, $\p_{n+1}w$ from Proposition~\ref{prop:decompI} (ii), we have that around $y_0\in P\cap \mathcal{B}_{1/2}$
\begin{align*}
\sum_{\ell=n,n+1}b^{ij}_\ell\p_\ell w = \tilde{b}^{ij}_n(y_0)\p_nw(y_0)+ E_{y_0}^{ij},
\end{align*}
where $E_{y_0}^{ij}$ satisfies the same error bounds as (\ref{eq:Holderweight0}). Moreover, the functions $\tilde{b}^{ij}_n$ inherit the off-diagonal condition of $a^{ij}$, i.e. $\tilde{b}^{i,n+1}_n=0$ on $\{y_{n+1}=0\}$ for $i\in \{1,\dots, n\}$. Thus, using exactly the same estimate as for $\tilde{a}^{ij}G^{ij}(v)$ in Step 1 of Proposition~\ref{prop:error_gain}, we obtain that $\sum\limits_{i,j=1}^{n+1}\sum\limits_{\ell=n}^{n+1}b_\ell^{ij}\p_\ell w G^{ij}(v)\in Y_{\delta, \epsilon}(\mathcal{B}_1^+)$.
\item[(c)] The term 
\begin{align*}
-J(v)\left( \sum\limits_{j=1}^{n-1}b^j_k \p_j v + b^n_k y_n + b^{n+1}_k y_{n+1} \right)\p_k w,\quad k\in \{n,n+1\}
\end{align*}
is a lower order term contained in $Y_{\delta,\epsilon}(\mathcal{B}_1^+)$ as it is bounded by $r(y)^3=\dist(y,P)^3$ and satisfies the right Hölder bounds.
\item[(d)] Similarly, the contribution 
\begin{align*}
-\p_{m_{k\ell}}J(v)\left( \sum\limits_{j=1}^{n-1}\tilde{b}^j \p_j v + \tilde{b}^n y_n + \tilde{b}^{n+1} y_{n+1} \right)\p_{k\ell} w,\quad k,\ell\in \{n,n+1\}
\end{align*}
is a lower order term contained in $Y_{\delta,\epsilon}(\mathcal{B}_1^+)$ as it is also bounded by $\dist(y,P)^3$ and satisfies the right Hölder bounds.
\end{itemize}

We continue by proving the bounds for $\mathcal{P}_v$. To this end, we again consider the individual terms in the linearization separately:
\begin{itemize}
\item Estimate of $\left(F_{k\ell}(v,y)-F_{k\ell}(v_0,0)\right)\p_{k\ell}w$. We will only present the details of the estimate for $k,\ell\in \{1,\dots,n-1\}$. The estimates for the remaining terms are similar. \\ 
For $k,\ell\in \{1,\dots, n-1\}$, from \eqref{eq:op_lin} we have
\begin{align*}
F_{k\ell}(v,y)=-\tilde{a}^{k\ell}(y)J(v).
\end{align*}
By assumption $v\in X_{\delta,\epsilon}(\mathcal{B}_1^+)$ is in a $\delta_0$ neighborhood of $v_0$ and by Proposition~\ref{prop:decompI} we have the decomposition ($r=r(y)=(y_n^2+y_{n+1}^2)^{1/2}$)
\begin{align*}
\p_{ij}v(y)&=r^{-1+2\delta-\epsilon}C_{ij}(y),\\
\p_{in}v(y)&=B_i(y'')+r^{2\delta-\epsilon}C_{in}(y),\\
\p_{i,n+1}v(y)&=r^{2\delta-\epsilon}C_{i,n+1}(y),\\
\p_{nn}v(y)&=A_0(y'')y_n+r^{1+2\delta-\epsilon}C_{n,n}(y),\\
\p_{n,n+1}v(y)&=A_1(y'')y_{n+1}+r^{1+2\delta-\epsilon}C_{n,n+1}(y),\\
\p_{n+1,n+1}v(y)&=A_1(y'')y_{n}+r^{1+2\delta-\epsilon}C_{n+1,n+1}(y),
\end{align*}
where $B_i, A_0, A_1\in C^{0,\delta}(P\cap \mathcal{B}_1)$ with 
\begin{equation}\label{eq:smallness}
[B_i]_{\dot{C}^{0,\delta}}+\|A_0-1\|_{L^\infty}+\|A_1-(-1)\|_{L^\infty}\lesssim \delta_0,
\end{equation}
and $C_{ij}\in C^{0,\epsilon}_\ast(\mathcal{B}_1^+)$ with $[C_{ij}]_{\dot{C}^{0,\epsilon}_\ast(\mathcal{B}_1^+)}\lesssim \delta_0$.
Thus, for $k,\ell\in \{1,\dots, n-1\}$
\begin{align*}
&\quad F_{k\ell}(v,y)-F_{k\ell}(v_{0},0)\\
&=-\left(\tilde{a}^{k\ell}(y)-\tilde{a}^{k\ell}(0)\right)J(v_{0})-\tilde{a}^{k\ell}(y)\left(J(v)-J(v_{0})\right)\\
&\lesssim c_\ast(y_n^2+y_{n+1}^2)+ \delta_0 (y_n^2+y_{n+1}^2).
\end{align*}
Consequently,
\begin{align*}
\left|\left(F_{k\ell}(v,y)-F_{k\ell}(v_0,0)\right)\p_{k\ell}w(y)\right|\lesssim \max\{c_\ast,\delta_0\}r^{1+2\delta}|r^{-(1-2\delta)}\p_{k\ell}w(y)|.
\end{align*}
Moreover, it is not hard to check that 
\begin{align*}
\left[r(\cdot)^{-(1+2\delta-\epsilon)}(F_{k\ell}(v,\cdot)-F_{k\ell}(v_0,0))\p_{k\ell}w\right]_{\dot{C}^{0,\epsilon}_\ast(\mathcal{B}_{1}^+)}\lesssim \max\{c_\ast,\delta_0\}\|w\|_{X_{\delta,\epsilon}(\mathcal{B}_1^+)}.
\end{align*}
Hence, for $k,\ell\in \{1,\dots, n-1\}$
$$\left\|(F_{k\ell}(v,y)-F_{k\ell}(v_0,0))\p_{k\ell}w\right\|_{Y_{\delta,\epsilon}(\mathcal{B}_1^+)}\lesssim \max\{c_\ast,\delta_0\}\|w\|_{X_{\delta,\epsilon}(\mathcal{B}_1^+)}.$$

For the remaining second order terms in the linearization we argue similarly. 

\item Estimate for $F_k(v,y)\p_k$. To estimate the lower order terms $F_k(v,y)$, we need to assume that $a^{ij}\in C^{2,\gamma}$ for some $\gamma>0$. Furthermore, we assume $[D^2_xa^{ij}]_{\dot{C}^{0,\gamma}}\leq c_\ast$ (this assumption on the second derivatives becomes necessary as the term $F_k(v,y)$ with $k\in \{n,n+1\}$ involves $D^2a$, c.f. \eqref{eq:op_lin}). Note that combining this with the assumption (A6) and using an interpolation estimate we have $\|D^2_xa^{ij}\|_{L^\infty}\leq Cc_\ast$.\\
We begin with the terms with $k\in\{1,\dots, n-1\}$, i.e. with
\begin{align*}
F_k(v,y)&=-J(v)\sum_{j=1}^{n-1}\tilde{b}^j, \quad \tilde{b}^j=\sum_{i=1}^{n+1}\p_{x_i}a^{ij}|_{(y'',-\p_nv,-\p_{n+1}v)}.
\end{align*}
As by our assumption (A6) $\|\tilde{b}^i\|_{L^\infty}\lesssim c_\ast$, the asymptotics of $J(v)$ immediately yield that $\|F_k(v,y)\p_k w\|_{Y_{\delta,\epsilon}(\mathcal{B}_1^+)}\lesssim c_\ast\|w\|_{X_{\delta,\epsilon}(\mathcal{B}_1^+)}$ as long as $\epsilon\leq \gamma$.\\
For the contributions with $k\in\{n,n+1\}$, $F_k(v,y)$ is of the same structural form as the original nonlinear function $F(v,y)$, however with coefficients which contain an additional derivative, i.e. $a^{ij}$ is replaced by $b^{ij}_k$ and $\tilde{b}^j$ is replaced by $b^j_k$ (c.f. \eqref{eq:op_lin}). 
Thus, by the argument in Section \ref{sec:nonlinmap} on the mapping properties of the nonlinear function $F(v,y)$, we infer that $F_k(v,y)\p_k w \in Y_{\alpha,\epsilon}(\mathcal{B}_1^+)$, $k\in\{n,n+1\}$ (for $\epsilon \leq \gamma$) for any $w\in X_{\alpha,\epsilon}(\mathcal{B}_1^+)$. Moreover, it satisfies $\|F_{k}(v,y)\p_k w\|_{Y_{\delta,\epsilon}(\mathcal{B}_1^+)}\lesssim C [D^2 a^{ij}]_{\dot{C}^{0,\gamma}}\|w\|_{X_{\delta,\epsilon}(\mathcal{B}_1^+)}$.  
\end{itemize}
Combining the previous observations concludes the proof of Proposition \ref{prop:linear} in the case that $k\geq 2$.\\

For the case $k=1$, we notice that if $a^{ij}\in C^{1,\gamma}$ the linearization is simply given by $L_{v}=F_{k\ell}(v,y)\p_{k\ell}+\sum\limits_{i,j=1}^{n+1} b^{ij}_k G^{ij}(v)\p_k$. Thus, a similar proof as for the case $C^{k,\gamma}$ with $k\geq 2$ applies.
\end{proof}

\subsection{Hölder Regularity and Analyticity}
\label{sec:IFT1}
In this section we apply the implicit function theorem to show that if $a^{ij}\in C^{k,\gamma}$ with $k\geq 1$ and $\gamma\in (0,1)$, then the regular free boundary is locally in $C^{k+1,\gamma}$. \\
To this end, we first define a one-parameter family of diffeomorphisms which we compose with our Legendre function to create an ``artificially parameter-dependent problem''. Due to the regularity properties of $F$, this is then exploited to deduce the existence of a solution to the parameter-dependent problem, which enjoys good regularity properties in the artificially introduced parameter (c.f. Proposition \ref{prop:reg_a}). Finally, this regularity is transfered from the parameter variable to the original variables yielding the desired regularity properties of our Legendre function (c.f. Theorems \ref{prop:hoelder_reg_a}, \ref{prop:analytic}). This then proves the claims of Theorem \ref{thm:higher_reg}.\\

In the sequel, we will always assume that $v$ is a Legendre function (c.f. (\ref{eq:legendre})) which is associated with a solution of the variable coefficient thin obstacle problem (\ref{eq:varcoeff}), which satisfies the normalizations (A4) and (A5). The coefficient metric $a^{ij} \in C^{k,\gamma}$, $k\geq 1$, is assumed to obey the conditions (A1)-(A3) as well as (A6). We also suppose that $[D^2_xa^{ij}]_{\dot{C}^{0,\gamma}}\leq c_\ast$ if $k\geq 2$. By rescaling we assume that $v$ is well defined in $\mathcal{B}_2^+$. We have shown in Proposition~\ref{prop:error_gain2} that $v\in X_{\delta,\epsilon}(\mathcal{B}_1^+)$ for any $\delta\in (0,1)$ and $\epsilon\in (0,\gamma]$. Furthermore, we recall that  $\|v-v_0\|_{X_{\delta,\epsilon}(\mathcal{B}_1^+)}\leq C\max\{\epsilon_0,c_\ast\}$ (c.f. Remark~\ref{rmk:close}), where $v_0(y)=-\frac{1}{6}(y_n^3-3y_ny_{n+1}^2)$ is the model solution, which is the (rescaled) blow-up of $v$ at $0$.

\subsubsection{An infinitesimal translation.}
\label{subsec:IFT0}
For $y\in \R^{n+1}$ and $a\in \R^{n-1}$ fixed, we consider the following ODE 
\begin{equation}
\label{eq:ODE}
\begin{split}
\phi'(t)&=a ((3  /4 )^2-|\phi(t)|^2)^5_+ \eta(y_n,y_{n+1}),\\
\phi(0)&=y''.
\end{split}
\end{equation}
Here $\eta$ is an in the $y_n, y_{n+1}$ variables radially symmetric smooth cut-off function supported in $\{(y_n,y_{n+1})| \ y_n^2+y_{n+1}^2<1/2\}$, which is equal to one in $\{y_n^2+y_{n+1}^2\leq 1/4\}$. We denote the unique solution to the above ODE by $\phi_{a,y}(t)$ and let $$\Phi_a(y):= (\phi_{a,y}(1),y_n,y_{n+1}).$$ 
Due to the $C^5$ regularity of the right hand side of (\ref{eq:ODE}), we obtain that $\phi_{a,y}(1)$ is $C^5$ in $y$. 
Moreover, an application of a fixed point argument yields that $\phi_{a,y}(1)$ is analytic in the parameter $a$. We summarize these properties as:

\begin{lem}
For each $a\in \R^{n-1}$, $\Phi_a:\R^{n+1}\rightarrow \R^{n+1}$ is a $C^5$ diffeomorphism. The mapping $\R^{n-1}\ni a\mapsto \Phi_a\in C^5(\R^{n+1})$ is analytic. 
\end{lem}

Moreover, we note that $\Phi_a$ enjoys further useful properties:

\begin{lem}
\label{prop:psi}
\begin{itemize}
\item[(i)] For each $a\in \R^{n-1}$, $\Phi_a(\{y_n=0\})\subset\{y_n=0\}$ and $\Phi_a(\{y_{n+1}=0\})\subset\{y_{n+1}=0\}$. Moreover, $\Phi_a(y)=y$ if $y\notin \{y\in \R^{n+1}| \ |y''|<\frac{3 }{4}, \ y_n^2+y_{n+1}^2<\frac{1}{2}\}$.
\item[(ii)] For each $a\in \R^{n-1}$ and $y\in \{y\in\R^{n+1}| \ y_n^2+y_{n+1}^2<\frac{1}{4}\}$, we have $\Phi_a(y)=(\phi_{a,y''}(1),y_n,y_{n+1})$, i.e. $\Phi_a$ only acts on the tangential variables. 
\item[(iii)] For each $a\in \R^{n-1}$, $\p_{n+1}\Phi_a(y'',y_n,0)=(0,\dots, 0,1)$.
\item[(iv)] At $a=0$, $\Phi_0(y)=y$ for all $y\in \R^{n+1}$.
\end{itemize}
\end{lem}

\begin{proof}
This follows directly from the definition of $\Phi_a$ and a short calculation.
\end{proof}

Let $v$ be a Legendre function as described at the beginning of Section~\ref{sec:IFT1}. We use the family of diffeomorphisms $\Phi_a$ to define a one-parameter family of functions:
$$v_a(y):=v(\Phi_a(y)).$$
We first observe that the space $X_{\delta,\epsilon}$ (c.f. Definition \ref{defi:spaces}) is stable under the diffeomorphism $\Phi_a$:

\begin{lem} \label{prop:psi2}
If $v\in X_{\delta,\epsilon}(\mathcal{B}_{1}^+)$, then $v_a=v\circ \Phi_a\in X_{\delta,\epsilon}(\mathcal{B}_{1}^+)$ as well.
\end{lem}

\begin{proof}
We first check that $v_a$ satisfies the Dirichlet-Neumann boundary condition. Indeed, by (i) in Lemma~\ref{prop:psi}, if $v=0$ on $\{y_n=0\}$, then $v_a=0$ on $\{y_n=0\}$ as well. To verify the Neumann boundary condition, we compute
\begin{align*}
\p_{n+1}v_a(y)=\sum_{k=1}^{n+1}\p_{k}v(z)\big|_{z=\Phi_a(y)}\p_{n+1} \Phi_a^k(y).
\end{align*}
Thus by (i) and (iii) of Lemma~\ref{prop:psi}, $\p_{n+1}v_a=0$ on $\{y_{n+1}=0\}$. 

Next by property (ii) of Lemma~\ref{prop:psi}, for $y\in \{y\in \R^{n+1}| y_n^2+y_{n+1}^2<\frac{1}{4}\}$ and $i,j\in\{1,\dots,n-1\}$,
\begin{align*}
\p_iv_a(y)&=\sum_{k=1}^{n-1}\p_k v(z)\big|_{z=\Phi_a(y)}\p_i\Phi^k_a(y''), \\
\p_{ij}v_a(y)&=\sum_{k,\ell=1}^{n-1}\p_{k\ell}v(z)\big|_{z=\Phi_a(y)}\p_i\Phi^k_a(y'') \p_j\Phi^\ell_a(y'') + \sum_{k=1}^{n-1}\p_kv(z)\big|_{z=\Phi_a(y)}\p_{ij}\Phi_a^k(y''),\\
\p_{in}v_a(y)&=\sum_{k=1}^{n-1}\p_{kn}v(z)\big|_{z=\Phi_a(y)}\p_i\Phi_a^{k}(y'').
\end{align*}
Thus, combining these calculations with the fact that $\Phi_a$ fixes the $(y_n,y_{n+1})$ variables ((ii) of Lemma~\ref{prop:psi}), it is not hard to check that $v_a$ satisfies the decomposition in Proposition~\ref{prop:decompI} in the region $\{y\in \R^{n+1}| y_n^2+y_{n+1}^2<\frac{1}{4}\}$. The regularity of $\Phi_a$ and of $v$, entails that $v_a\in C^{2,\epsilon}_\ast$ outside of the region $\{y\in \R^{n+1}| y_n^2+y_{n+1}^2<\frac{1}{4}\}$. Thus, $v_a\in X_{\delta,\epsilon}(\mathcal{B}_1^+)$. 
\end{proof}

Since $v$ satisfies $F(v,y)=0$, the function $v_a(y)=v(\Phi_a(y))$ solves a new equation $F_a(u,y)=0$. Here 
\begin{align}\label{eq:Fa}
F_a(u,y)=F(u(\Phi^{-1}_a(z)),z)\big|_{z=\Phi_a(y)}.
\end{align}
For this equation we note the following properties:

\begin{prop}\label{prop:reg_a}
Let $a^{ij}\in C^{k,\gamma}$ with $k\geq 1$, $\gamma\in (0,1]$. Then for each $a\in \R^{n-1}$, $F_a$ maps $X_{\delta,\epsilon}(\mathcal{B}_{1}^+)$ into $Y_{\delta,\epsilon}(\mathcal{B}_{1}^+)$. Moreover, 
\begin{itemize}
\item[(i)] for each $a\in \R^{n-1}$, the mapping 
\begin{align*}
F_a(\cdot, y): X_{\delta,\epsilon}(\mathcal{B}_1^+)\rightarrow Y_{\delta,\epsilon}(\mathcal{B}_1^+), \quad u\mapsto F_a(u,y),
\end{align*}
is $C^{k-1,\gamma-\epsilon}$ in $u$. 
\item[(ii)] For each $u\in X_{\alpha,\epsilon}(\mathcal{B}_1^+)$, the mapping
\begin{align*}
F_{\cdot}(u,y):\R^{n-1} \rightarrow Y_{\delta,\epsilon}(\mathcal{B}_1^+), \ a  \mapsto F_a(u,y),
\end{align*}
is $C^{k-1,\gamma-\epsilon}$ in $a$. If $a^{ij}$ is real analytic in $B_1^+$, then $F_a$ is real analytic in $a$. 
\end{itemize} 
\end{prop}

\begin{proof}
We first check the mapping property of $F_a$. Let $\Psi_a(z):=\Phi_a^{-1}(z)$ and let $\tilde{u}_a(z):=u(\Phi_a^{-1}(z))$. A direct computation shows that for $ i,j\in\{1,\dots,n-1\}$, $\eta,\xi\in\{n,n+1\}$ and $y\in\{y_n^2+ y_{n+1}^2 \leq \frac{{1}}{4}\}$
\begin{align*}
\p_{ij}\tilde{u}_a(z)&=\sum_{k,\ell=1}^{n-1}\p_{k\ell}u(\Psi_a(z))\p_i \Psi_a^{k}(z) \p_j\Psi_a^\ell(z)+\sum_{k=1}^{n-1}\p_k u(\Psi_a(z))\p_{ij}\Psi_a^k(z),\\
\p_{i\xi}\tilde{u}_a(z)&=\sum_{k=1}^{n-1}\p_{k\xi}u(\Psi_a(z))\p_i\Psi^{k}_a(z),\\
\p_{\xi}\tilde{u}_a(z)&=\p_{\xi}u(\Psi_a(z)),\\
\p_{\eta\xi}\tilde{u}_a(z)&=\p_{\eta\xi}u(\Psi_a(z)).
\end{align*}
By property (ii) of Lemma~\ref{prop:psi} and a similar argument as in Lemma~\ref{prop:psi2} we have that $\tilde{u}_a=u\circ \Psi_a \in X_{\delta,\epsilon}(\mathcal{B}_{{1}/4}^+)$, if $u\in X_{\delta,\epsilon}(\mathcal{B}_{1}^+)$. Thus, by Proposition \ref{prop:nonlin_map}, $F(\tilde{u}_a,z)\in Y_{\delta,\epsilon}(\mathcal{B}_{1/4}^+)$. By (ii) in Lemma~\ref{prop:psi}, $F(\tilde{u}_a,z)\big|_{z=\Phi_a(y)}\in Y_{\delta,\epsilon}(\mathcal{B}_{1/4}^+)$ as well. Outside of $\{y_{n}^2 + y_{n+1}^2 \leq \frac{1}{4}\}$, the statement follows without difficulties.\\
Next we show the regularity of $F_a(u,y)$ in $u$ and in the parameter $a$ which were claimed in the statements (i) and (ii). We first show that when $a=0$, $u\mapsto F(u,y)$ is $C^{k-1,\gamma-\epsilon}$. Indeed, we recall the expression of $F(v,y)$ from the beginning of Section~\ref{sec:grushin}. By a similar estimate as in Proposition~\ref{prop:linear} we have that $u\mapsto \sum_{i,j}a^{ij}(z'',-\p_nu,-\p_{n+1}u)G^{ij}(u)$ is $C^{k-1,\gamma-\epsilon}$ regular. To estimate the contribution $J(u)(b^j(u)\p_ju+b^n(u)y_n+b^{n+1}(u)y_{n+1})$ (in the case $k\geq 2$), we use the bound
\begin{align*}
\left|b(u_1,x)-b(u_2,y)\right|\lesssim \|b\|_{\dot{C}^{0,\gamma}}\|u_1-u_2\|^{\gamma-\epsilon}_{X_{\alpha,\epsilon}(\mathcal{B}_1^+)}|x-y|^\epsilon.
\end{align*}
Here we used the decomposition property of $u\in X_{\delta,\epsilon}(\mathcal{B}_1^+)$ from Proposition~\ref{prop:decompI}, that $b$ is $C^{0,\gamma}$ as a function of its arguments and the definition $b(u,y):=b(y'',-\p_nu,-\p_{n+1}u)$. Combining this we infer that 
$$\|(D^{k-1}_{u_1}F-D^{k-1}_{u_2}F)(h^{k-1})\|_{Y_{\alpha,\epsilon}(\mathcal{B}_1^+)}\lesssim_k\|u_1-u_2\|_{X_{\alpha,\epsilon}(\mathcal{B}_1^+)}^{\gamma-\epsilon}\|h\|_{X_{\alpha,\epsilon}(\mathcal{B}_1^+)}^{k-1}.$$ 
To show the regularity of $u\mapsto F_a(u,y)$ for nonzero $a$, we use the definition of $F_a$ in \eqref{eq:Fa} and the computation for $D^2\tilde{u}_a$ from above. The argument is the same as for $a=0$.\\
Now we show the regularity of $F_a(u,y)$ in $a$ for fixed $u$. We only show the case when $k=1$. The remaining cases follow analogously. We recall that
\begin{align*}
F(u,z)&=\sum_{i,j}^{n+1}a^{ij}(z'',-\p_{n}u,-\p_{n+1}u)G^{ij}(u)+ f(z),\\
&\text{where } f(z)=-J(v(z))\left(\sum_{j=1}^{n-1}\tilde{b}^j(z)\p_j v(z)+\tilde{b}^n(z)z_n+\tilde{b}^{n+1}(z)z_{n+1}\right),
\end{align*} 
and that $F_a(u,y)=F(u(\Psi_a(z)),z)\big|_{z=\Phi_a(y)}$ from \eqref{eq:Fa}.
Since $a\mapsto \Psi_a$ and $a\mapsto \Phi_a$ are real analytic and since $F(v,z):X_{\delta, \epsilon}(\mathcal{B}_{1}^+)\rightarrow Y_{\delta,\epsilon}(\mathcal{B}_{1}^+)$ is $C^{k-1,\gamma-\epsilon}$ regular in $v$, it suffices to note the regularity of the mappings 
\begin{align*}
a &\mapsto a^{ij}((\Phi_a(y))'',-\p_n u,-\p_{n+1}u)G_a^{ij}(u)
\end{align*} 
in $a$ as functions from $\R^{n-1}$ to $Y_{\delta,\epsilon}(\mathcal{B}_1^+)$. For the term $f(z)|_{z=\Phi_a(y)}$, since $f(z)$ has the form $f(z)=r(z)^{3-\epsilon}\tilde{f}(z)$ with $r(z)=(z_n^2+z_{n+1}^2)^{1/2}$ and since $\tilde{f}(z)$ is $C^{0,\gamma}$ in its tangential variables (due to the regularity of $\tilde{b}^j$ and the fact that $f(z)$ involves only lower order derivatives of $v$), the map 
$$\R^{n-1}\ni a\mapsto \tilde{f}_a(y):=\tilde{f}(\Phi_a(y))\in C^{0,\epsilon}_\ast$$ 
is $C^{0,\gamma-\epsilon}$ regular.
\end{proof}

\subsubsection{Application of the implicit function theorem and regularity}
\label{sec:IFTAppl}

With this preparation we are now ready to invoke the implicit function theorem. We seek to apply the implicit function theorem in the spaces $X_{\delta,\epsilon}$ and $Y_{\delta,\epsilon}$ from Definition \ref{defi:spaces}. However, the Legendre function $v$ is only defined in $\mathcal{B}_1^+$. Thus, we extend it into the whole quarter space $Q_+$. In order to avoid difficulties at (artificially created) boundaries, we base our argument not on $v_a$ but instead consider $w_a:= v_a - v$, where $v$ is the original Legendre function. For this function we first note that $\supp(w_a) \Subset \mathcal{B}_{3 /4}^+$, which follows from the definition of the diffeomorphism $\Phi_a$. Moreover, $w_a$ solves the following fully nonlinear, degenerate elliptic equation:
\begin{align*}
\tilde{F}_a(w_a,y) := F_a(w_a+v,y)=0 \mbox{ in } \mathcal{B}_{1}^+.
\end{align*}
We extend $w_a$ to the whole quarter space $Q_+$ by setting $w_a=0$ in $Q_+\setminus \mathcal{B}_{1}^+$. Using $w_a=0$ in $Q_+\setminus \mathcal{B}_{3/4}^+$, the function $w_a$ solves the equation 
\begin{align*}
G_a(w_a,  y):= \eta(d_G(y,0)) \tilde{F}_a(w_a + v,y)  + (1-\eta(d_G(y,0))) \D_G w_a= 0, 
\end{align*}
in $Q_+$. Here $\eta: [0,\infty) \rightarrow \R$ is a smooth cut-off function with $\eta(s)=1$ for $s\leq \frac{3 }{4}$ and $\eta(s)=0$ for $s\geq 1$.
This extension is chosen such that the operator is of ``Baouendi-Grushin type'' around the degenerate set $P=\{y_n=y_{n+1}=0\}$ and the Baouendi-Grushin type estimates from Proposition \ref{prop:invert} and from the Appendix, Section \ref{sec:quarter_Hoelder}, can be applied in a neighborhood of $P$. The function $G_a$ satisfies the following mapping properties:

\begin{prop}\label{prop:invertible}
Assume that $a^{ij}\in C^{k,\gamma}$ with $k\geq 1$ and $\gamma\in (0,1]$. Given $a\in \R^{n-1}$, $G_a$ maps from $X_{\delta,\epsilon}$ into $Y_{\delta,\epsilon}$ for each $\delta\in (0,1)$ and $\epsilon\in (0,\gamma)$. Let $L:=D_w G_a\big|_{(a,w)=(0,0)}$ be the linearization of $G_a$ at $w=0$ and $a=0$. Then $L:X_{\delta,\epsilon}\rightarrow Y_{\delta,\epsilon}$ is invertible.
\end{prop}

\begin{proof}
Let $\bar{\eta}(y):= \eta(d_G(y,0))$. By Proposition~\ref{prop:error_gain} the Legendre function $v\in X_{\delta,\epsilon}(\mathcal{B}_{1}^+)$ for any $\delta\in (0,1)$. Thus by Proposition~\ref{prop:reg_a}, $\tilde{F}_a(w,y)=F_a(w+v,y)\in Y_{\delta,\epsilon}(\mathcal{B}_{1}^+)$ for any $w\in X_{\delta,\epsilon}$ and $y\in \mathcal{B}_{1}^+$. Since the Baouendi-Grushin Laplacian also has this mapping property, i.e. $\Delta_G:X_{\delta,\epsilon}\rightarrow Y_{\delta,\epsilon}$, and using the support assumption of $\eta$, we further observe that $G_a=\bar \eta F_a+(1-\bar\eta)\Delta_G$ maps $X_{\delta,\epsilon}$ into $Y_{\delta,\epsilon}$.

By (iv) in Proposition~\ref{prop:psi}, it is not hard to check that the linearization of $G_a$ at $(0,0)$ is given by
\begin{equation}
\begin{split}
\label{eq:lin2}
L=(D_w G_a)|_{(0,0)}= \bar{\eta} \left(F_{k\ell}(v,y)\p_{k\ell}+  F_{k}(v,y)\p_k\right) + (1-\bar{\eta})\D_G.
\end{split}
\end{equation}
Firstly, by Proposition \ref{prop:linear}, $L$ maps $X_{\delta,\epsilon}$ into $Y_{\delta,\epsilon}$. Moreover, $L$ can be written as $L=\Delta_G+\bar\eta\mathcal{P}_v$. Since $\|v-v_0\|_{X_{\alpha,\epsilon}(\mathcal{B}_1^+)}\lesssim \max\{\epsilon_0,c_\ast\}$ by Remark~\ref{rmk:close}, Proposition~\ref{prop:linear} implies that $\|\bar\eta \mathcal{P}_v(w)\|_{Y_{\delta,\epsilon}}\lesssim \max\{\epsilon_0,c_\ast\}\|w\|_{X_{\delta,\epsilon}}$ for $\epsilon\in (0,\gamma)$. Thus if $\epsilon_0$ and $c_\ast$ are sufficiently small, $L:X_{\delta,\epsilon}\rightarrow Y_{\delta,\epsilon}$ is invertible, as $\Delta_G:X_{\delta,\epsilon}\rightarrow Y_{\delta,\epsilon}$ is invertible (c.f. Lemma~\ref{lem:inverse}).
\end{proof}

\begin{thm}[H\"older regularity]
\label{prop:hoelder_reg_a}
Let $a^{ij}\in C^{k,\gamma}(B_1^+,\R^{(n+1)\times(n+1)}_{sym})$ with $k\geq 1$ and $\gamma\in (0,1)$. Let $w:B_1^+\rightarrow \R$ be a solution of the variable coefficient thin obstacle problem with metric $a^{ij}$. Then locally $\Gamma_{3/2}(w)$ is a $C^{k+1,\gamma}$ graph.
\end{thm}

\begin{proof}
\emph{Step 1: Almost optimal regularity.}
We apply the implicit function theorem to $G_a: X_{\delta,\epsilon}\rightarrow Y_{\delta,\epsilon}$ with $\delta$ and $\epsilon$ chosen such that $\epsilon\in (0,\gamma/2)$, $\delta=1-\epsilon \in (0,1)$ (as explained above, for $k=1$ we here interpret the lower order term as a function of $y$ in the linearization). 
We note that as a consequence of Proposition~\ref{prop:reg_a}, for $v\in X_{\delta,\epsilon}$, $G_a(v)$ (interpreted as the function $G_{\cdot}(v): \R^{n-1}\ni a \mapsto G_a(v)\in Y_{\delta,\epsilon}$) is $C^{k-1,\gamma-\epsilon}$ in $a$. Thus, the implicit function theorem yields a unique solution $\tilde{w}_a$ in a neighborhood $B''_{\epsilon_0}(0)\times \mathcal{U}$ of $(0,0)\in \R^{n-1}\times X_{\delta,\epsilon}$ (c.f. Proposition \ref{prop:reg_a}). Moreover, the map $\R^{n-1}\ni a\mapsto \tilde{w}_a\in X_{\delta,\epsilon}$ is $C^{k-1,\gamma-\epsilon}$. 
Hence, for all multi-indices $\beta=(\beta'',0,0)$ with $|\beta|=k-1$,
\begin{align*}
\left\| \frac{\partial^{\beta}_a \tilde{w}_{a_1}- \partial^{\beta}_{a} \tilde{w}_{a_2}}{|a_1-a_2|^{\gamma-\epsilon}}\right\|_{X_{\delta,\epsilon}}\leq C \left\| \partial_a^{\beta} \frac{G_{a_1}(\tilde{w}_{a_1})-G_{a_2}(\tilde{w}_{a_1})}{|a_1-a_2|^{\gamma-\epsilon}} \right\|_{Y_{\delta,\epsilon}} <\infty.
\end{align*}
In particular,
\begin{align}
\label{eq:differences}
\left[ \frac{\partial^{\beta}_a \p_{in}\tilde{w}_{a_1}- \partial^{\beta}_{a} \p_{in}\tilde{w}_{a_2}}{|a_1-a_2|^{\gamma-\epsilon}}\right]_{\dot{C}^{0,\delta}(P)}\leq C \left\| \partial_a^{\beta} \frac{G_{a_1}(\tilde{w}_{a_1})-G_{a_2}(\tilde{w}_{a_1})}{|a_1-a_2|^{\gamma-\epsilon}} \right\|_{Y_{\delta,\epsilon}}  <\infty.
\end{align}
Since by Lemma~\ref{prop:psi2} $w_a=v_a-v\in \mathcal{U}$ if $a\in B''_{\epsilon_1}(0)$ for some sufficiently small radius $\epsilon_1$, the local uniqueness of the solution implies that $\tilde{w}_a=w_a$ for $a\in B''_{\epsilon_1}(0)$. Thus, $v_a=v+w_a=v+\tilde{w}_a$ is $C^{k-1,\gamma-\epsilon}$ in $a$. Combined with (\ref{eq:differences}) this in particular implies that for any multi-index $\beta=(\beta'',0,0)$ with $|\beta|= k-1$
\begin{align}
\label{eq:reg_tan}
\left[\frac{\partial^{\beta}_a \p_{in}v_{a_1} - \partial^{\beta}_a \p_{in}v_{a_2}}{|a_1 - a_2|^{\gamma - \epsilon}} \right]_{\dot{C}^{0,\delta}(P)}\leq  C <\infty.
\end{align} 
Recalling that the $a$-derivative corresponds to a tangential derivative in $\mathcal{B}_{1/2}$ and the fact that Hölder and Hölder-Zygmund spaces agree for non-integer values (c.f. \cite{Triebel}), this implies that for any multi-index $\beta=(\beta'',0,0)$ with $|\beta|\leq k-1$, $\p^\beta\p_{in} v\in C^{1,\gamma-2\epsilon}(P\cap \mathcal{B}_{1/2})$. By the characterization of the free boundary as $\Gamma_w = \{x \in B_{1}'| x_n = -\p_{n}v(x'',0,0)\}$ this implies that $\Gamma_w$ is a  $C^{k+1,\gamma-2\epsilon}$ graph for any $\epsilon\in (0,\gamma/2)$. As $\epsilon>0$ can be chosen arbitrarily small, this completes the proof of the \emph{almost} optimal regularity result.\\

\emph{Step 2: Optimal regularity.}
In order to infer the \emph{optimal} regularity result, we argue by scaling and our previous estimates. More precisely, we have that
\begin{equation*}
\label{eq:est}
\begin{split}
[\Delta_{a}^{\gamma-\epsilon}\p_{in}\partial_a^{\beta} \tilde{w}_a ]_{\dot{C}^{0,\delta}}
&\leq \|\Delta_{a}^{\gamma-\epsilon}\p_{in}\partial_a^{\beta} \tilde{w}_a\|_{X_{\delta,\epsilon}} \leq C \| \Delta_{a}^{\gamma-\epsilon}\partial_a^{\beta} G_a(\tilde{w}_{a_1}) \|_{Y_{\delta,\epsilon}}\\
& \leq C \left( \| \Delta_{a}^{\gamma-\epsilon}\partial_a^{\beta} G_a^1(\tilde{w}_{a_1}) \|_{Y_{\delta,\epsilon}} + \| \Delta_{a}^{\gamma-\epsilon}\partial_a^{\beta} G_a^2(\tilde{w}_{a_1}) \|_{Y_{\delta,\epsilon}}\right).
\end{split}
\end{equation*}
Here $G^{1}_a(\cdot)$ is the term that originates from $ F^1(v,y)= \sum\limits_{i,j=1}^{n+1}\tilde{a}^{ij}(y)G^{ij}(v)$ and $G^{2}_a(\cdot)$ is the contribution that originates from the lower order contribution 
$$ F^2(v,y)= -J(v(y))\left(\sum_{j=1}^{n-1}\tilde{b}^j(y)\p_j v(y)+\tilde{b}^n(y)y_n+\tilde{b}^{n+1}(y)y_{n+1}\right).$$
The notation $\D_a^{\gamma-\epsilon}$ denotes the difference quotient in $a$ with exponent $\gamma-\epsilon$.
We now consider the norms on the right hand side of (\ref{eq:est}) more precisely and consider their rescalings. A typical contribution of $\| \Delta_{a}^{\gamma-\epsilon}\partial_a^{\beta} G_a^2(\tilde{w}_{a_1}) \|_{Y_{\delta,\epsilon}}$ for instance is
\begin{align*}
[ r^{-(1+2\delta-\epsilon)}\Delta_{a}^{\gamma-\epsilon} \partial^{\beta}_a \tilde{b}^{j}_a J(v_{a_1})\p_j v_{a_1} ]_{C^{0,\epsilon}_{\ast}(\mathcal{B}_{2}^+)}.
\end{align*}
We focus on this contribution and on the case $k=1$. The other terms can be estimated by using similar ideas. We consider the rescaled function $v_{\lambda,a}(y)$, where $v_{\lambda}(y):=\frac{v(\delta_{\lambda}(y))}{\lambda^{3}}$ (with $\delta_\lambda(y)=(\lambda^2y'',\lambda y_n,\lambda y_{n+1})$) and $v_{\lambda,a}(y):= v_{\lambda}(\Phi_a(y))$. The function $w_{\lambda, a}(y):=v_{\lambda,a}(y)-v_{\lambda}(y)$ is defined as its analogue from above. It is compactly supported in $\mathcal{B}_{3/4}^+$ (by definition of $\Phi_a$) and the functions $v_{\lambda}$ and $w_{\lambda,a}$ satisfy similar equations as $v, w_a$. Thus, we may apply estimate (\ref{eq:differences}) to $w_{\lambda,a}$. Inserting $\delta =1 -\epsilon$, using the support condition for $w_{\lambda,a}$ yields (with slight abuse of notation, as there are additional right hand side contributions, which however by the compact support assumption on $w_{\lambda,a}$ have the same or better scaling)
\begin{align*}
 [\p_{ijn}v_{\lambda}]_{C^{0,\gamma-2\epsilon}(\mathcal{B}_{1}^+\cap P)} &\leq C \lambda^{-1}[r^{-3+3\epsilon}J(v)|_{\delta_{\lambda}(y)}\p_j v|_{\delta_{\lambda}(y)}]_{C^{0,\epsilon}_{\ast}}[\tilde{b}^{j}|_{\delta_{\lambda}(y)}]_{C^{0,\gamma}(\mathcal{B}_{2}^+)}\\
& \leq C \lambda^{-1}[r^{-3+\epsilon}J(v)|_{\delta_{\lambda}(y)}\p_j v|_{\delta_{\lambda}(y)}]_{C^{0,\epsilon}_{\ast}(\mathcal{B}_{2}^+)}[\tilde{b}^{j}|_{\delta_{\lambda}}]_{C^{0,\gamma}(\mathcal{B}_{2}^+)}.
\end{align*}
Comparing this to the left hand side of the estimate and rescaling both sides of the inequality therefore amounts to
\begin{align*}
\lambda^{2+2\gamma-4\epsilon} [\p_{ijn}v]_{C^{0,\gamma-2\epsilon}(\mathcal{B}_{ \lambda}^+\cap P)} & \leq C \lambda^{2+2\gamma}[r^{-3+\epsilon}J(v)\p_j v]_{C^{0,\epsilon}_{\ast}(\mathcal{B}_{2 \lambda}^+)}[\tilde{b}^{j}]_{C^{0,\gamma}(\mathcal{B}_{2  \lambda}^+}),
\end{align*}
which yields
\begin{align*}
 [\p_{ijn}v]_{C^{0,\gamma-2\epsilon}(\mathcal{B}_{ \lambda}^+\cap P)} & \leq C \lambda^{4\epsilon}[r^{-3+\epsilon}J(v)\p_j v]_{C^{0,\epsilon}_{\ast}(\mathcal{B}_{2\lambda}^+)}[\tilde{b}^{j}]_{C^{0,\gamma}(\mathcal{B}_{2 \lambda}^+)}.
\end{align*}
As a result considering two points $x,y\in P$ with $|x-y|= \lambda^2$, yields
\begin{align*}
\frac{|\p_{ijn}v(x)-\p_{ijn}v(y)|}{|x-y|^{\gamma-2\epsilon}} \leq C \lambda^{4\epsilon} = C |x-y|^{2\epsilon}.
\end{align*}
Thus,
\begin{align*}
[\p_{ijn}v]_{C^{0,\gamma}(\mathcal{B}_{ \lambda}^+\cap P)} \leq C,
\end{align*}
which proves the optimal regularity result.
\end{proof}

\begin{rmk}[$\gamma=1$]
\label{rmk:gamma=1} 
As expected from elliptic regularity, we can only deduce the full $C^{k+1,\gamma}$ regularity of the free boundary in the presence of $C^{k,\gamma}$ metrics, if $\gamma=1$. This is essentially a consequence of the elliptic estimates of Proposition \ref{prop:error_gain2}. On a technical level this is exemplified in the fact that in Step 2 of the previous proof, we for instance also have to deal with the term $(\D_a^{\gamma-\epsilon} a^{n,n})\p_{n+1,n+1}v_{\lambda}$ with the expansion $\p_{n+1,n+1}v(y) = a_1(y'')y_n + r^{1+2\delta-\epsilon}C_{n+1,n+1}(y)$ with $\delta\in (0,1)$. As the coefficients $a_1(y'')$ are in general not better than $C^{0,\delta}$, we do not have the full gain of $\lambda^{4\epsilon}$ if $\gamma=1$.
\end{rmk}

\begin{rmk}[Optimal regularity]
\label{rmk:optreg}
Let us comment on the optimality of the gain of the free boundary regularity with respect to the regularity of the metric $a^{ij}$:
Proposition \ref{prop:bulk_eq} in combination with our linearization results (c.f. Example \ref{ex:linear} and Section \ref{sec:grushin}) illustrates that $F$ can be viewed as a nonlinear perturbation of the degenerate, elliptic (second order) Baouendi-Grushin operator with metric $a^{ij}$. As such, we can not hope for a gain of more than \emph{two orders} of regularity for $v$ compared to the regularity of the metric $a^{ij}$ (by interior regularity in appropriate Hölder spaces, c.f. Section \ref{sec:holder}). Hence, for the regular free boundary we can in general hope for a gain of at most \emph{one order} of regularity with respect to the regularity of the metric. This explains our expectation that the regularity results from Theorem \ref{thm:higher_reg} are sharp higher order regularity results.\\

By a simple transformation it is possible to construct an example to the sharpness of this claim: 
In $\R^3$ the function  $w(x_1,x_2,x_3)=\Ree((x_2-x_1)/\sqrt{2}+ix_3)^{3/2}$ is a solution to the thin obstacle problem $\Delta w=0$ in $\R^3_+$ with the free boundary $\Gamma_w=\{(x_1,x_2,0)\in B'_1: x_2=x_1\}$. Applying a transformation of the form $$y(x):=(x_1,h(x_2),x_3),$$
with $h$ being a $W^{k+1,p}$ diffeomorphism from $(-1,1)$ to $(-1,1)$ yields that $\tilde{w}(y):=\Ree((h^{-1}(y_2)-y_1)/\sqrt{2}+iy_3)^{3/2}$ solves the variable coefficient thin obstacle problem 
\begin{align*}
\p_{11}\tilde{w}+ \p_2(h'(x_2)\p_2 \tilde{w})+\p_{33}w=0 \mbox{ in } B'_1,
\end{align*}
with Signorini conditions on $B'_1$. We note that the free boundary of $\tilde{w}$ is given by the graph $\Gamma_{\tilde{w}}=\{(y_1,y_2,0)\in B'_1: y_2=h(y_1)\}$.   If $h$ is not better than $W^{k+1,p}$ regular, the coefficients in the bulk equation are no more than $W^{k,p}$ regular. The free boundary is $W^{k+1,p}$ regular. Since it is a graph, it does not admit a more regular parametrization.\\

We further note that our choice of function spaces was crucial in deducing the full gain of regularity for the free boundary with respect to the metric. Indeed, considering the equation (\ref{eq:nonlineq1}), we note that also second order derivatives of the metric are involved. Yet, in order to deduce regularity of the free boundary (which corresponds to \emph{partial} regularity of the Legendre function $v$) this loss of regularity does not play a role as it is a ``lower order bulk term'' (this is similar in spirit to the gain of regularity obtained in boundary Harnack inequalities).
\end{rmk}

Finally, we give the argument for the analyticity of the free boundary in the case that the coefficients $a^{ij}$ are analytic:

\begin{thm}[Analyticity]
\label{prop:analytic}
Let $a^{ij}:B_{1}^+ \rightarrow \R^{(n+1)\times (n+1)}_{sym}$ be an analytic tensor field. Let $w:B_{1}^+ \rightarrow \R$ be a solution of the variable coefficient thin obstacle problem with metric $a^{ij}$. Then locally $\Gamma_{3/2}(w)$ is an analytic graph.
\end{thm}

\begin{proof}
This follows from the analytic implicit function theorem (c.f. \cite{Dei10}). Indeed, due to Proposition \ref{prop:reg_a}, $F_a$ is a real analytic function in $a$ and hence also $G_a$ is a real analytic function in $a$. Applying the analytic implicit function theorem similarly as in Step 1 of the previous proof, we obtain an in $a$ analytic function $\tilde{w}_a$. As before, this coincides with our function $w_a$. Therefore $w_a$ depends analytically on $a$. As differentiation with respect to $a$ however directly corresponds to differentiation with respect to the tangential directions $y''$, $w$ (and hence $v$) is an analytic function in the tangential variables. 
\end{proof}

\begin{rmk}[Regularity in the normal directions]
In Theorems \ref{prop:hoelder_reg_a} and \ref{prop:analytic} we proved partial analyticity for the Legendre function $v$: We showed that in the \emph{tangential} directions, the regularity of $v$ in a quantitative way matches that of the metric (i.e. a $C^{k,\gamma}$ metric yields $C^{k+1,\gamma}$ regularity for $\p_n v(y'',0,0)$). Although this suffices for the purposes of proving regularity of the (regular) free boundary, a natural question is whether it is also possible to obtain corresponding higher regularity for $v$ in the \emph{normal} directions $y_n, y_{n+1}$. Intuitively, an obstruction for this stems from working in the corner domain $Q_+$. That this set-up of a corner domain really imposes restrictions on the normal regularity can be seen by checking a compatibility condition: As we are considering an expansion close to the regular free boundary point, we know that the Legendre function asymptotically behaves like a multiple of the function $v_0(y)=-(y_1^3 - 3 y_1 y_2^2)$. If additional regularity were true in the normal directions, we could expand the Legendre function $v$ further, for instance into a fifth order polynomial (with symmetry obeying the mixed Dirichlet-Neumann boundary conditions), which has $v_0$ as its leading order expansion. Hence, working in the two-dimensional corner domain $Q_+:=\{y_1\geq 0, y_2\leq 0\}$, we make the ansatz that
\begin{equation}
\label{eq:ansatz}
v(y) = -(y_1^3 - 3y_1 y_2^2) + c_1 y_1^4 + c_2 y_1^2 y_2^2 + c_3 y_1 y_2^3 + c_4 y_1^5 + c_5 y_1^3 y_1^2 + c_6 y_1^2 y_2^3 + c_7 y_1 y_2^4 + h.o.t,
\end{equation}
where $h.o.t$ abbreviates terms of higher order. We seek to find conditions on the metric $a^{ij}$ which ensure that such an expansion for $v$ up to fifth order exists. Without loss of generality we may further assume that
\begin{align*}
a^{ij}(0) = \delta^{ij}, \ a^{12}(x_1,0)=a^{21}(x_1,0)=0,
\end{align*}
which corresponds to a normalization at zero and the off-diagonal condition on the plane $\{x_2=0\}$. Transforming the equation
\begin{align*}
\p_i a^{ij} \p_j w = 0 \mbox{ in } \R^2_+,
\end{align*}
into the Legendre-Hodograph setting with the associated Legendre function $v$ yields
\begin{align*}
&a^{11}(-\p_1 v,  -\p_2 v) \p_{22}v + a^{22}(-\p_1 v,  -\p_2 v) \p_{11}v - 2 a^{12}(-\p_1 v, - \p_2 v) \p_{12}v\\
& - J(v)[\p_1 a^{11}(-\p_1 v,  -\p_2 v) +\p_2 a^{12}(-\p_1 v, - \p_2 v) ] y_1 \\
&- J(v)[\p_1 a^{12}(-\p_1 v, - \p_2 v) +\p_2 a^{22}(-\p_1 v,  -\p_2 v) ] y_2 =0 \mbox{ in } Q_+:= \{y_1,y_1 \geq 0\}.
\end{align*}
Here $J(v) = \det\begin{pmatrix} \p_{11}v & \p_{12}v\\
\p_{21}v & \p_{22}v \end{pmatrix}$.
Carrying out a Taylor expansion of the metric thus gives
\begin{align*}
&\Delta v -2\left(\p_2a^{12}(0)(-\p_2 v)\right)\p_{12}v\\
&+\left((-\p_1 v)\p_1a^{11}(0)+(-\p_2 v)\p_2a^{11}(0)\right)\p_{22}v+\left((-v_1)\p_1a^{22}(0)+(-v_2)\p_2a^{22}(0)\right)\p_{11}v\\
&-\det\begin{pmatrix}
\p_{11}v & \p_{12}v\\
\p_{21}v & \p_{22}v
\end{pmatrix}\left((\p_1a^{11}(0)+\p_2a^{21}(0))y_1+\p_2a^{22}(0)y_2\right)+h.o.t.=0.
\end{align*}
Inserting the ansatz (\ref{eq:ansatz}) into this equation, matching all terms of order up to three and using the off-diagonal condition, eventually yields the compatibility condition
\begin{align*}
\p_2(a^{11} + a^{22})(0)=0.
\end{align*}
Due to our normalization this necessary condition for having a polynomial expansion up to degree five can thus be formulated as
\begin{align*}
(\p_2 \det(a^{ij}))(0)=0.
\end{align*}
In particular this shows that on the transformed side, i.e. in the Legendre-Hodograph variables, one cannot expect arbitrary high regularity for $v$ in the normal directions $y_n, y_{n+1}$ in general. Compatibility conditions involving the metric $a^{ij}$ have to be satisfied to ensure this.
\end{rmk}

\section{$W^{1,p}$ Metrics and Nonzero Obstacles}
\label{sec:W1p}
In this section we consider the previous set-up in the presence of inhomogeneities $f\in L^p$ and possibly only Sobolev regular metrics. More precisely, in this section we assume that  $a^{ij}:B_1^+ \rightarrow \R^{(n+1)\times (n+1)}_{sym}$ is a uniformly elliptic $W^{1,p}$, $p\in (n+1,\infty]$, metric and consider a solution $w$ of the variable coefficient thin obstacle problem with this metric:
\begin{equation}
\label{eq:inhom}
\begin{split}
\p_{i} a^{ij} \p_j w & = f \mbox{ in } B_1^+,\\
 w \geq 0,\ \p_{n+1}w \leq 0, \ w\p_{n+1}w&=0 \mbox{ on } B_1'.
\end{split}
\end{equation}
We will discuss two cases: 
\begin{itemize}
\item[(1)] $f=0$, $a^{ij}\in W^{1,p}$ with $p\in (n+1,\infty]$,
\item[(2)] $a^{ij}\in W^{1,p}$, $f\in L^p$ with $p\in (2(n+1),\infty]$. 
\end{itemize}
In both cases all the normalization conditions (A1)-(A7) from Section \ref{sec:conventions} as well as the asymptotic expansions (c.f. Proposition~\ref{prop:asym2}) hold. We observe that case (2) in particular contains the setting with non-flat obstacles.

\subsection{Hodograph-Legendre transformation for $W^{1,p}$ metrics}
\label{sec:ext}

In the sequel, we discuss how the results from Sections \ref{sec:asymp}- \ref{sec:Legendre} generalize to the less regular setting of $W^{1,p}$, $p\in (n+1,\infty]$, metrics. We note that in this case the solution $w$ is only $W^{2,p}_{loc}(B_1^+\setminus\Gamma_w)$ regular away from the free boundary $\Gamma_w$. Thus, our Hodograph-Legendre transformation method from the previous sections does not apply directly (as it relies on the pointwise estimates of $D^2v$, and hence $D^2w$). Thus, a key ingredient in our discussion of this set-up will be the splitting result, Proposition 3.9, from \cite{KRSI}. In order to apply it, we extend $w$ and the metric $a^{ij}$ from $B_1^+$ to $B_1$ by an even reflection as in \cite{KRSI}. We now split our solution into two components, $w=u+\tilde{u}$, where $\tilde{u}$ solves
\begin{align}
\label{eq:split1}
a^{ij}\p_{ij}\tilde{u}-\dist(x,\Gamma_w)^{-2}\tilde{u}=f - (\p_i a^{ij})\p_j w \text{ in } B_1\setminus \Lambda_w, \quad \tilde{u}=0\text{ on }\Lambda_w,
\end{align}
and the function $u$ solves 
\begin{align}
\label{eq:split2}
a^{ij}\p_{ij}u=-\dist(x,\Gamma_w)^{-2}\tilde{u}\text{ in } B_1\setminus \Lambda_w, \quad u=0\text{ on } \Lambda_w.
\end{align}
As in \cite{KRSI} the intuition is that $\tilde{u}$ is a ``controlled error'' and that $u$ captures the essential behavior of $w$. Moreover, as we will see later, $u$ will be $C^{2,1-\frac{n+1}{p}}_{loc}$ regular away from $\Gamma_w$ and that $\Gamma_w = \Gamma_u$ (c.f. the discussion below Lemma \ref{lem:lower1'}). Thus, in the sequel, we will apply the Hodograph-Legendre transformation to the function $u$.\\

In order to support this intuition, we recall the positivity of $\p_e u$ as well as the fact that $u$ inherits the complementary boundary conditions from $w$.

\begin{lem}[\cite{KRSI}, Lemma 4.11]
\label{lem:lower1'}
Let $a^{ij}\in W^{1,p}(B_1^+, \R^{(n+1)\times(n+1)}_{sym})$ and let $f\in L^p(B_1^+)$. Suppose that either 
\begin{itemize}
\item[(1)] $p\in (n+1,\infty]$ and $f=0$ or, 
\item[(2)] $p\in (2(n+1),\infty]$. 
\end{itemize}
Let $w:B_1^+ \rightarrow \R$ be a solution to the thin obstacle problem with inhomogeneity $f$, and let $u$ be defined as at the beginning of this section.  Then we have that $u\in C^{2,1-\frac{n+1}{p}}_{loc}(B_1^+\setminus \Gamma_w)\cap C^{1,\min\{\frac{1}{2},1-\frac{n+1}{p}\}}_{loc}(B_1^+)$. Moreover, there exist constants $c, \eta>0$ such that for $e\in \mathcal{C}'_\eta(e_n)$, $\p_eu $ satisfies the lower bound
\begin{align*}
\p_{e}u (x) \geq c\dist(x,\Lambda_w)\dist(x,\Gamma_w)^{-\frac{1}{2}}.
\end{align*}
A similar statement holds for $\p_{n+1}u$ if $\Lambda_w$ is replaced by $\Omega_w$.
\end{lem}
 
We remark that the lower bound in Lemma 3.13 does not necessarily hold for $\p_ew$. This is due to the insufficient decay properties of $\p_e \tilde{u}$ in the decomposition $\p_ew = \p_e\tilde{u}+\p_eu$. More precisely, the decay of $\p_e\tilde{u} $ to $\Lambda_w$ is in general only of the order $\dist(x,\Lambda_w)^{1-\frac{n+1}{p}}$, which cannot be controlled by $\dist(x,\Lambda_w)$.\\

We further note that the symmetry of $u$ about $x_{n+1}$ and the regularity of $u$ imply that $\p_{n+1}u=0$ in $B'_1\setminus \Lambda_w$. In particular, this yields the complementary boundary conditions:
\begin{align*}
u \p_{n+1}u = 0 \mbox{ on } B_1'.
\end{align*}
Most importantly, Lemma \ref{lem:lower1'} combined with the previous observations on the behavior of $\nabla u$ on $B_1'$ implies that $\Gamma_w = \Gamma_u$. Hence, seeking to investigate $\Gamma_w$, it suffices to study $u$ and its boundary behavior. In this context, Lemma \ref{lem:lower1'} plays a central role as it allows us to deduce the sign conditions for $\p_e u$ and $\p_{n+1}u$ which are crucial in determining the image of the Legendre-Hodograph transform which we will associate with $u$.\\

In accordance with our intuition that $\tilde{u}$ is a ``controlled error", the function $u$ inherits the asymptotics of the solution $w$ around $\Gamma_w$. As in Proposition \ref{prop:invertibility} in Section \ref{sec:Hodo}, this is of great importance in proving the invertibility of the Legendre-Hodograph transform which we will associate with $u$. We formulate the asymptotic expansions in the following proposition:

\begin{prop}
\label{prop:improved_reg1}
Let $a^{ij}\in W^{1,p}(B_1^+, \R^{(n+1)\times(n+1)}_{sym})$ and let $f\in L^p(B_1^+)$. 
Suppose that either 
\begin{itemize}
\item[(1)] $p\in (n+1,\infty]$ and $f=0$ or, 
\item[(2)] $p\in (2(n+1),\infty]$. 
\end{itemize}
Let $w:B_1^+ \rightarrow \R$ be a solution to the thin obstacle problem with inhomogeneity $f$, and let $u$ be defined as at the beginning of this section. There exist small constants $\epsilon_0>0$ and $c_\ast>0$ depending on $n,p$ such that if 
\begin{itemize}
\item[(i)] $ \|w-w_{3/2}\|_{C^1(B_1^+)}\leq \epsilon_0$,
\item[(ii)]$ \|\nabla a^{ij}\|_{L^p(B_1^+)}+\|f\|_{L^p(B_1^+)}\leq c_\ast,$
\end{itemize}
then the asymptotics (i)-(iii) in Proposition~\ref{prop:asym2} hold for $\p_eu$, $\p_{n+1}u$ and $u$. The exponent $\alpha$ in the error term satisfies $\alpha\in (0,1-\frac{n+1}{p}]$ in case (1) and $\alpha\in (0,\frac{1}{2}-\frac{n+1}{p}]$ in case (2). 
\end{prop}

\begin{proof}
By the growth estimate of Remark 3.11 in \cite{KRSI} we have that
\begin{equation}
\label{eq:auxv}
\begin{split}
|\tilde{u}(x)|& \lesssim c_\ast \dist(x,\Lambda_w)\dist(x,\Gamma_w)^{\frac{3}{2}-\frac{n+1}{p}} \text{ in case (1);}\\
|\tilde{u}(x)|& \lesssim c_\ast  \dist(x,\Lambda_w)\dist(x,\Gamma_w)^{1-\frac{n+1}{p}}  \text{ in case (2).}
\end{split}
\end{equation}
In particular this implies that 
\begin{align*}
&|\tilde{u}(x)|\lesssim c_\ast \dist(x,\Gamma_w)^{\frac{3}{2}+\delta_0},\quad
|\nabla \tilde{u}(x)|\lesssim c_\ast \dist(x,\Gamma_w)^{\frac{1}{2}+\delta_0},\\
&\text{where }\delta_0=\left\{\begin{array}{ll}
1-\frac{n+1}{p} &\text{ if } p\in (n+1,2(n+1)],\\
\frac{1}{2}-\frac{n+1}{p} &\text{ if } p\in (2(n+1),\infty].
\end{array}
\right.
\end{align*}
Since $\delta_0>0$, in both cases the functions $\tilde{u}$ and $\nabla \tilde{u}$ are of higher vanishing order at $\Gamma_w$ compared to the leading term in the corresponding asymptotics of $w$ and $\nabla w$ (which are of order $\dist(x,\Gamma_w)^{3/2}$ and $\dist(x,\Gamma_w)^{1/2}$).
\end{proof}

In addition to these results the second order asymptotics for $u$ (not for the whole function $w$) remain valid under the conditions of Proposition \ref{prop:improved_reg1}. More precisely we have the following result:

\begin{prop}
\label{prop:improved_reg'}
Under the same assumptions as in Proposition~\ref{prop:improved_reg1}, we have the following: 
For each $x_0\in \Gamma_w\cap B^+_{1/2}$, for all $x$ in an associated non-tangential cone $\mathcal{N}_{x_0}$ and for all multi-indeces $\beta$ with $|\beta|\leq 2$,
\begin{align*}
\left|\p^\beta u(x)-\p^\beta \mathcal{W}_{x_0}(x)\right|& \leq C_{n,p,\beta} \max\{\epsilon_0, c_*\} |x-x_0|^{\frac{3}{2}+\alpha-|\beta|},\\
\left[\p^\beta u-\p^\beta \mathcal{W}_{x_0}\right]_{\dot{C}^{0,\gamma}(\mathcal{N}_{x_0}\cap (B_{3\lambda/4}(x_0)\setminus B_{\lambda/2}(x_0)))}& \leq C_{n,p,\beta}  \max\{\epsilon_0, c_*\} \lambda^{\frac{3}{2}+\alpha-\gamma-|\beta|}.
\end{align*}
Here $\gamma=1-\frac{n+1}{p}$ and $\lambda \in (0,1)$.
\end{prop}

\begin{proof}
We only prove the case of $|\beta|=2$, the other cases are already contained in Proposition \ref{prop:improved_reg1}. Since the arguments for case (1) and (2) are similar we only prove case (2), i.e. $p\in (2(n+1),\infty]$. As in Proposition \ref{prop:improved_reg} the result follows from scaling.
We consider the function
\begin{align*}
\bar{u}(x):= \frac{u(x_0+ \lambda x)- \mathcal{W}_{x_0}(x_0+\lambda x)}{\lambda^{3/2 + \alpha}},
\end{align*}
and note that it satisfies
\begin{align*}
a^{ij}(x_0 + \lambda \cdot) \p_{ij }\bar{u} = \tilde{G} + \tilde{g}_1, \quad \ell\in\{1,\dots,n\},
\end{align*}
for $x\in \mathcal{N}_{0}\cap (B_{1}\setminus B_{1/4})$ and $\mathcal{N}_0:= \{x\in B_{1/4}^+| \ \dist(x,\Gamma_{w_{x_0,\lambda}}) > \frac{1}{2}|x|\}$. Here
\begin{align*}
\tilde{G}(x) & := -\lambda^{1/2 - \alpha}\dist(x_0+\lambda x,\Gamma_w)^{-2}\tilde{u}(x_0+\lambda x),\\
&=-\lambda^{-\frac{3}{2}-\alpha}\dist(x,\Gamma_{w_{x_0,\lambda}})^{-2}\tilde{u}(x_0+\lambda x),\\
\tilde{g}_1(x) & := \lambda^{1/2-\alpha} (a^{ij}(x_0 + \lambda x)-a^{ij}(x_0)) \partial_{ij} \mathcal{W}_{x_0}(x_0 +\lambda x).
\end{align*}
In the definition of $\tilde{g}_1$ we have used that $a^{ij}(x_0)\p_{ij}\mathcal{W}_{x_0}=0$ in $ \mathcal{N}_{x_0}\cap (B_{\lambda}(x_0)\setminus B_{\lambda/4}(x_0))$. Using \eqref{eq:auxv} and the regularity of $\tilde{u}$ (and abbreviating $\gamma=1-\frac{n+1}{p}$) yields
\begin{align*}
\|\tilde{G}\|_{C^{0,\gamma}(\mathcal{N}_{0}\cap (B_{1}\setminus B_{1/4}))}\leq C \lambda^{-\frac{3}{2}-\alpha} \lambda^{2-\frac{n+1}{p}}=C\lambda^{\frac{1}{2}-\alpha-\frac{n+1}{p}}.
\end{align*}
Recalling the $C^{0,1-\frac{n+1}{p}}$ regularity of $a^{ij}$ and the explicit expression of $\mathcal{W}_{x_0}$, we estimate
\begin{align*}
\|\tilde{g}_1\|_{C^{0,\gamma}(\mathcal{N}_{0}\cap (B_{1}\setminus B_{1/4}))}\leq C\lambda^{\frac{1}{2}-\alpha}\lambda^{1-\frac{n+1}{p}}\lambda^{-\frac{1}{2}}=C\lambda^{1-\alpha-\frac{n+1}{p}}.
\end{align*}
Hence, applying the interior Schauder estimate to $\bar u$ we obtain
\begin{align*}
\|\bar u\|_{C^{2,\gamma}(\mathcal{N}_{0}\cap (B_{3/4}\setminus B_{1/2}))}&\leq C\left(\|\tilde{ G}\|_{C^{0,\gamma}(\mathcal{N}_0\cap (B_{1}\setminus B_{1/4}))}+\|\tilde{g}_1\|_{C^{0,\gamma}(\mathcal{N}_0\cap (B_{1}\setminus B_{1/4}))}\right.\\
&\quad \left.+\|\bar u\|_{L^\infty(\mathcal{N}_0\cap (B_{1}\setminus B_{1/4}))}\right)\\
&\leq C\left(\lambda^{\frac{1}{2}-\alpha-\frac{n+1}{p}}+1 \right)\leq C \quad (\text{since }\alpha\in (0,\frac{1}{2}-\frac{n+1}{p}]).
\end{align*} 
Scaling back, the error estimates become
\begin{align*}
\|\p_e u -\p_e \mathcal{W}_{x_0}\|_{L^\infty(\mathcal{N}_{x_0}\cap (B_{3\lambda/4}(x_0)\setminus B_{\lambda/2}(x_0)))}&\leq C\lambda^{\frac{1}{2}-\alpha},\\
\|\p_{ee'}u-\p_{ee'}\mathcal{W}_{x_0}\|_{L^\infty(\mathcal{N}_{x_0}\cap (B_{3\lambda/4}(x_0)\setminus B_{\lambda/2}(x_0)))}
&\leq C\lambda^{-\frac{1}{2}-\alpha},\\
[\p_{ee'}u-\p_{ee'}\mathcal{W}_{x_0}]_{\dot{C}^{0,\gamma}(\mathcal{N}_{x_0}\cap (B_{3\lambda/4}(x_0)\setminus B_{\lambda/2}(x_0)))}
&\leq C\lambda^{-\frac{1}{2}-\alpha-\gamma}.
\end{align*}
Since this holds for every $\lambda\in (0,1)$, we obtain the asymptotic expansion for $\p_e w $ and $\p_{e e'}w$. The asymptotics for $\p_{ij}w$ with $i$ or $j=n+1$ are derived analogously. 
\end{proof}

Due to the above discussion, the associated Hodograph transform with respect to $u$,
$$
T(x):=(x'',\p_n u(x), \p_{n+1}u(x)),$$ 
still enjoys all the properties stated in Section \ref{sec:Hodo}. In particular, it is possible to define the associated Legendre function 
\begin{align}
\label{eq:Leg_split}
v(y):= u(x)-x_n y_n - x_{n+1} y_{n+1},
\end{align}
 for $x= T^{-1}(y)$. This function satisfies an analogous nonlinear PDE as the one from Section \ref{sec:Legendre}:
\begin{align*}
\tilde{F}(v,y) = g(y).
\end{align*}
Here,
\begin{equation}\label{eq:w1p}
\begin{split}
\tilde{F}(v,y)&=-\sum_{i,j=1}^{n-1}\tilde{a}^{ij}\det\begin{pmatrix}
\p_{ij}v& \p_{in}v & \p_{i,n+1}v\\
\p_{jn}v& \p_{nn}v & \p_{n,n+1}v\\
\p_{j,n+1}v & \p_{n,n+1}v &\p_{n+1,n+1}v
\end{pmatrix}\\
&+2\sum_{i=1}^{n-1}\tilde{a}^{i,n}\det\begin{pmatrix}
\p_{in}v & \p_{i,n+1}v\\
\p_{n,n+1}v & \p_{n+1,n+1}v
\end{pmatrix}+2 \sum_{i=1}^{n-1}\tilde{a}^{i,n+1}\det\begin{pmatrix}
\p_{i,n+1}v & \p_{in}v\\
\p_{n,n+1}v & \p_{nn}v
\end{pmatrix}\\
&+\tilde{a}^{nn}\p_{n+1,n+1}v+\tilde{a}^{n+1,n+1}\p_{nn}v-2\tilde{a}^{n,n+1}\p_{n,n+1}v,\\
\tilde{a}^{ij}(y)&:=a^{ij}(x)\big|_{x=(y'',-\p_nv(y),-\p_{n+1}v(y))},\\
J(v)& := \p_{nn} v(y) \p_{n+1,n+1}v(y) - (\p_{n,n+1}v(y))^2,\\
g(y) &:= -J(v(y)) \dist(T^{-1}(y), T^{-1}(P))^{-2}\tilde{u}(T^{-1}(y)),\quad P:=\{y_n=y_{n+1}=0\}.
\end{split}
\end{equation}  
From the asymptotics of $J(v)$ and \eqref{eq:auxv} for $\tilde{u}$ we have 
\begin{align}
\label{eq:w1p_g}
|g(y)|\leq \left\{ \begin{array}{ll}\
C(y_n^2+y_{n+1}^2)^{\frac{3}{2}-\frac{n+1}{p}} &\text{ if } p\in (n+1,\infty] \text{ and } f=0,\\
C(y_n^2+y_{n+1}^2)^{1-\frac{n+1}{p}} &\text{ if }p\in (2(n+1),\infty].
\end{array}\right.
\end{align}

The result of Proposition \ref{prop:improved_reg'} in combination with an argument as in the proof of Proposition \ref{prop:holder_v} also yields that $v\in X_{\alpha,\epsilon}$ for a potentially very small $\alpha>0$.\\

We summarize all this in the following Proposition:

\begin{prop}
\label{eq:bulk_new}
Let $a^{ij}\in W^{1,p}$  and let $f\in L^p$. 
Suppose that either 
\begin{itemize}
\item[(1)] $p\in (n+1,\infty]$ and $f=0$ or, 
\item[(2)] $p\in (2(n+1),\infty]$. 
\end{itemize}
Let $v:T(B_{1/2}^+) \rightarrow \R$ be the Legendre function associated with $u$ defined in (\ref{eq:Leg_split}). Then $v\in C^1(T(B_{1/2}^+) )\cap X_{\alpha,\epsilon}(\mathcal{B}_{r_0}^+)$ for some $\alpha\in (0,1)$ (which is the same as in Proposition \ref{prop:improved_reg'}) and it satisfies the fully nonlinear equation
\begin{align*}
\tilde{F}(v,y) = g(y).
\end{align*}
Here $\tilde{F}, g$ are as in (\ref{eq:w1p}) and $g$ satisfies the decay estimate (\ref{eq:w1p_g}).
\end{prop}

\begin{rmk}
\label{rmk:decay}
We note that the leading contribution in the decay estimate for $g$ originates from the decay behavior of $\tilde{u}$ in \eqref{eq:auxv}. Therefore, the decay of $g$ is influenced by $-(\p_{i}a^{ij})\p_j w$ and by the inhomogeneity $f$ from \eqref{eq:split1}. 
\end{rmk}

\subsection{Regularity of the free boundary}
\label{sec:free_boundary_reg_1}

In this section we discuss the implications of the results from Section \ref{sec:ext} on the free boundary regularity. In order to understand the different ingredients to the regularity results, we treat two different scenarios: First we address the setting of $W^{1,p}$ metrics with $p\in (n+1,\infty]$, and zero obstacles, i.e. with respect to Sections \ref{sec:HLTrafo} - \ref{sec:fb_reg} we present a result under even weaker regularity assumptions of the metric (c.f. Section \ref{subsec:w1p_zero}). Secondly, in Section \ref{sec:nonzero} we address the set-up with inhomogeneities. This in particular includes the case of non-zero obstacles. We treat this in the $W^{1,p}$ and the $C^{k,\gamma}$ framework.

\subsubsection{$W^{1,p}$ metrics without inhomogeneity}
\label{subsec:w1p_zero}

We now specialize to the setting in which $a^{ij}\in W^{1,p}$ with $p\in (n+1,\infty]$ and $f=0$ in (\ref{eq:inhom}). In this framework, we prove the following quantitative regularity result for the free boundary:

\begin{prop}[$C^{1,1-\frac{n+1}{p}}$ regularity] 
\label{prop:W1p}
Let $a^{ij}\in W^{1,p}$ with $p\in (n+1,\infty]$ and $f=0$. Let $w$ be a solution of (\ref{eq:inhom}) and assume that the normalizations (A1)-(A7) from Section \ref{sec:conventions} hold. Then, if $p<\infty$, $\Gamma_w$ is a $C^{1,1-\frac{n+1}{p}}(B_{1/2}')$ graph and if $p=\infty$, it is a $C^{1,1-}(B_{1/2}')$ graph.
\end{prop}

\begin{proof}
We prove this result similarly as in the case of $C^{1,\gamma}$ metrics but instead of working with the original solution $w$, we work with the modified function $u$ from Section~\ref{sec:ext}.\\
We begin by splitting $w=u+\tilde{u}$ as in Section \ref{sec:ext}. Moreover, we recall that by Lemma~\ref{lem:lower1'} (and the discussion following it) $\Gamma_w = \Gamma_{u}$. The Legendre function $v$ with respect to $u$ (c.f. (\ref{eq:Leg_split})) satisfies the nonlinear equation $\tilde{F}(v,y)=g(y)$ (c.f. \eqref{eq:w1p}), which in the notation in Section~\ref{subsec:improvement}, can be written as 
$$\tilde{F}(v,y)=\sum_{i,j=1}^{n+1}\tilde{a}^{ij}(y)G^{ij}(v)=g(y).$$
Furthermore, $g$ satisfies the decay condition (\ref{eq:w1p_g}). 
Keeping this in the back of our minds, we begin by proving analogues of Propositions \ref{prop:error_gain}, \ref{prop:error_gain2}.
To this end, we use a Taylor expansion to obtain that 
$$\tilde{a}^{ij}(y)=\tilde{a}^{ij}(y_0)+ \hat E^{y_0,ij}(y), \quad y\in \mathcal{B}_{1/2}^+(y_0),$$
for each $y_0\in P\cap \mathcal{B}_{1/2}$.
Due to the $C^{0,1-\frac{n+1}{p}}$ Hölder regularity of $a^{ij}$, for $\epsilon \in (0,1-\frac{n+1}{p})$ the function $\hat E^{y_0,ij}(y)$ satisfies 
\begin{equation}\label{eq:w1p_metric}
\begin{split}
\left\|d_G(\cdot,y_0)^{-2(1-\frac{n+1}{p})}\hat E^{y_0,ij}\right\|_{L^\infty(\mathcal{B}_{1/2}^+(y_0))}\\
+\left[d_G(\cdot,y_0)^{-2(1-\frac{n+1}{p}-\epsilon) }\hat E^{y_0,ij}\right]_{\dot{C}^{0,\epsilon}_\ast(\mathcal{B}_{1/2}^+(y_0))}\leq C.
\end{split}
\end{equation}
Recalling \eqref{eq:expand_v}, we expand the nonlinear function $G^{ij}(v)$ as 
$$G^{ij}(v)=G^{ij}(v_{y_0})+\p_{m_{k\ell}}G^{ij}(v_{y_0})\p_{k\ell}(v-v_{y_0})+\tilde{E}^{y_0,ij}_1(y),$$
where $v_{y_0}$ is the asymptotic profile of $v$ at $y_0$ and $\tilde{E}^{y_0,ij}_1(y)$ denote the same functions as in \eqref{eq:expand_v}.
Due to Proposition~\ref{prop:improved_reg1} the error term $\tilde{E}_1^{y_0,ij}$ satisfies the estimate (\ref{eq:Holderweight}). Hence, as in Step 1c of Proposition~\ref{prop:error_gain}, we can rewrite our nonlinear equation $\tilde{F}(v,y)=g(y)$ as 
\begin{align*} 
L_{y_0}v= L_{y_0}v_{y_0}+\tilde{f}
\end{align*}
with $L_{y_0}=\tilde{a}^{ij}(y_0)\p_{m_{k\ell}}G^{ij}(v_{y_0})\p_{k\ell}$ being the same leading term as in Remark~\ref{rmk:error_gain3} and 
\begin{align*}
\tilde{f}(y)&=-\tilde{a}^{ij}(y_0)\tilde{E}^{y_0,ij}_1(y)-\hat E^{y_0,ij}(y)G^{ij}(v_{y_0})+g(y).
\end{align*}
Due to the error bounds for $\tilde{E}_1^{y_0,ij}$ and $\hat E^{y_0,ij}(y)$, the linear estimate for $G^{ij}(v)$ and the estimate \eqref{eq:w1p_g} for $g$, we infer that
\begin{align*}
|\tilde{f}(y)|\leq C d_G(y,y_0)^{\eta_0}, \quad \eta_0=\min\left\{1+4\alpha, 3-\frac{2(n+1)}{p}\right\}.
\end{align*}
Hence, as long as $1+4\alpha<3-\frac{2(n+1)}{p}$ we bootstrap regularity as in Proposition~\ref{prop:error_gain2}, in order to obtain an increasingly higher modulus of regularity for $v$ at $P$. In particular, by the compactness argument in the Appendix, c.f. Section \ref{sec:quarter_Hoelder}, this allows us to conclude that the Legendre function $v$ is in $X_{\delta,\epsilon}(\mathcal{B}_{1/2}^+)$ for all $\delta \in (0,1-\frac{n+1}{p}]$ if $p<\infty$ and in $X_{\delta,\epsilon}(\mathcal{B}_{1/2}^+)$ for all $\delta \in (0,1-\frac{n+1}{p})$ if $p= \infty$. This shows the desired regularity of $v$ and hence of $\Gamma_u$.
\end{proof}

\subsubsection{Regularity results in the presence of inhomogeneities and obstacles}
\label{sec:nonzero}
In this section we consider the regularity of the free boundary in the presence of non-vanishing inhomogeneities $f$. In particular, this includes the presence of obstacles (c.f. Remark \ref{rmk:obstacles_1}).

In this set-up we show the following results:

\begin{prop}[Inhomogeneities]
\label{prop:inhomo_2}
Let $w$ be a solution of the thin obstacle problem with metric $a^{ij}$ satisfying the assumptions (A1)-(A7) from Section \ref{sec:conventions}. 
\begin{itemize}
\item[(i)] Assume further that $a^{ij}\in W^{1,p}(B_1^+, \R^{(n+1)\times(n+1)}_{sym})$ and $f\in L^p$ for some $p\in (2(n+1),\infty]$. Then $\Gamma_w$ is locally a $C^{1,\frac{1}{2}-\frac{n+1}{p}}$ graph.
\item[(ii)] Assume further that $a^{ij}\in C^{k,\gamma}(B_1^+, \R^{(n+1)\times(n+1)}_{sym})$ and that $f\in C^{k-1,\gamma}$ with $k\geq 1$, $\gamma \in (0,1)$. Then we have that $\Gamma_w$ is locally a $C^{k+[\frac{1}{2}+\gamma], (\frac{1}{2}+\gamma - [\frac{1}{2}+\gamma])}$ graph.
\end{itemize}
\end{prop}

We point out that compared with the result without inhomogeneities we lose half a derivative. This is due to the worse decay of the inhomogeneity in (\ref{eq:w1p_g}).\\

Similarly as for the zero obstacle case the proofs for Proposition \ref{prop:inhomo_2} rely on the Hodograph-Legendre transformation. In case (i) of Proposition~\ref{prop:inhomo_2}, we consider the Legendre transformation with respect to the modified solution $u$ after applying the splitting method. This is similar as in Section~\ref{subsec:w1p_zero}, where we dealt with $W^{1,p}$ metrics with zero right hand side. In case (ii) of Proposition~\ref{prop:inhomo_2}, we consider the Legendre transformation with respect to the original solution $w$. We remark that the presence of the inhomogeneity changes neither the leading order asymptotic expansion of $\nabla w$ around the free boundary, nor of the second derivatives $D^2w$ in the corresponding non-tangential cones (assuming $\|f\|_{L^\infty}$ is sufficiently small, which can always be achieved by scaling). In particular, in this case the Hodograph-Legendre transformation is well defined, and the asymptotic expansion of the Legendre function (c.f. Section~\ref{sec:Leg}) remains true.

\begin{proof}
We prove the result of Proposition \ref{prop:inhomo_2} in three steps. First we consider the set-up of (i). Then we divide the setting of (ii) into the cases $k=1$ and $k\geq 2$.
\begin{itemize}
\item In the case of $W^{1,p}$ metrics and $W^{2,p}$ obstacles, we proceed similarly as in Section~\ref{subsec:w1p_zero} by using the splitting method from above. The only changes occur when we estimate the inhomogeneity $g(y)$, where $g(y)$ is as in (\ref{eq:w1p}). Indeed, in the case of $f\neq 0$ we can in general only use the decay estimate \eqref{eq:w1p_g} for $g(y)$. This yields 
\begin{align*}
|g(y)|\leq Cd_G(y,y_0)^{\eta_0},\quad \eta_0=\min\left\{1+4\alpha, 2-\frac{2(n+1)}{p}\right\}.
\end{align*}
Thus, we obtain that $v\in X_{\delta,\epsilon}(\mathcal{B}_{1/2}^+)$ for all $\delta \in (0,\frac{1}{2}-\frac{n+1}{p}]$. In particular, this entails that  $\p_{in}v\in C^{0,\frac{1}{2}-\frac{n+1}{p}}(P\cap \mathcal{B}_{\frac{1}{2}})$. Hence, the regular free boundary $\Gamma_{\frac{3}{2}}(w)$ is locally a $C^{1,\frac{1}{2}-\frac{n+1}{p}}$ submanifold. 
\item In the case of a $C^{1,\gamma}$ metric $a^{ij}$ and a $C^{0,\gamma}$ inhomogeneity $f$, we carry out an analogous expansion as in Proposition \ref{prop:error_gain} and estimate the right hand side of the equation by $d_G(y,y_0)^{2}$. Hence, an application of the bootstrap argument from Proposition \ref{prop:error_gain2} implies that $v\in X_{\delta,\epsilon}(\mathcal{B}_{1/2}^+)$ for all $\delta \in (0,\frac{1}{2}]$. Combining this with the application of the implicit function theorem as in Section \ref{sec:IFT1} hence yields that $\p_{in}v \in C^{[1/2+\gamma], (1/2+\gamma - [1/2+\gamma])}$. This implies the desired regularity.
\item
In the case of $C^{k,\gamma}$, $k\geq 2$ metrics we first apply the implicit function theorem (note that in our set-up the functional $\tilde{F}_a(w_a,y)=F_a(w_a+v,y)-g_a(y)$ is still $C^{k-1,\gamma-\epsilon}$ regular in the parameter $a$). In contrast to the argument in Section \ref{sec:IFT1} we can however now only apply the implicit function theorem in the spaces $X_{\delta,\epsilon}$ with $\delta \in (0,1/2]$. Thus, by the implicit function theorem argument (Step 1 in Theorem~\ref{prop:hoelder_reg_a} in Section \ref{sec:IFT1}) we infer that $\p_{in}v\in C^{k+[1/2+\gamma], (1/2+\gamma - [1/2+\gamma])}$.
\end{itemize}
This concludes the proof of Proposition \ref{prop:inhomo_2}.
\end{proof}

Finally, we comment on the relation of our regularity results with inhomogeneities and the presence of non-zero obstacles.

\begin{rmk}
\label{rmk:obstacles_1}
We note that the set-up of the present Section \ref{sec:W1p} (c.f. \eqref{eq:inhom}) in particular includes the set-up on non-zero obstacles: Indeed, let $a^{ij}:B_1^+ \rightarrow \R^{(n+1)\times (n+1)}_{sym}$ be a uniformly elliptic $W^{1,p}$, $p\in (2(n+1),\infty]$, metric satisfying (A1)-(A3) and let $\phi:B'_1 \rightarrow \R$ be a $W^{2,p}$ function. Suppose that $\tilde{w}$ is a solution to the thin obstacle problem with metric $a^{ij}$ and obstacle $\phi$. Then $w:=\tilde{w}-\phi$ is a solution of the thin obstacle problem 
\begin{align*}
\p_{i} a^{ij} \p_j w & = f \mbox{ in } B_1^+,\\
\p_{n+1}w \leq 0, \ w \geq 0, \ w\p_{n+1}w&=0 \mbox{ in } B_1'.
\end{align*}
Hence, the inhomogeneity now reads $f=-\p_ia^{ij}\p_j\phi$ and is in $L^p$. In particular, Proposition \ref{prop:inhomo_2} is applicable and yields the $C^{1,\frac{1}{2}-\frac{n+1}{p}}$ regularity of the free boundary. Analogous reductions hold for more regular metrics and non-vanishing obstacles.
\end{rmk}

\section{Appendix}
\label{sec:append}

Last but not least, we provide proofs of the estimates which we used in the application of the implicit function theorem. This in particular concerns the spaces $X_{\delta,\epsilon}, Y_{\delta,\epsilon}$ and the mapping properties of $\D_G$ in these: After giving the proof of the characterization of the spaces $X_{\delta,\epsilon}$, $Y_{\delta,\epsilon}$ in terms of decompositions into Hölder functions (c.f. Proposition \ref{prop:decompI}) in Section \ref{sec:decomp}, we present the proof of the (local) $X_{\delta,\epsilon}$ estimates for solutions of the Baouendi-Grushin operator with mixed homogeneous Dirichlet-Neumann data (c.f. Proposition \ref{prop:invert}) in Section \ref{sec:quarter_Hoelder}. Here we argue by an iterative approximation argument, which exploits the scaling properties of the Baouendi-Grushin operator similarly as in \cite{Wa03}. Finally in Sections \ref{sec:XY} and \ref{sec:kernel}, we use this to show the necessary mapping properties of $\D_G$ in the spaces $X_{\delta,\epsilon},Y_{\delta,\epsilon}$.

\subsection{Proof of Proposition \ref{prop:decompI}}
\label{sec:decomp}

In this section we present the proof of the characterization of the spaces $X_{\alpha,\epsilon}$, $Y_{\alpha,\epsilon}$ in terms of decompositions into Hölder functions (c.f. Proposition \ref{prop:decompI} in Section \ref{sec:functions}). 

\begin{proof}
We argue in two steps and first discuss the decomposition of functions in $Y_{\alpha,\epsilon}$ and then the corresponding property of functions in $X_{\alpha,\epsilon}$:\\
(i) Given $f\in Y_{\alpha,\epsilon}$, we denote 
$$f_0(y''):=\p_nf(y''), \quad f_1(y):=r(y)^{-(1+2\alpha-\epsilon)}(f(y)-f_0(y'')y_n).$$ 
In particular, this yields
$f(y)=f_0(y'')y_n+r^{1+2\alpha-\epsilon}f_1(y)$. Moreover, we note that $f_1$ is well-defined on $P$, where it vanishes as a consequence of the boundedness of the $Y_{\alpha,\epsilon}$ norm and of Remark \ref{rmk:homo}. Hence, it suffices to prove the Hölder regularity of $f_0$ and $f_1$. \\
To show that $f_0 \in C^{0,\alpha}(P)$ (in the classical sense), we consider points $y_0,y_1\in P$, $y_0\neq y_1$ and a point $y=(y'',y_n,y_{n+1})\notin P$ with the property that $r(y)=|y_0-y_1|^{1/2}$ and $y_0,y_1\in \mathcal{B}_{2r(y)}(y)$. Then, by the boundedness of the norm and by recalling the estimates in Remark \ref{rmk:homo} (i), we have
\begin{align*}
|f(y)-f_0(y_0)y_n|&\leq C r^{1+2\alpha},\quad |f(y)-f_0(y_1)y_n|\leq C r^{1+2\alpha}.
\end{align*}
Thus, by the triangle inequality,
$$|f_0(y_0)y_n-f_0(y_1)y_n|\leq Cr^{1+2\alpha}.$$
Choosing $y$ with $y_{n+1}=0$, $|y_n|=r(y)>0$ and dividing by $|y_{n}|$ yields
\begin{align*}
|f_0(y_0)-f_0(y_1)|\leq Cr^{2\alpha}=|y_0-y_1|^\alpha.
\end{align*}
This shows the $C^{0,\alpha}$ regularity of $f_0$ (if $\alpha \in (0,1]$).\\
We proceed with the $C^{0,\epsilon}_\ast(Q_+)$ regularity of $f_1$. First we observe that since $|f(y)-f_0(y'')y_n|\leq C r(y)^{1+2\alpha}$ (which follows from Remark \ref{rmk:homo} (i)), we immediately infer that $|f_1(y)|\leq Cr(y)^{\epsilon}$. Thus, if $y_1, y_2 \in Q_+$ are such that 
$d_G(y_1,y_2)\geq \frac{1}{10}\max\{r(y_1), r(y_2) \}$, we have
\begin{align*}
|f_1(y_1)-f_1(y_2)| \leq C r(y_1)^{\epsilon}+Cr(y_2)^\epsilon\leq C d_G(y_1,y_2)^{\epsilon}.
\end{align*}
If $y_1,y_2\in Q_+$ are such that $d_G(y_1,y_2)<\frac{1}{10}\max\{r(y_1),r(y_2)\}$, then there is a point $\bar y\in P$ such that $y_1,y_2\in \mathcal{B}_1^+(\bar y)$ (for example assuming $r(y_1)\geq r(y_2)$ we can let $\bar y=(y''_1,0,0)$). Then  the H\"older regularity follows from the $C^{0,\epsilon}_\ast(\mathcal{B}_1^+(\bar y))$ regularity of $d(\cdot, \bar y)^{-(1+2\alpha-\epsilon)}(f-P_{\bar y})$ and the $C^{0,\alpha}(P)$ regularity of $f_0$. 
More precisely,
\begin{equation}
\label{eq:f1}
\begin{split}
&\quad \left|f_1(y_1)-f_1(y_2)\right|\\
&=\left|r(y_1)^{-1-2\alpha+\epsilon}\left(f(y_1)-f_0(y''_1)(y_1)_n\right)-r(y_2)^{-1-2\alpha+\epsilon}\left(f(y_2)-f_0(y_2'')(y_2)_n\right)\right|\\
&\leq r(y_1)^{-1-2\alpha+\epsilon}\left|\left(f(y_1)-f_0(y''_1)(y_1)_n\right)-\left(f(y_2)-f_0(y''_1)(y_2)_n\right)\right|\\
& \quad + r(y_1)^{-1-2\alpha+\epsilon}\left|f_0(y''_1)(y_2)_n-f_0(y_2'')(y_2)_n\right|\\
&\quad + |r(y_1)^{-1-2\alpha+\epsilon}-r(y_2)^{-1-2\alpha+\epsilon}||f(y_2)-f_0(y_2'')(y_2)_n|.
\end{split}
\end{equation}
By the definition of the norm of $Y_{\alpha,\epsilon}$, we have
\begin{align*}
&\left|\left(f(y_1)-f_0(y''_1)(y_1)_n\right)-\left(f(y_2)-f_0(y''_1)(y_2)_n\right)\right|\\
&=\left|(f(y_1)-P_{y''_1}(y_1))-(f(y_2)-P_{y''_1}(y_2))\right|\lesssim r(y_1)^{1+2\alpha-\epsilon}d_G(y_1,y_2)^{\epsilon}.
\end{align*}
Moreover, the $C^{0,\alpha}$ regularity of $f_0$ as well as $|(y_2)_n|\sim r$ yields
\begin{align*}
r(y_1)^{-1-2\alpha+\epsilon}\left|f_0(y''_1)(y_2)_n-f_0(y_2'')(y_2)_n\right|&\lesssim r(y_1)^{-1-2\alpha+\epsilon}r(y_1)d_G(y_1,y_2)^{2\alpha}\\
&  \lesssim d_G(y_1,y_2)^{\epsilon}.
\end{align*}
Here we have used that $2\alpha \geq \epsilon$ and that w.l.o.g. $0\leq r(y_2)\leq r(y_1)$.
Finally, the last term in (\ref{eq:f1}) is estimated by the $C^{0,\epsilon}_\ast$ regularity of $r(y_1)^{\epsilon}$ and by recalling the definition of the norm on $Y_{\alpha,\epsilon}$ in combination with Remark \ref{rmk:homo} once more.  Combining all the previous observations, we have
\begin{align*}
|f_1(y_1)-f_1(y_2)|\lesssim  d_G(y_1,y_2)^{\epsilon}.
\end{align*}

This completes the proof of (i).\\

(ii) The proof for the decomposition of $v$ is similar. Given any $y\in Q_+\setminus P$, we denote by $y_0:=(y'',0,0)\in P$ the projection of $y$ onto $P$. Since $v$ is $C^{3,\alpha}_\ast$ at $y_0$, there exists a Taylor polynomial 
\begin{align*}
P_{y_0}(z)=\p_n v(y_0)z_n+\p_{in}v(y_0)(z_i-y_i)z_n + \frac{1}{6}\p_{nnn}v(y_0)z_n^3+\frac{1}{2}\p_{n,n+1,n+1}(y_0)z_nz_{n+1}^2,
\end{align*}
such that $|v(z)-P_{y_0}(z)|\leq Cd_G(z,y_0)^{3+2\alpha}$ for each $z\in \mathcal{B}_{1}^+(y_0)$. Due to the regularity of $v$ at $P$, the coefficients have the desired regularity properties: $\p_{n}v(y'')\in C^{1,\alpha}(P\cap \mathcal{B}_{1/2})$, $\p_{in}v(y''), \p_{nnn}v(y''), \p_{n,n+1,n+1}v(y'')\in C^{0,\alpha}(P\cap \mathcal{B}_{1/2})$. Moreover, their Hölder semi-norms are bounded from above by $C\|v\|_{X_{\alpha,\epsilon}}$. \\

In order to show the $C^{0,\epsilon}_\ast$ estimates of $C_1$, $V_i$ and $C_{ij}$, we argue similarly as in (i) for $f\in Y_{\alpha,\epsilon}$. For simplicity we only present the argument for
$$V_n(y):=r^{-(2+2\alpha-\epsilon)}\p_n(v-P_{y''})(y), \quad y\in \mathcal{B}_1^+\setminus P.$$ 
The others are analogous.
First, the boundedness of the first two terms in the norm $\|v\|_{X_{\alpha,\epsilon}}$ (c.f. Definition~\ref{defi:spaces}) and an interpolation estimate imply that for each $y\in \mathcal{B}_1^+\setminus P$ fixed, $d_G(z,y'')^{-(2+2\alpha-\epsilon)}\p_n(v-P_{y''})(z)$ as a function of $z$ is in $C^{0,\epsilon}_\ast(\mathcal{B}_{r(y)/2}(y))$, with norm bounded by $C\|v\|_{X_{\alpha,\epsilon}}$.  
Next, for any points $z_1$ and $z_2$ in the non-tangential ball $\mathcal{B}_{r(y)/2}(y)$ we have
\begin{align*}
&\quad |V_n (z_1)-V_n(z_2)|\\
&=\left|d(z_1,z''_1)^{-(2+2\alpha-\epsilon)}\p_{n}(v-P_{z''_1})(z_1)-d(z_2,z''_2)^{-(2+2\alpha-\epsilon)}\p_{n}(v-P_{z''_2})(z_2)\right|\\
&\leq \left|d(z_1,z''_1)^{-(2+2\alpha-\epsilon)}\p_{n}(v-P_{z''_1})(z_1)-d(z_2,z''_1)^{-(2+2\alpha-\epsilon)}\p_{n}(v-P_{z''_1})(z_2)\right|\\
& \quad + \left|\left(d(z_2,z''_1)^{-(2+2\alpha-\epsilon)}-d(z_2,z''_2)^{-(2+2\alpha-\epsilon)}\right)\p_{n}(v-P_{z''_1})(z_2)\right|\\
& \quad +\left|d(z_2,z''_2)^{-(2+2\alpha-\epsilon)}\p_{n}(P_{z''_2}-P_{z''_1})(z_2)\right|:=I+II+III.
\end{align*}
By the definition of the $X_{\alpha,\epsilon}$-norm and by interpolation, $I\leq C\|v\|_{X_{\alpha,\epsilon}}d_G(z_1, z_2)^{\epsilon}$. Using the fact that $|\p_{n}(v-P_{z''_1})(z_2)|\leq C\|v\|_{X_{\alpha,\epsilon}} d_G(z_2,z''_1)^{2+2\alpha}$ and that $d_G(z''_2,z''_1)\leq Cd_G(z_2,z_1)\leq C\min \{d_G(z_2,z''_2),d_G(z_1,z''_1)\}$ for $z_1,z_2$ in the non-tangential ball $\mathcal{B}_{r(y)/2}(y)$, we also have that $II\leq C\|v\|_{X_{\alpha,\epsilon}}d_G(z_2,z_1)^{\epsilon}$. To estimate $III$ we notice that
\begin{align*}
\p_n P_{z''_1}(z)&=\p_nv(z''_1)+\sum_{i=1}^{n-1}\p_{in}v(z''_1)(z_i-(z_1)_i)\\ 
& \quad +\frac{1}{2}\p_{nnn}v(z''_1)z_n^2+\frac{1}{2}\p_{n,n+1,n+1}v(z''_1)z_{n+1}^2.
\end{align*}
Recalling that $\p_{n}v(y'')\in C^{1,\alpha}(P)$ and using a Taylor expansion of $\p_nv(y'')$ at $z''_1$, we infer that
\begin{align*}
\left|\p_nv(z''_2)-\left(\p_nv(z''_1)+\sum_{i=1}^{n-1}\p_{in}v(z''_1)((z_2)_i-(z_1)_i)\right)\right|\leq C\|v\|_{X_{\alpha,\epsilon}}|z''_2-z''_1|^{1+\alpha}.
\end{align*}
Thus, recalling the $C^{0,\alpha}$ regularity of $\p_{nnn}v(z'')$ and $\p_{n,n+1,n+1}v(z'')$, we obtain
\begin{align*}
\p_n( P_{z''_1}-P_{z''_2}) (z_2)&= \left(\p_n v(z''_1)+\sum_{i=1}^{n-1}\p_{in}v(z''_1)( (z_2)_i-(z_1)_i) - \p_n v(z''_2)\right)\\
&+\frac{1}{2}\left(\p_{nnn}v(z''_1)-\p_{nnn}v(z''_2)\right)(z_2)_n^2 \\
&+\frac{1}{2}\left(\p_{n,n+1,n+1}v(z''_1)-\p_{n,n+1,n+1}v(z''_2)\right)(z_2)_{n+1}^2\\
&\leq C\|v\|_{X_{\alpha,\epsilon}}d_G(z''_2,z''_1)^{2(1+\alpha)}+ C\|v\|_{X_{\alpha,\epsilon}}d_G(z''_1,z''_2)^{2\alpha}d_G(z_2,z''_2)^2.
\end{align*}
Due to the same reason as for $f$, this implies the estimate for $III$.
\end{proof}

\subsection{Proof of Proposition~\ref{prop:invert}}
\label{sec:quarter_Hoelder}

In this section, we present the proof of the (local) $X_{\delta,\epsilon}$ estimates for the Baouendi-Grushin operator (c.f. Proposition \ref{prop:invert} in Section \ref{sec:functions}).
We begin by recalling the natural energy spaces associated with the Baouendi-Grushin operator:

\begin{defi}
\label{defi:GrushinLp}
Let $\Omega\subset \R^{n+1}$ be an open subset.
The Baouendi-Grushin operator is naturally associated with the following Sobolev spaces (recall Definition~\ref{defi:Hoelder1} for the vector fields $\tilde{Y}_j$): 
\begin{align*}
M^{1}(\Omega)&:= \{u\in L^2(\Omega)| \tilde{Y}_ju \in L^2(\Omega) \mbox{ for } j\in \{1,\dots,2n\}\},\\
M^{2}(\Omega)&:= \{u\in L^2(\Omega)|\tilde{Y}_ju, \tilde{Y}_{k}\tilde{Y}_\ell u\in L^2(\Omega)\text{ for } j,k,\ell\in \{1,\ldots, 2n\}\}.
\end{align*}
\end{defi}

We prove Proposition~\ref{prop:invert} in two steps. Firstly, we obtain a polynomial approximation (in the spirit of Campanato spaces) near the points at which the ellipticity of the operator degenerates, $P:=\{y_n=y_{n+1}=0\}$ (c.f. Proposition \ref{prop:Hoelder0}). Then we interpolate these estimates with the uniformly elliptic estimates which hold away from the degenerate points. Here we follow a compactness argument which was first outlined in this form by Wang, \cite{Wa03}. It proceeds via approximation and iteration steps.\\

In the sequel, we deduce a first regularity estimate in the energy space. This serves as a compactness result for the following approximation lemmata:

\begin{prop}
\label{prop:Sobolevreg}
Let $0<r\leq R<\infty$. Let $f:\mathcal{B}_{R}^+(0) \rightarrow \R$ be an $L^{2}$ function and let $u:\mathcal{B}_{R}^+\rightarrow \R$ be a solution of 
\begin{equation}
\label{eq:Grushin}
\begin{split}
\Delta_G u& =f \text{ in } \mathcal{B}_R^+,\\
u&=0 \text{ on } \mathcal{B}_R^+\cap \{y_n=0\},\\
\p_{n+1}u&=0 \text{ on } \mathcal{B}_R^+\cap\{y_{n+1}=0\}.
\end{split}
\end{equation}
Then 
\begin{align}
\label{eq:Lpapriori}
\left\| u \right\|_{M^{2}(\mathcal{B}_r^+)} \leq C(n,r,R)\left( \| f\|_{L^{2}(\mathcal{B}_{R}^+)} + \| u\|_{L^{\infty}(\mathcal{B}_{R}^+)} \right).
\end{align}
\end{prop}

\begin{proof}
The result is obtained by an even and odd reflection from the whole space result (in particular by the kernel estimate, see e.g. Lemma \ref{lem:ker} in Section \ref{sec:kernel}).
\end{proof}

With Proposition \ref{prop:Sobolevreg} at hand, we prove our first approximation result: We approximate solutions of the \emph{inhomogeneous} Baouendi-Grushin equation by solutions of the \emph{homogeneous} equation, provided the inhomogeneity is sufficiently small.

\begin{lem}
\label{lem:compactness}
Assume that $u:\mathcal{B}_{1}^+ \rightarrow \R$ is a solution of (\ref{eq:Grushin}) which satisfies
\begin{align*}
\frac{1}{|\mathcal{B}_{1}^+(0)|}\int\limits_{\mathcal{B}_1^+(0)} u^2 dx \leq 1.
\end{align*}
For any $\epsilon>0$ there exists a constant $\delta=\delta(\epsilon)>0$ such that if 
\begin{align*}
\frac{1}{|\mathcal{B}_{1}^+(0)|} \int\limits_{\mathcal{B}_1^{+}(0)}f^2 dx \leq \delta^2,
\end{align*} 
then there is a solution $h$ of the homogeneous Baouendi-Grushin equation with mixed Dirichlet-Neumann data, i.e. 
\begin{equation}
\label{eq:hGrushin}
\begin{split}
\D_G h & = 0 \mbox{ on } \mathcal{B}_R^+(0),\\
h&=0 \mbox{ on } \{y_{n}=0\}\cap \mathcal{B}_R^+(0),\\
\p_{n+1} h&=0 \mbox{ on } \{y_{n+1}=0\}\cap \mathcal{B}_R^+(0),
\end{split}
\end{equation}
such that
\begin{align*}
\frac{1}{|\mathcal{B}_{1/2}^+(0)|} \int\limits_{\mathcal{B}_{1/2}^+(0)}|u-h|^2 dx \leq \epsilon^2.
\end{align*}
\end{lem}

\begin{proof}
We argue by contradiction and compactness. Assume that the statement were wrong. Then there existed $\epsilon>0$ and sequences, $\{u_m\}_{m}$, $\{f_m\}_m$, such that on the one hand
\begin{equation}
\label{eq:contra}
\begin{split}
\D_G u_m &= f_m \mbox{ in } \mathcal{B}_1^+(0),\\
u_m & = 0 \mbox{ on } \{y_n = 0\}\cap \mathcal{B}_{1}^+(0),\\
\p_{n+1} u_m &= 0 \mbox{ on } \{y_{n+1}=0\} \cap \mathcal{B}_{1}^+(0).
\end{split}
\end{equation}
and
\begin{align*}
&\frac{1}{|\mathcal{B}_{1}^+(0)|}\int\limits_{\mathcal{B}_1^+(0)} u_m^2 dx \leq 1, \ \frac{1}{|\mathcal{B}_{1}^+(0)|} \int\limits_{\mathcal{B}_1^{+}(0)}f_m^2 dx \leq \frac{1}{m}.
\end{align*}
On the other hand
\begin{align*}
\frac{1}{|\mathcal{B}_{1}^+(0)|}\int\limits_{\mathcal{B}_1^+(0)} |u_m - h|^2 dx \geq \epsilon^2, 
\end{align*}
for all $h$ which satisfy the homogeneous equation (\ref{eq:hGrushin}). By (\ref{eq:Lpapriori}), we however have compactness for $u_m$ in $M^{1}$:
\begin{align*}
u_m \rightarrow u_0 \mbox{ in } M^{1}(\mathcal{B}_{3/4}^+(0)).
\end{align*}
Testing the weak form of (\ref{eq:contra}) with a $C_{0}^{\infty}(\mathcal{B}_{1/2}^+(0))$ function, we can pass to the limit and infer that $u_0$ is a weak solution of the homogeneous bulk equation from (\ref{eq:hGrushin}).
Finally, by the boundedness of $u_m\in M^{2}$ and the corresponding trace inequalities or a reflection argument, we obtain that $u_0$ satisfies the mixed Dirichlet-Neumann conditions from (\ref{eq:hGrushin}). This yields the desired contradiction.
\end{proof}

We now prove a further approximation result for solutions of the homogeneous Baouendi-Grushin equation in the quarter space. More precisely, we now seek to approximate solutions of the homogeneous equation (\ref{eq:contra}) by associated (eigen-) polynomials.
To this end, we recall the notion of \emph{homogeneous polynomials} in Section~\ref{sec:holder}.

\begin{rmk}
We note that all homogeneous polynomial solutions (e.g. the ones up to degree five) of $\Delta_Gv=0$ which satisfy the Dirichlet-Neumann boundary conditions can be computed explicitly. For instance, the degree less than five polynomial solutions are given by the linear combination of
\begin{align*}
y_n, \ y_jy_n, \ j\in\{1,\dots,n-1\},\ y_{n}^3 - 3y_n y_{n+1}^2.
\end{align*}
\end{rmk}

Using the notion of homogeneous polynomials, we proceed to our second approximation lemma:

\begin{lem}
\label{lem:approx}
Let $u:\mathcal{B}_{1}^+(0) \rightarrow \R$ be a solution of (\ref{eq:hGrushin}) with $\| u \|_{L^2(\mathcal{B}_{1}^+(0))}\leq \bar{c}$. Then there exists a polynomial $p$ of homogeneous degree less than or equal to three which solves (\ref{eq:hGrushin}), i.e.
\begin{align*}
p(y)=y_n \left(a_0+\sum_{i=1}^{n-1}a_iy_i+b(y_n^2-3y_{n+1}^2)\right),
\end{align*}
such that for all $0<r\leq \frac{1}{2}$
\begin{align}
\label{eq:approx}
\frac{1}{|\mathcal{B}_{r}^+(0)|}\int\limits_{\mathcal{B}_r^+(0)}|u-p|^2 dy \leq C(\bar{c}) r^{10},
\end{align}
and
\begin{align*}
\sum_{i=0}^{n-1}|a_i|+|b|\leq C\bar c,
\end{align*}
where $C$ is a universal constant.
\end{lem}

\begin{proof}
After a conformal change of variables, the Baouendi-Grushin operator can be rewritten as
\begin{align*}
\D_G = (\dt^2 + \D_{\Sigma}),
\end{align*}
where $\Sigma:= \{(y'',y_n,y_{n+1})| \ |y''|^4 + y_n^2 + y_{n+1}^2 = 1\}$ denotes the Baouendi-Grushin sphere. In our setting this is augmented with the Dirichlet and Neumann conditions from (\ref{eq:hGrushin}). The eigenfunctions of $\D_{\Sigma}$ can be extended in the radial direction to yield homogeneous solutions of the homogeneous Baouendi-Grushin equation. As the Baouendi-Grushin operator is hypoelliptic, these solutions are polynomials (this remains true in the cone $\Sigma \cap \mathcal{B}_1^+$, as the eigenfunctions on the Baouendi-Grushin quarter sphere can be identified as a subset of the eigenfunctions on the whole sphere by appropriate (even and odd) reflections). Moreover, the eigenfunctions on $\Sigma$ are orthogonal and as a consequence, the same is true for the correspondingly associated polynomials. The Baouendi-Grushin polynomials hence form an orthogonal basis into which a solution of the homogeneous Baouendi-Grushin problem can be decomposed. 
We denote these polynomials by $p_k(y)$ and normalize them with respect to $\mathcal{B}^{+}_{1}(0)$. Since $u$ is bounded in $L^2(\mathcal{B}_{1}^+(0))$, we have
\begin{align}
\label{eq:poly_decomp}
u(y) = \sum\limits_{k=0}^{\infty} \alpha_k p_k(y) \mbox{ with } \sum\limits_{k=0}^{\infty}|\alpha_k|^2 \leq \bar{c}^2.
\end{align}
The previous decomposition can also be seen ``by hand'': Making the ansatz that a homogeneous solution of the Baouendi-Grushin problem is of the form
\begin{align*}
u(t,\theta) = \sum\limits_{k\in \Z} \alpha_k(t) u_{k}(\theta),
\end{align*}
where $u_k(\theta)$ denotes the spherical eigenfunctions, we obtain that
\begin{align*}
0& =(u_k, \D_G u)_{L^{2}(\Sigma)} = \alpha_k''(t) + (u_k, \D_{\Sigma} u)_{L^2(\Sigma)} \\
& =  \alpha_k''(t) -\lambda_k^2 \alpha_k(t)
+ \int\limits_{\partial \Sigma} u_k (\nu\cdot \nabla_{\Sigma} u) d\mathcal{H}^{n-1} -  \int\limits_{\partial \Sigma} u (\nu\cdot \nabla_{\Sigma} u_k) d\mathcal{H}^{n-1}.
\end{align*}
Here $\lambda_k^2$ is the eigenvalue associated with $u_k(\theta)$ and $\nu:\partial \Sigma \rightarrow \R^{n-1}$ is the outer unit normal field.
As both $u$ and $u_k$ satisfy the mixed Dirichlet-Neumann boundary conditions, this yields that
\begin{align*}
\alpha''_k(t) - \lambda_k^2 \alpha_k(t)=0.
\end{align*}
As the Dirichlet data imply that $\alpha_k(-\infty)=0$, this results in $\alpha_k(t) = \alpha_k(0) e^{|\lambda_k| t} $. By hypoellipticity, $\lambda_k \in \Z$, so that we obtain a decomposition into polynomials, after undoing the conformal change of coordinates.
After an appropriate normalization, we again infer (\ref{eq:poly_decomp}).\\

We define $p(y):= \sum\limits_{k=0}^{3} \alpha_k p_k(y)$.
Thus, recalling that there are no eigenpolynomials of (homogeneous) degree four which satisfy our mixed Dirichlet-Neumann conditions and computing the difference $u-p$, we arrive at
\begin{align*}
\left\| u- p \right\|_{L^2(\mathcal{B}_r^+)}^2 =  \sum\limits_{k=5}^{\infty}|a_k|^2 \| p_k\|_{L^2(\mathcal{B}_r^+)}^2
&\leq  \sum\limits_{k=5}^{\infty}|a_k|^2 r^{10}|\mathcal{B}_r^+|\| p_k\|_{L^2(\mathcal{B}_1^+)}^2\\
&\leq r^{10+ 2n}C(\bar{c}),
\end{align*}
where we used the scaling of the Baouendi-Grushin cylinders from Definition \ref{defi:Grushincylinder} and the boundedness of $u$ (c.f. (\ref{eq:poly_decomp})).
This yields the desired result.
\end{proof}

\begin{rmk}
\label{rmk:approx}
We stress that the approximation from Lemma \ref{lem:approx} is not restricted to third order polynomials. It can be extended to polynomials of arbitrary (homogeneous) degree.
\end{rmk}

Combining the previous results, we obtain the key building block for the iteration which yields regularity at the hyperplane $\{y_n=y_{n+1}=0\}$ at which the ellipticity of the Baouendi-Grushin operator degenerates.

\begin{lem}[Iteration]
\label{prop:iteration}
Let $\alpha\in(0,1)$.
Assume that $u:\mathcal{B}_{1}^+ \rightarrow \R$ is a solution of (\ref{eq:Grushin}) which satisfies
\begin{align*}
\frac{1}{|\mathcal{B}_{1}^+(0)|}\int\limits_{\mathcal{B}_1^+(0)} u^2 dx \leq 1.
\end{align*}
There exist a radius $r_0\in (0,1)$, a universal constant $C>0$ and a constant $\epsilon>0$ such that if
\begin{align*}
\frac{1}{|\mathcal{B}_{1}^+(0)|} \int\limits_{\mathcal{B}_1^{+}(0)}f^2 dx \leq \epsilon^2,
\end{align*} 
then there exists a polynomial $p$ of order less than or equal to three
satisfying (\ref{eq:hGrushin}), i.e.
\begin{align*}
p(y)=y_n \left(a_0+\sum_{i=1}^{n-1}a_iy_i+b(y_n^2-3y_{n+1}^2)\right),
\end{align*}
such that
\begin{align*}
\frac{1}{|\mathcal{B}_{r_0}^+(0)|}\int\limits_{\mathcal{B}_{r_0}^+(0)} |u-p|^2 dy \leq r^{2(3+2\alpha)}_0,
\end{align*}
and
\begin{align*}
\sum_{i=0}^{n-1}|a_i|+|b|\leq C.
\end{align*}
\end{lem}

\begin{proof}
By our first approximation result, Lemma \ref{lem:compactness}, there exists a function $h$ which solves (\ref{eq:hGrushin}) and satisfies
\begin{align}
\label{eq:approx1}
\int\limits_{\mathcal{B}_{1/2}^+(0)}|u-h|^2 dy \leq \delta^2.
\end{align}
In particular, $\| h \|_{L^2(\mathcal{B}_{1/2}^+(0))}\leq C$. Hence, by our second approximation result, Lemma \ref{lem:approx}, there exists a (homogeneous) third order Baouendi-Grushin polynomial satisfying the Dirichlet-Neumann condition such that
\begin{align*}
\frac{1}{|\mathcal{B}_{r}^+(0)|} \int\limits_{\mathcal{B}_{r}^+(0)} |h-p|^2 dy \leq C r^{10}, \quad \mbox{ for all } 0<r<1/2.
\end{align*}
Consequently, by rescaling, we obtain for each $0<r<1/2$,
\begin{align*}
\frac{1}{|\mathcal{B}_r^{+}(0)|} \int\limits_{\mathcal{B}_r^+(0)}|u-p|^2 dy & \leq 2 \frac{1}{|\mathcal{B}_r^{+}(0)|} \int\limits_{\mathcal{B}_r^+(0)}|u-h|^2 dy
+ 2 \frac{1}{|\mathcal{B}_r^{+}(0)|} \int\limits_{\mathcal{B}_r^+(0)}|h-p|^2 dy\\
& \leq 2 r^{-2n} \int\limits_{\mathcal{B}_{1/2}^+(0)}|u-h|^2 dy + 2C r^{10}\\
& \leq 2 r^{-2n}\delta^2 + 2C r^{10}, \
\end{align*}
where we used (\ref{eq:approx1}) to estimate the first term. First choosing $0<r_0<1$ universal, but so small such that $2C r_0^{10}\leq \frac{1}{2}r_0^{2(3+2\alpha)}$, and then choosing $\delta>0$ universal such that $2 r_0^{-2n}\delta^2 \leq\frac{1}{2}r_0^{2(3+2\alpha)}$, yields the desired result.
\end{proof}

As a corollary of Lemma \ref{prop:iteration}, we can iterate in increasingly finer radii.

\begin{cor}
\label{cor:iteration}
Let $\alpha\in(0,1)$.
Assume that $u:\mathcal{B}_{1}^+ \rightarrow \R$ is a solution of (\ref{eq:Grushin}) which satisfies
\begin{align*}
\frac{1}{|\mathcal{B}_{1}^+(0)|}\int\limits_{\mathcal{B}_1^+(0)} u^2 dx \leq 1,
\end{align*}
and that for each $k\in \mathbb{N}_+$
\begin{align*}
\frac{1}{|\mathcal{B}_{r_0^{k-1}}^+(0)|} \int\limits_{\mathcal{B}_{r_0^{k-1}}^{+}(0)}f^2 dx \leq \epsilon^2 r_0^{2(k-1)(1+2\alpha)}.
\end{align*} 
Then there exists a polynomial $p_k$ of (homogeneous) degree (less than or equal to) three solving (\ref{eq:hGrushin}) such that
\begin{align*}
\frac{1}{|\mathcal{B}_{r^{k}_0}^{+}(0)|} \int\limits_{\mathcal{B}_{r^{k}_0}^+(0)}|u-p_k|^2 dx \leq r^{2k (3+ 2\alpha)}_0.
\end{align*}
Moreover, it is of the form
\begin{align*}
p_k(y) = y_n\left(a^0_k + \sum\limits_{j=1}^{n-1}a_k^j y_j\right) + b_k (y_n^3 - 3y_n y_{n+1}^2), 
\end{align*}
and we have
\begin{equation}
\label{eq:coefficients}
\begin{split}
|a^0_k - a^0_{k-1}| &\leq C r_0^{k(2+{2\alpha})},\\
|a_k^j - a_{k-1}^j| &\leq C r_0^{{2k\alpha}},\quad j=1,\dots, n-1,\\
|b_k - b_{k-1}| & \leq C r_0^{{2k\alpha}}.
\end{split}
\end{equation}
\end{cor}

\begin{proof}
We argue by induction on $k$ and take $p_0 = 0$ and $p_1$ as the polynomial from Proposition \ref{prop:iteration}. We assume that the statement is true for $k$ and show it for $k+1$. For that purpose, we consider the rescaled and dilated functions
\begin{align*}
u_k(y):= \frac{(u-p_k)(r_0^{2k} y'', r_0^k y_n, r_0^k y_{n+1})}{r_0^{k(3+{2\alpha})}}.
\end{align*}
Hence,
\begin{align*}
\D_G u_k = \frac{r_0^{2 k}f_{r_0}}{r_0^{(3+{2\alpha})k}} = r_0^{-k (1+{2\alpha})} f_{r_0},
\end{align*}
where $f_{r_0}(y'',y_n,y_{n+1})= f(r_0^{2} y'', r_0 y_n, r_0 y_{n+1})$. 
Using the smallness assumption on $f$, we obtain that
\begin{align*}
\frac{1}{|\mathcal{B}_{1}^{+}(0)|} \int\limits_{\mathcal{B}_{1}^+(0)}|r_0^{-k(1+{2\alpha})} f_{r_0}|^2 dx \leq r_0^{-2k (1+{2\alpha})} \frac{1}{|\mathcal{B}_{r_0^{k}}^{+}(0)|} \int\limits_{\mathcal{B}_{r_0^k}^+(0)}f^2 dx \leq  \epsilon^2.
\end{align*}
Hence by Proposition \ref{prop:iteration}, we obtain a (homogeneous) polynomial, $q$, of degree less than or equal to three, which satisfies (\ref{eq:hGrushin}) and is of the form
\begin{align*}
q(y'',y_n, y_{n+1}) =  y_n\left(a_0+ \sum_{j=1}^{n-1}a_j y_j\right) + b (y_n^3 - 3y_n y_{n+1}^2),
\end{align*}
and 
$$\sum_{j=0}^{n-1}|a_j|+|b|\leq C,$$
such that
\begin{align*}
\frac{1}{|\mathcal{B}_{r_0}^{+}(0)|} \int\limits_{\mathcal{B}_{r_0}^+(0)}|u_k - q|^2 dx  \leq r_0^{2(3+{2\alpha})}.
\end{align*}
Rescaling therefore gives us that the polynomial
\begin{align*}
p_{k+1}(y'',y_n,y_{n+1}) = p_k(y'',y_{n},y_{n+1}) + r_0^{(3+{2\alpha})k}q\left(\frac{y''} {r_0^{2k}}, \frac{y_n}{r_0^k}, \frac{y_{n+1}}{r_0^k} \right),
\end{align*}
satisfies the claim of the corollary.
\end{proof}

We summarize the previous compactness and iteration arguments in the following intermediate result:

\begin{prop}
\label{prop:Hoelder0}
Let $\alpha\in(0,1)$.
Assume that $u:\mathcal{B}_{1}^+ \rightarrow \R$ is a solution of (\ref{eq:Grushin}). Suppose that the inhomogeneity $f:\mathcal{B}_{1}^+(0)\rightarrow \R$ is $C^{1,\alpha}_{\ast}$ at $y=0$ in the sense of Definition \ref{defi:diff}, i.e.
\begin{align*}
|f - f(0) - \p_n f(0) y_n | \leq F_0 r^{1+{2\alpha}}
\end{align*}
for any $0<r<1$. 
Then there is a polynomial $p$ with \\
$\|p\|_{C^{3}(\mathcal{B}_{1}^+)}\leq C \left(\| u\|_{L^{2}(\mathcal{B}_1^+)} + |f(0)|+|\p_nf(0)|\right)$ of (homogeneous) degree less than or equal to three such that
\begin{align*}
\frac{1}{|\mathcal{B}_{r}^{+}(0)|} \int\limits_{\mathcal{B}_{r}^+(0)}|u-p|^2 dx \leq C(\| u\|_{L^2(\mathcal{B}_1^+(0))}^2 + F_0^2)r^{2(3+{2\alpha})}.
\end{align*}
\end{prop}

\begin{proof}
Without loss of generality we may assume that $f(0)=\p_n f(0)=0$. Indeed, this follows by considering the function $v(y'',y_n,y_{n+1}):= u(y) - q(y)$, where $q(y)$ is a homogeneous polynomial of homogeneous degree less than or equal to three such that $\Delta_G q=f(0)+\p_{n}f(0)y_n$ and $q=0$ on $\{y_n=0\}$, $\p_{n+1}q=0$ on $\{y_{n+1}=0\}$ (for example, one can consider $q(y)=\frac{1}{2}f(0) y_n^2+ cy_ny_{n+1}^2 +dy_n^3$ with $2c+6d=\p_nf(0)$).
Considering $\tilde{v}:=\frac{\epsilon}{F_0}v$, then also gives the smallness assumptions of Corollary \ref{cor:iteration}. Thus, for each $k\in \N_+$ there exists a Baouendi-Grushin polynomial $p_k$ such that
\begin{align*}
\frac{1}{|\mathcal{B}_{r_0^k}^+(0)|} \int\limits_{\mathcal{B}_{r_0^k}^+(0)}|\tilde{v}-p_k|^2 dy \leq r_0^{2k(3+{2\alpha})}.
\end{align*}
Due to the estimates (\ref{eq:coefficients}) on the coefficients of $p_k$, which were derived in Corollary \ref{cor:iteration}, $p_k \rightarrow p_{\infty}$, where $p_{\infty}$ is a polynomial of (homogeneous) degree at most three and which satisfies
\begin{align*}
\frac{1}{|\mathcal{B}_{r_0^k}^+(0)|} \int\limits_{\mathcal{B}_{r_0^k}^+(0)}|p_{\infty}-p_k|^2 dy \leq C r_0^{2k(3+{2\alpha})}.
\end{align*}
Consequently, by the triangle inequality
\begin{align*}
\frac{1}{|\mathcal{B}_{r_0^k}^+(0)|} \int\limits_{\mathcal{B}_{r_0^k}^+(0)}|\tilde{v} -p_{\infty}|^2 dy \leq C r_0^{2k(3+{2\alpha})}
\end{align*}
for $k\in \N$. Rescaling then yields the desired result.
\end{proof}

\begin{rmk}
\begin{itemize}
\item The previous result yields the ``H\"older regularity at the point'' $y=0$. For other points $y_0=(y_0'',0,0)$ an analogous result holds by translation invariance of the equation and the boundary conditions in the $y''$ directions (c.f. (\ref{eq:Grushin})). In this translated case, the conditions on the inhomogeneity $f:\mathcal{B}_{1}^+(y_0)\rightarrow \R$ read
\begin{align*}
|f - f(y_0) - \p_{n}f(y_0) y_n|  \leq F_0 r^{2(1+{2\alpha})}.
\end{align*}
\item Instead of imposing the $C^{1,\alpha}_{\ast}$ condition in the sense of Definition \ref{defi:diff}, it would have sufficed to assume the weaker condition 
\begin{align*}
 \frac{1}{|\mathcal{B}_{r}^{+}(y_0)|} \int\limits_{\mathcal{B}_{r}^+(y_0)}|f - f(y_0) - \p_{n}f(y_0) y_n|^2 dy  \leq F_0 r^{2(1+{2\alpha})}.
\end{align*}
\item In order to argue as we have outlined above, we have to require the compatibility condition $\p_{n+1}f(y'',0,0)=0$ (c.f. Definition \ref{defi:spaces}). However, apart from the described $C^{1,\alpha}_{\ast}$ regularity, we do not have to pose further restrictions on $f$.
\end{itemize}
\end{rmk}

Building on the precise description of the regularity of solutions close to the hyperplane $\{y_n=y_{n+1}=0\}$, we can now derive the full regularity result of Proposition \ref{prop:invert} by additionally invoking the uniform ellipticity which holds at a sufficiently far distance from $P$. This then concludes the argument for Proposition \ref{prop:invert}.

\begin{proof}[Proof of Proposition~\ref{prop:invert}]
It suffices to prove the corresponding regularity result in $X_{\alpha,\epsilon}(\mathcal{B}_3^+)$ (c.f. Definition \ref{defi:spaces_loc}).
Indeed, the Hölder estimate,
\begin{align*}
\|v\|_{C^{2,\epsilon}_{\ast}(Q_+)} \lesssim \|\D_G v\|_{C^{0,\epsilon}_{\ast}(Q_+)},
\end{align*}
follows similarly. As a consequence of the support assumption on $\Delta_G v$, this then yields the bound
\begin{align*}
\|v\|_{C^{2,\epsilon}_{\ast}(Q_+)} \lesssim \|\D_G v\|_{Y_{\alpha,\epsilon}},
\end{align*}
which together with the local estimate in $X_{\alpha,\epsilon}(\mathcal{B}_3^+)$ provides the full bound from Proposition \ref{prop:invert}.\\

\emph{Step 1. Polynomial approximation at $P=\{y_n=y_{n+1}=0\}$.} We note that for $f\in Y_{\alpha,\epsilon}$ and $y_0\in P$ , there exists a first order polynomial $p_{y_0}(y)$ which is of the form $p_{y_0}(y) =f_0(y_0)y_n$ such that
\begin{align*}
\frac{1}{|\mathcal{B}_s^+(y_0)|}\int_{\mathcal{B}_s^+(y_0)}|f(y)-f_0(y_0)y_n|^2\leq C s^{2(1+2\alpha)}, \quad \forall s\in(0,1).
\end{align*}
By considering $v(y)-f_0(y_0)y_n^3/6$ and by still denoting the resulting function by $v$, we may assume that $f_0(y_0)=0$. The same arguments as before lead to the existence of a third order (in the homogeneous sense) polynomial $P_{y_0}$, where $$P_{y_0}(y)=a_0\left(y_n^3-3y_ny_{n+1}^2\right)+\sum_{i=1}^{n-1}b_iy_ny_i+c_0y_n,$$ for some constants $a_0,b_i,c_0$ depending on $y_0$, such that 
\begin{align}
\label{eq:approx_a}
\frac{1}{|\mathcal{B}_s^+(y_0)|}\int_{\mathcal{B}_s^+(y_0)}|v-P_{y_0}|^2\leq C\left(\|v\|_{L^2(\mathcal{B}_1^+)}^2+\|f\|_{Y_{\alpha,\epsilon}}^2\right)s^{2(3+2\alpha)}
\end{align}
for any $0<s<1/2$. \\

\emph{Step 2. Interpolation.} For $y\notin P$ with $\sqrt{y_n^2+y_{n+1}^2}=\lambda>0$, let
\begin{align*}
\tilde{v}_\lambda(\xi):=\frac{(v-P_{y_0})(y_0+\lambda ^2 \xi'',\lambda \xi_n, \lambda \xi_{n+1})}{\lambda ^{3+2\alpha}},
\end{align*}
where $y_0$ is the projection of $y$ on $P$. Let $\xi_0$ be the image point of $y$ under this rescaling. By Step 1, $\tilde{v}_\lambda(\xi)\in L^2(\mathcal{B}_{1/2}(\xi_0))$ with 
\begin{align}
\label{eq:rhs_est}
\|\tilde{v}_\lambda\|_{L^2(\mathcal{B}_1(\xi_0))}\leq C\left(\|v\|_{L^2}+\|f\|_{Y_{\alpha,\epsilon}}\right). 
\end{align}
Moreover, 
\begin{align*}
\Delta_G \tilde{v}_\lambda(\xi)= f_{\lambda}(\xi),
\end{align*}
where $f_{\lambda}(\xi):=\frac{1}{\lambda^{\epsilon}}f(y_0+\lambda^2\xi'',\lambda\xi_n,\lambda \xi_{n+1})$. We note that by the definition of $Y_{\alpha,\epsilon}$ and by $f_0(y''_0)=0$,
\begin{align*}
\| f_{\lambda}\|_{C^{0,\epsilon}(\mathcal{B}_{1/2}(\xi_0))} \leq \| f\|_{Y_{\alpha,\epsilon}}.
\end{align*}
In $\mathcal{B}_{1/2}(\xi_0)$, $\Delta_G$ is uniformly elliptic. Thus, by the classical $C^{2,\epsilon}$ Schauder estimates 
\begin{align*}
\|\tilde{v}_\lambda\|_{C^{2,\epsilon}(\mathcal{B}_{1/4}(\xi_0))} \leq C\left( \|\tilde{v}_\lambda\|_{L^2(\mathcal{B}_{1/2}(\xi_0))}+\|f_\lambda\|_{C^{0,\epsilon}(\mathcal{B}_{1/2}(\xi_0))}\right).
\end{align*}
Rescaling back and letting $\tilde{v}:=v-P_{y_0}$, we in particular infer 
\begin{align*}
\lambda^{-1-2\alpha+\epsilon}\sum\limits_{i,j=1}^{n+1}[Y_i Y_j \tilde{v}]_{C^{0,\epsilon}(\mathcal{B}_{\lambda/4}(y))} 
\leq C\left(\|v\|_{L^2(\mathcal{B}_1^+)}+\|f\|_{Y_{\alpha,\epsilon}}\right).
\end{align*}
Here we used (\ref{eq:rhs_est}) to estimate the right hand side contribution. 
Recalling the $L^{\infty}$ estimate $\| v\|_{L^{\infty}(Q_+)} \leq C \| \D_G v \|_{L^{\infty}}$ (c.f. the kernel bounds in Lemma \ref{lem:ker} in Section \ref{sec:kernel}) and the support conditions for $\D_G v$ and for $f$ allows us to further bound 
\begin{align*}
\|v\|_{L^2(\mathcal{B}_1^+)}+\|f\|_{Y_{\alpha,\epsilon}} \leq C \| f\|_{Y_{\alpha,\epsilon}}.
\end{align*}
This implies
\begin{align}\label{eq:err_est}
\lambda^{-1-2\alpha+\epsilon}\sum\limits_{i,j=1}^{n+1}[Y_i Y_j \tilde{v}]_{C^{0,\epsilon}_\ast(\mathcal{B}_{\lambda/4}(y))} 
+ \lambda^{-1-2\alpha}\sum\limits_{i,j=1}^{n+1}\|Y_i Y_j \tilde{v}\|_{L^{\infty}(\mathcal{B}_{\lambda/4}(y))} 
\leq C\|f\|_{Y_{\alpha,\epsilon}}.
\end{align} 
Passing through a chain of non-tangential balls, we infer that \eqref{eq:err_est} holds in a non-tangential cone at $y_0$:
\begin{align*}
&\sum\limits_{i,j=1}^{n+1}[d_G(y,y_0)^{-1-2\alpha+\epsilon}Y_i Y_j \tilde{v}]_{C^{0,\epsilon}_\ast(\mathcal{N}_G(y_0))} \\
&+ \sum\limits_{i,j=1}^{n+1}\|d_G(y,y_0)^{-1-2\alpha}Y_i Y_j \tilde{v}\|_{L^{\infty}(\mathcal{N}_G(y_0))} 
\leq C\|f\|_{Y_{\alpha,\epsilon}}.
\end{align*} 
We note that it is possible to derive \eqref{eq:err_est} for $v-P_{\bar y}$ at each $\bar y\in P$ and that hence $v$ is $C^{3,\alpha}_{\ast}(P)$ in the sense of Definition \ref{defi:diff}. As by Proposition~\ref{prop:decompI} the map $\bar y\mapsto P_{\bar y}$ is  $C^{0,\alpha}(P)$ regular, a triangle inequality and a covering argument yield the estimate in the full neighborhood of $y_0$
\begin{align*}
&\sum\limits_{i,j=1}^{n+1}[d_G(y,y_0)^{-1-2\alpha+\epsilon}Y_i Y_j \tilde{v}]_{C^{0,\epsilon}_\ast(\mathcal{B}_1^+(y_0))} \\
&+ \sum\limits_{i,j=1}^{n+1}\|d_G(y,y_0)^{-1-2\alpha}Y_i Y_j \tilde{v}\|_{L^{\infty}(\mathcal{B}_1^+(y_0))} 
\leq C\|f\|_{Y_{\alpha,\epsilon}}.
\end{align*}
This concludes the local estimate and hence concludes the proof of Proposition \ref{prop:invert}.
\end{proof}

\subsection{Invertibility of the Baouendi-Grushin Laplacian in $X_{\alpha,\epsilon} $, $Y_{\alpha,\epsilon}$}
\label{sec:XY}

We provide the proofs of the completeness of the spaces $X_{\alpha,\epsilon},Y_{\alpha,\epsilon}$ (c.f. Definition~\ref{defi:spaces}) and the desired invertibility of the Baouendi-Grushin Laplacian as an operator from $X_{\alpha,\epsilon}$ to $Y_{\alpha,\epsilon}$ (c.f. Lemma \ref{lem:inverse}).

\begin{lem}
\label{lem:Banach}
Let $X_{\alpha,\epsilon},Y_{\alpha,\epsilon}$ be as in Definition~\ref{defi:spaces}. Then $(X_{\alpha,\epsilon},\| \cdot \|_{X_{\alpha,\epsilon}}), (Y_{\alpha,\epsilon}, \| \cdot \|_{Y_{\alpha,\epsilon}})$ are Banach spaces.
\end{lem}

\begin{proof}
(i) We first note that by the definition of $Y_{\alpha,\epsilon}$, $\supp(f)\subset \mathcal{B}_3^+$ . Hence, it suffices to consider the behavior of functions on $\bar{\mathcal{B}_3^+}$.  By Proposition~\ref{prop:decompI} and Remark~\ref{rmk:characterize} a function $f\in Y_{\alpha,\epsilon}$ can be decomposed as 
\begin{align*}
f(y)=f_0(y'')y_n+ r(y)^{1+2\alpha-\epsilon}f_1(y),\quad r(y)=\sqrt{y_n^2+y_{n+1}^2},
\end{align*}
with $f_0\in C^{0,\alpha}(P)$ and $f_1\in C^{0,\epsilon}_\ast(Q_+)$ and $f_0$, $f_1$ are obtained by Taylor approximation of $f$ (c.f. the proof of Proposition \ref{prop:decompI} in Section \ref{sec:decomp}).
Moreover, $[f_0]_{\dot{C}^{0,\alpha}(P\cap \mathcal{B}_3)} +[f_1]_{\dot{C}^{0,\epsilon}_\ast(\mathcal{B}_3^+)}$ is equivalent to $\|f\|_{Y_{\alpha,\epsilon}}$. Thus, in order to obtain the desired Banach property, it suffices to show the equivalence of the homogeneous Hölder norms and their inhomogeneous counterparts for  $y\in \mathcal{B}_3^+$. \\

We start by making the following observation: For any $f\in Y_{\alpha,\epsilon}$,  $\supp(f)\subset \mathcal{B}_3^+$ (in combination with the definition of $f_0, f_1$) implies that
\begin{align*}
f_0(y'')=0  \mbox{ and } f_1(y)=0 \mbox{ for } y=(y'',y_n, y_{n+1}) \mbox{ such that } (y'',0,0)  \in P\setminus \mathcal{B}_3^+.
\end{align*}
Thus,
\begin{align*}
\|f_0 \|_{L^\infty(P\cap \mathcal{B}_3)} \leq C [f_0]_{\dot{C}^{0,\alpha}(P \cap \mathcal{B}_3)}, \quad \|f_1\|_{L^\infty(\mathcal{B}_{3}^+)}\leq C[f_1]_{\dot{C}^{0,\epsilon}_\ast(\mathcal{B}_3^+)}.
\end{align*}
In particular, this immediately entails that $\|f_0\|_{C^{0,\alpha}(P\cap \mathcal{B}_3)}\leq C[f_0]_{\dot{C}^{0,\alpha}(P\cap \mathcal{B})}$ and $\|f_1\|_{C^{0,\epsilon}_\ast(\mathcal{B}_3^+)}\leq C[f_1]_{\dot{C}^{0,\epsilon}_\ast(\mathcal{B}_3^+)}$. Therefore, $Y_{\alpha,\epsilon}$ is a Banach space. \\

(ii) Let $v\in X_{\alpha,\epsilon}$. Since $\supp(\Delta_G v)\subset \mathcal{B}_3^+$, we infer that
\begin{align}\label{eq:compact_supp2}
\|\Delta_G v\|_{C^{0,\epsilon}_\ast(Q_+)}\leq C [\Delta_G v]_{\dot C^{0,\epsilon}_\ast(Q_+)}\leq C\|v\|_{X_{\alpha,\epsilon}}.
\end{align}
Moreover, the Dirichlet-Neumann boundary conditions allow us to extend $v$ and $\Delta_Gv$ evenly about $y_{n+1}$ and oddly about $y_n$. After the extension, the assumption that $v\in C_0(Q_+)$ yields the representation 
$$v(x)=\int\limits_{\R^{n+1}} K(x,y) \Delta_G v(y)dy,$$ 
where $K$ is the fundamental solution of $\Delta_G$ in $\R^{n+1}$ (c.f. Lemma \ref{lem:ker} in Section \ref{sec:kernel}). We remark that a priori $v$ deviates from $\int K(x,y) \Delta_G v(y)dy$ by (at most) a third order polynomial as we only control the semi-norm $[v]_{C^{2,\epsilon}_{\ast}(Q_+\setminus \mathcal{B}_1^+)}$ in the bulk and the deviation of $Y_{i}Y_{j}v$ at the boundary of $\mathcal{B}_3^+$.
However, the decay property at infinity forces $v$ to coincide with $\int K(x,y) \Delta_Gv(y) dy$. By the kernel estimates for the fundamental solution (c.f. Lemma \ref{lem:ker} in Section \ref{sec:kernel}) and by the support assumption (\ref{eq:compact_supp2})
\begin{align*}
\| v \|_{L^{\infty}(Q_+)} \leq C \| \D_G v\|_{L^\infty(Q_+)} \leq C \|v\|_{X_{\alpha,\epsilon}}.
\end{align*}
Thus, we are able to control the $L^\infty $ norm of the coefficients of the approximating  polynomial $P_{\bar y}$ at each point $\bar y\in P$.
\end{proof}

Last but not least, we show the invertibility of the Baouendi-Grushin Laplacian as an operator on these spaces:

\begin{lem}
\label{lem:inverse}
Let $X_{\alpha,\epsilon},Y_{\alpha,\epsilon}$ be as in Definition~\ref{defi:spaces}. Then, $\D_G:X_{\alpha,\epsilon} \rightarrow Y_{\alpha,\epsilon}$ is an invertible operator.
\end{lem}

\begin{proof}
We show that for each $f\in Y_{\alpha,\epsilon}$, there exists a unique $u\in X_{\alpha,\epsilon}$ such that $\Delta_Gu=f$. Moreover, by Section \ref{sec:quarter_Hoelder}
\begin{align}\label{eq:apriori}
\|u\|_{X_{\alpha,\epsilon}}\leq C\|f\|_{Y_{\alpha,\epsilon}}.
\end{align}

Indeed, given $f\in Y_{\alpha,\epsilon}$, we extend $f$ oddly about $y_{n}$ and evenly about $y_{n+1}$ and (with slight 
abuse of notation) still denote the extended function by $f$. Let $u(x)= \int K(x,y) f(y)dy$, where $K$ is the kernel from Section \ref{sec:kernel}. In particular, the decay estimates for $K$ (c.f. Lemma \ref{lem:ker} in Section \ref{sec:kernel}) imply that $u\in C_0(\R^{n+1})$. Since $f\in L^\infty(\R^{n+1})$ and $\supp(f)\subset \mathcal{B}_3$, we obtain that $u\in M^{2,p}(\R^{n+1})$ for any $1<p<\infty$ (c.f. the Calderon-Zygmund estimates in Section \ref{sec:kernel}). Moreover, by the symmetry of the extension, $u$ is odd in 
$y_n$ and even in $y_{n+1}$, which implies that $u=0$ on $\{y_n=0\}$ and $\p_{n+1}u=0$ on
$\{y_{n+1}=0\}$. We restrict $u$ to $Q_+$ and still denote it by $u$. By the interior estimates from 
the previous Section \ref{sec:quarter_Hoelder} and a scaling argument, we further obtain that 
$u\in X_{\alpha,\epsilon}$ and that it satisfies \eqref{eq:apriori}.

It is immediate that $\supp(\Delta_Gu)=\supp(f)\subset \mathcal{B}_3^+$. Moreover, by using the equation, $\p_{nn}u=f-(y_n^2+y_{n+1}^2)\sum_{i=1}^{n-1}\p_{ii}u-\p_{n+1,n+1}u=0 $ on $\{y_{n}=y_{n+1}=0\}$. This shows 
the existence of $u\in X_{\alpha,\epsilon}$ which satisfies $\Delta_Gu=f$. Due to \eqref{eq:apriori} such a function $u$ is unique in $X_{\alpha,\epsilon}$.
\end{proof}

\subsection{Kernel estimates for the Baouendi-Grushin Laplacian}
\label{sec:kernel}

Last but not least, we provide the arguments for the mapping properties of the Baouendi-Grushin Laplacian in the whole space setting. This in particular yields the kernel bounds, which are used in the previous subsection.\\

Our main result in this section are the following Calderon-Zygmund estimates:

\begin{prop}[Calderon-Zygmund estimates] 
\label{prop:CZ}
Let $Y:=(Y_1,\dots,Y_{n+1})$ with $Y_i$ denoting the vector fields from Definition \ref{defi:Grushinvf} and let $F=(F^1,\dots,F^{n+1})\in L^p(\R^{n+1},\R^{n+1})$, $f\in L^p(\R^{n+1})$.
Suppose that 
\[   \Delta_G u =  Y_i F^{i}. \]
Then, there exists a constant $c_n= c(n)>0$ such that 
\[ \sum\limits_{i=1}^{n+1} \Vert Y_i u \Vert_{L^p(\R^{n+1})} \le c_{n} \frac{p^2}{p-1} \Vert F \Vert_{L^p(\R^{n+1})}. \]
If 
\[ \Delta_ G u = f ,\] 
then there exists a constant $c_n= c(n)>0$ such that 
\[ \sum\limits_{i,j=1}^{n+1}  \Vert Y_{i}Y_j u \Vert_{L^p(\R^{n+1})} \le c_{n} \frac{p^2}{p-1} \Vert f \Vert_{L^p(\R^{n+1})}. \]
If $0<s<1$ and $F\in \dot C^s(\R^{n+1}, \R^{n+1})$, then there exists a constant $c_n= c(n)>0$ such that 
\[ \sum\limits_{i=1}^{n+1}  \Vert Y_i u \Vert_{\dot C^s(\R^{n+1})} \le c_{n} \frac1{s(1-s)}  \Vert F \Vert_{\dot C^s(\R^{n+1})}. \] 
Moreover, if $F$ is supported on ball of radius one,  
\[ \sum\limits_{i=1}^{n+1}  \Vert Y_i u (x)\Vert_{L^\infty(\R^{n+1})} 
\le c \Vert F \Vert_{C^{s}(B_1)}. \] 
\end{prop}

The key auxiliary result to infer the regularity estimates of Proposition \ref{prop:CZ} is the following existence and regularity result for a kernel to our problem:

\begin{lem} 
\label{lem:ker}
Let $u:\R^{n+1}\rightarrow \R$ be a solution of $\D_G u = f$. Then there exists a kernel $k(z,w): \R^{(n+1)\times (n+1)}\rightarrow \R$ such that 
\begin{align*}
u(x)= \int\limits_{\R^{n+1}} k(x,y)f(y)dy.
\end{align*}
Let $\tilde{Y}^{\alpha}$ denote the composition of the vector fields $\tilde{Y}_{\alpha_{1}}\dots \tilde{Y}_{\alpha_{|\alpha|}}$ where $\tilde{Y}_i$, $i\in\{1,\dots,2n\}$, denote the modified vector fields from Definition \ref{defi:Hoelder1}. Then for all multi-indeces $\alpha, \beta$ the following estimates hold
\[ \left|\tilde{Y}_{z}^\alpha \tilde{Y}_{w}^\beta k(z,w)\right|
\le c_{\alpha, \beta} d_G(z,w)^{2-|\alpha|-|\beta|} (\vol(B_{d_G(z,w)}(z)))^{-1}  .
\] 
Here the subscript $z,w$ in the vector fields $\tilde{Y}_z^{\alpha}, \tilde{Y}_w^{\beta}$ indicates the variable the vector fields are acting on.
\end{lem} 
 
Relying on this representation, we can proceed to the proof of Proposition \ref{prop:CZ}:

\begin{proof}[Proof of Proposition \ref{prop:CZ}]
Let $K_{ij}(z,w)= Y_{i,z} Y_{j,w} k(z,w)$ for any pair $i,j\in\{1,\dots,n+1\}$, where the indeces $z,w$ refer to the variables which the vector field are acting on and $k(z,w)$ denotes the kernel from Lemma \ref{lem:ker}. The function $K_{ij}(z,w)$ is related to the 
obvious Calderon-Zygmund operator $T$ which maps $L^p$ to $L^p$. This proves the desired $L^p$ bounds. Hence, it remains to prove the Hölder estimates. Formally, it maps constants to zero. Thus, 
\[ Tf(x) = T(f-f(x))(x)= \int\limits_{\R^{n+1}} K_{ij}(x,y)(f(y)-f(x)) dy .\] 
Now let $d_G(z,w) = 3$. We choose a smooth cutoff function $\phi$ which is equal to $1$ in $\mathcal{B}_1(0)$ and equal to $0$ outside $\mathcal{B}_{3/2}(0)$ and set
$f(x) = \phi(x-z) f(x) + \phi(x-w)f(x) + (1-\phi(x-z)-\phi(x-w)) f(x). $ 
We claim that $|T(f)(w)-T(f)(z)| \le c $. This follows from the kernel estimates of Lemma \ref{lem:ker}.
\end{proof} 

Finally, to conclude our discussion of the mapping properties of the Baouendi-Grushin operator, we present the proof of Lemma \ref{lem:ker}:

\begin{proof}[Proof of Lemma \ref{lem:ker}]
We begin by considering the equation
\[ \Delta_G  u = \sum\limits_{i=1}^{n+1}Y_i F \mbox{ in } \R^{n+1}. \]
For this we have the energy estimate
\[ \sum\limits_{i=1}^{n+1} \Vert Y_i u \Vert_{L^2(\R^{n+1})}^2 \le \sum\limits_{i=1}^{n+1} \Vert F^i \Vert_{L^2(\R^{n+1})}^2 . \] 
Also, the Sobolev embedding 
\[ \Vert u \Vert_{L^p(\R^{n+1})} \le c \sum\limits_{i=1}^{n+1} \Vert Y_i u \Vert_{L^q(\R^{n+1})} , \]
holds with 
\[ \frac1q - \frac1{2n-2} = \frac1p. \]
By duality, if 
\[ \frac12 - \frac1{2n-2} = \frac1p, \]
 we have the embedding 
\[  \Vert f \Vert_{\dot M^{-1}(\R^{n+1})} \leq c \Vert f \Vert_{L^{p'}(\R^{n+1})}  .\]
Here $\dot{M}^{-1}(\R^{n+1})$ denotes the dual space of $\dot{M}^1(\R^{n+1})$ (and $\dot{M}^{1}(\R^{n+1})$ is the homogeneous version of the space introduced in Definition \ref{defi:GrushinLp} in Section \ref{sec:quarter_Hoelder}).
As discussed in Section \ref{sec:holder} the symbol of $\D_G$ defines the sub-Riemannian metric
\[   g_{y}(v,w) = (y_n^2 + y_{n+1}^2)^{-1} \sum\limits_{i=1}^{n-1} v_{i}w_{i} + v_{n}w_{n} + v_{{n+1}}w_{{n+1}}, \] 
which itself correspondingly defines a metric $d_G$ on $\R^{n+1}$. 
The operator $\D_G$ satisfies the Hörmander condition with the vector fields
\[ \tilde{Y}_i, \ i\in\{1,\dots,2n\}. \]
Hence, it is hypoelliptic and any local distributional solution is smooth. 
More precisely, if $u \in L^1(B_1(y))$ satisfies $\Delta_G u =0$, then for any multi-index $\alpha$
\[ \Vert \tilde{Y}^{\alpha} u \Vert_{L^\infty(B_{1/2}(y))} 
\le c_{\alpha} \Vert u \Vert_{L^1(B_1(y))}. \] 
Let $\frac1p+\frac1{2n-2}= \frac12$ and let $p'$ the Hölder conjugate
exponent. Then, by the embeddings, if $f \in L^{p'}$ and 
\[ \Delta_G u = f \mbox{ in } \R^{n+1}, \] 
then $\Vert u \Vert_{L^p(\R^{n+1})} \le c \Vert f \Vert_{L^{p'}(\R^{n+1})}$.

 By the Schwartz kernel theorem there is a kernel $k(z,w)$ so that 
\[  u(z) = \int\limits_{\R^{n+1}} k(z,w) f(w) dw. \]
More precisely, if $f \in \dot M^{-1}$, then $u \in \dot M^{-1}$. 
In particular, if $f$ is supported in ball  $B_1(w)$ then $u$ is a solution to 
the homogeneous problem outside. In particular if $ d_G(z,w) \ge 3$ then $u$ is bounded together with all derivatives in $B_1(z)$. We fix $z$. Then, 
\[    M^{-1}(B_1(w)) \ni f \to u(z), \] 
is a linear continuous map, which is represented by $\tilde w \to k(z,\tilde w) \in M^1(B_1(w))$. Since $\Delta_G$ is self-adjoint,$k(z,w) = k(w,z)$. Repeating previous arguments we see that 
\[ \tilde{Y}^\alpha_w k(z,\tilde w)   \] 
is bounded in $B_{1/2}(w)$. Repeating the arguments and dualizing once more we obtain that 
\[ |\tilde{Y}_z^\alpha  \tilde{Y}^\beta_w k(z,w)| \le c, \] 
provided $d_G(z,w)=1$. Hence rescaling leads to the desired kernel estimates.
\end{proof}

\bibliography{citations}
\bibliographystyle{alpha}

\end{document}